\documentclass[11pt,letterpaper]{amsart}
        \usepackage[utf8]{inputenc}
        \usepackage[english]{babel}
        \usepackage{amsmath}
        \usepackage{amssymb}
        \usepackage{amsfonts}
        \usepackage{amsthm}
        \usepackage{mathrsfs}
        \usepackage{graphicx}
        \usepackage{enumitem}
        \setlist[itemize]{leftmargin=*}
        \usepackage{xcolor}
        \usepackage{url}
        \usepackage{bm}
        \usepackage[all]{xy}
        \usepackage{xcolor}
        \usepackage{hyperref}
        \usepackage{esint}
        \usepackage{todonotes}
        \usepackage{xcolor}
        \usepackage{soul}
        \usepackage{stmaryrd}
        \hypersetup{
        	colorlinks,
        	linkcolor={blue!60!black},
        	citecolor={green!60!black},
        	urlcolor={green!60!black}
        }
        \usepackage{doi}
        \usepackage[capitalise,nameinlink]{cleveref}
        \usepackage{microtype}
        \usepackage[doi=false,isbn=false,url=false,eprint=false,backend=bibtex,giveninits=true,sorting=nyt]{biblatex}
\DeclareFieldFormat[article,book,misc,inbook]{title}{\textbf{#1}}
\DeclareFieldFormat[inbook]{chapter}{#1}
\renewbibmacro{in:}{}
\DeclareFieldFormat[article]{volume}{vol\adddot\addnbspace #1}
\DeclareFieldFormat[article]{number}{no\adddot\addnbspace #1}
\renewbibmacro{volume+number+eid}{%
	\setunit{\addcomma\space}%
	\printfield{volume}%
	\setunit{\addcomma\space}%
	\printfield{number}%
	\setunit{\addcomma\space}%
	\printfield{eid}%
}
\AtEveryBibitem{
  \clearlist{language}
}
        
        \crefformat{equation}{#2(#1)#3}
        \crefrangeformat{equation}{#3(#1)#4 to~#5(#2)#6}
        \crefmultiformat{equation}{#2(#1)#3}{ and~#2(#1)#3}{, #2(#1)#3}{ and~#2(#1)#3}
        \usepackage{mathtools}
        \usepackage{etoolbox}
        \usepackage{chngcntr}

        \newcommand{\N}{\mathbb{N}}
        \newcommand{\Z}{\mathbb{Z}}
        
        \newcommand{\R}{\mathbb{R}}
        \newcommand{\C}{\mathbb{C}}
        \newcommand{\E}{\bm{\mathrm{E}}}
        \newcommand{\G}{\mathcal{G}}

        \newcommand{\de}{\partial}
        \newcommand{\mz}{\frac{1}{2}}
        \newcommand{\uno}{\bm{1}}
        
        \newcommand{\weakto}{\rightharpoonup}
        \newcommand{\weakstarto}{\stackrel{*}{\rightharpoonup}}

        \renewcommand{\bar}{\overline}
        
        \DeclareMathOperator{\dist}{dist}
        \newcommand{\hau}{\mathcal{H}}
        \DeclareMathOperator{\spt}{spt}
        
        \newcommand{\vol}{\mathrm{vol}}
        \newcommand{\ang}[1]{\langle #1\rangle}

        \renewcommand{\epsilon}{\varepsilon}
        \DeclareMathOperator{\Gr}{Gr}

        \theoremstyle{definition}
        \newtheorem{definition}{Definition}
        \newtheorem{rmk}[definition]{Remark}
        \newtheorem*{definition*}{Definition}
        \newtheorem*{rmk*}{Remark}
        \newtheorem*{ack*}{Acknowledgement}
        \newtheorem*{acks*}{Acknowledgements}
        \theoremstyle{plain}
        \newtheorem{thm}[definition]{Theorem}
        \newtheorem{lemma}[definition]{Lemma}
        \newtheorem{corollary}[definition]{Corollary}
        \newtheorem{proposition}[definition]{Proposition}
        
        \newtheorem{conj}[definition]{Conjecture}
        
        \newtheorem*{thm*}{Theorem}
        
        \newtheorem*{lemma*}{Lemma}
        \newtheorem*{corollary*}{Corollary}
        \newtheorem*{proposition*}{Proposition}
        \newtheorem*{claim*}{Claim}
        \newtheorem*{conj*}{Conjecture}
        
        \numberwithin{equation}{section}
        \numberwithin{definition}{section}
        \usepackage[top=3.5cm,bottom=3.5cm,left=3cm,right=3cm]{geometry}
        \renewcommand{\ln}{\log}
        \renewcommand{\backslash}{\setminus}
        \renewcommand{\textbf}[1]{\bm{\mathrm{#1}}}
        \renewcommand{\div}{\operatorname{div}}
        \renewcommand{\deg}{\operatorname{deg}}
\makeatletter
\renewcommand{\tocsection}[3]{%
  \indentlabel{\@ifnotempty{#2}{\bfseries\ignorespaces#1 #2\quad}}\bfseries#3}
\renewcommand{\tocsubsection}[3]{%
  \indentlabel{\@ifnotempty{#2}{\ignorespaces#1 #2\quad}}#3}

\newcommand\@dotsep{4.5}
\def\@tocline#1#2#3#4#5#6#7{\relax
  \ifnum #1>\c@tocdepth 
  \else
    \par \addpenalty\@secpenalty\addvspace{#2}%
    \begingroup \hyphenpenalty\@M
    \@ifempty{#4}{%
      \@tempdima\csname r@tocindent\number#1\endcsname\relax
    }{%
      \@tempdima#4\relax
    }%
    \parindent\z@ \leftskip#3\relax \advance\leftskip\@tempdima\relax
    \rightskip\@pnumwidth plus1em \parfillskip-\@pnumwidth
    #5\leavevmode\hskip-\@tempdima{#6}\nobreak
    \leaders\hbox{$\m@th\mkern \@dotsep mu\hbox{.}\mkern \@dotsep mu$}\hfill
    \nobreak
    \hbox to\@pnumwidth{\@tocpagenum{\ifnum#1=1\bfseries\fi#7}}\par
    \nobreak
    \endgroup
  \fi}
\AtBeginDocument{%
\expandafter\renewcommand\csname r@tocindent0\endcsname{0pt}
}
\def\l@subsection{\@tocline{2}{0pt}{2.5pc}{5pc}{}}
\makeatother        
        
\begin{document}
        \title[Decay of excess for the abelian Higgs model]{Decay of excess for the abelian Higgs model}
        
        \author[G.\ De Philippis]{Guido De Philippis}
        \address{Courant Institute of Mathematical Sciences,
        New York University.
        251 Mercer Street,
        New York, NY 10012-1185}
        \email{guido@cims.nyu.edu}
        
        \author[A.\ Halavati]{Aria Halavati}
        \address{Courant Institute of Mathematical Sciences,
        New York University.
        251 Mercer Street,
        New York, NY 10012-1185}
        \email{aria.halavati@cims.nyu.edu}
        
        \author[A.\ Pigati]{Alessandro Pigati}
        \address{Bocconi University, Department of Decision Sciences. Via Guglielmo Roentgen 1, 20136,
        Milano, Italy}
        \email{alessandro.pigati@unibocconi.it}
        	\begin{abstract}
                In this article we prove that entire critical points $(u,\nabla)$ of the self-dual $U(1)$-Yang--Mills--Higgs functional $E_1$, with energy
                $$E_1(u,\nabla;B_R):=\int_{B_R}\left[|\nabla u|^2+\frac{(1-|u|^2)^2}{4}+|F_\nabla|^2\right]\leq(2\pi+\tau(n)) \omega_{n-2}R^{n-2}$$
                for all $R>0$, have unique blow-down.
                Moreover, we show that they are two-dimensional in ambient dimension $2\leq n\leq4$, or in any dimension $n\ge2$ assuming that $(u,\nabla)$ is a local minimizer, thus establishing a co-dimension-two analogue of Savin's theorem. The main ingredient is an Allard-type improvement of flatness.
        	\end{abstract}

            \maketitle
            \tableofcontents
            \tableofcontents
            \frenchspacing
            \section{Introduction}
            \subsection{Background on the Allen--Cahn and abelian Higgs models}
           Area  of geometric shapes is one of  the oldest geometric functional considered in mathematics.
            Given an ambient Riemannian manifold
            $(M^n,g)$ (possibly the flat Euclidean space $\R^n$) and given an integer $1\le k\le n-1$,
            one  looks for $k$-dimensional objects, such as $k$-dimensional submanifolds or singular versions of them,
            which are \emph{critical points} for the $k$-area $\mathcal{H}^k$.
            These are called \emph{minimal submanifolds} (provided they are regular enough, depending on the context).

            Besides its intrinsic interest, the study of minimal submanifolds
            in a given ambient often reveals global topological structure, especially when coupled with curvature information.
           
            These applications motivate a systematic existence and regularity theory of such critical points.
            In spite of its apparent simplicity, it is notoriously difficult to use the area functional directly in the context of the \emph{calculus of variations}, especially when $k\ge2$.
            Leaving out a number of very important ways to deal with this problem,
            such as the approach via parametrizations when $k=2$, see, e.g., \cite{Douglas,Rado,SU,Riviere} among others,
            area-minimizing \emph{currents} and \emph{sets of finite perimeter} in the context of minimization, \cite{DeGiorgi54, FedererFleming60}  and the monographs   \cite{Federer,Simon}, and the Almgren--Pitts theory involving \emph{varifolds}, \cite{Allard-1,Pitts}. In this paper we focus on the  approximation of minimal surfaces as limit of  \emph{diffuse}  physical energies.

            Starting from  the pioneering ideas of  De Giorgi, Modica \cite{Modica}, Ilmanen \cite{Ilmanen}, and
            Hutchinson--Tonegawa \cite{HutTon},
            it was understood that smooth critical points $u:M\to\R$ for
            the \emph{Allen--Cahn energy}
            $$E_\epsilon(u):=\int_M\left[\epsilon|du|^2+\frac{(1-u^2)^2}{4\epsilon}\right]$$
            are effective diffuse approximations of minimal hypersurfaces.
            The Allen--Cahn functional  is a well studied model for phase transitions; a typical critical point $u$ takes values in $[-1,1]$, with $u\approx\pm1$ (the pure phases) except in a transition region of thickness $\approx \epsilon$, where most of the energy concentrates. Roughly speaking, this region is an $\epsilon$-neighborhood of a minimal hypersurface, which acts as an interface between the two phases, and the energy density decays exponentially fast away from this interface.

            This understanding brought a novel, PDE-based
            way to attack variational problems for the co-dimension-one area \cite{Guaraco},
            which often allows to obtain more refined results compared to other methods \cite{CM}.

            In co-dimension two, similar attempts have been made by looking at the same energy for maps $u:M\to\C$, replacing $u$ with $|u|$ in the second term. This corresponds to a  simplified version of the Ginzburg--Landau model of superconductivity,
            popularized by Bethuel--Brezis--H\'elein \cite{BBH}, where one neglects  the magnetic field. The asymptotic analysis of this energy is substantially more involved, due to the lack of the aforementioned exponential decay, and brought mixed results:
            see, for instance, \cite{LR,BBO} in the positive direction
            and \cite{PigatiSternGL} in the negative one.

            On the other hand, including the magnetic field and looking at the so-called \emph{self-dual regime} (also called \emph{critical coupling}),
            we can consider the alternative energy
            $$E_\epsilon(u,\alpha):=\int_M\left[|du-i\alpha u|^2+\frac{(1-|u|^2)^2}{4\epsilon^2}+\epsilon^2|d\alpha|^2\right].$$
            Apart from the different normalization, it differs from the previous energies by an additional variable, the one-form $\alpha\in\Omega^1(M;\R)$, which twists the Dirichlet term and appears in the Yang--Mills term $|d\alpha|^2$ (indeed,
            the latter equals $|F_\nabla|^2$, where $F_\nabla$ is the curvature of the connection $\nabla:=d-i\alpha$ on the trivial complex line bundle $\C\times M$).

            This energy, in this specific self-dual regime (i.e., the choice of constants in front of each term), is well known in gauge theory, where it is often called \emph{$U(1)$-Yang--Mills--Higgs},
            or simply  \emph{abelian Higgs model}.
            It received a thorough treatment in dimension \(2\), with a complete classification of critical planar pairs $(u,\nabla)$ of finite energy by Taubes \cite{Taubes,Taubes-1}. See also \cite{HJS} for the case of  Riemann surface and \cite{Bradlow} for  K\"ahler manifolds.
            Recently, in \cite{Pigati-1}, Stern and the third-named author developed the asymptotic analysis in arbitrary Riemannian manifolds, obtaining the precise co-dimension-two analogue of  the result by Hutchinson--Tonegawa: see \cref{varifold-limit-theorem} below. Related facts, including $\Gamma$-convergence and the gradient flow convergence to mean curvature flow, have also been verified, by Parise, Stern, and the third-named author \cite{Pigati-2,Pigati-3}.

            Based on some new functional inequalities \cite{Halavati-inequality}, the second-named author recently obtained
            a quantitative refinement of the work of Taubes, who showed (among other facts)
            that critical pairs on the plane minimize the energy among pairs with the same degree at infinity:
            namely, in \cite{Halavati-stability} a quantitative \emph{stability} is proved; the precise statement is recalled in
            \cref{ymh-dim-2-thm}. Together with the main result from \cite{Pigati-1}, this result will be instrumental for the analysis in the present paper.
            
            \subsection{Savin's theorem}

            Since the work of De Giorgi \cite{Degriogi-1}
            and Allard \cite{Allard-1}, it is known that almost-flat
            minimal submanifolds enjoy an \emph{improvement of flatness},
            i.e., they become even closer to a plane at smaller scales, in a quantitative way. Iteration of this improvement of flatness is the key mechanism in proving (quantitative) regularity of minimal surfaces. The key analytical fact behind this decay property is the observation  that the linearization of the minimal graph equation is the Laplace equation, whose solutions enjoy similar decay properties.

            A related question, in the spirit of the classical Liouville theorem,  is  whether globally defined objects should be planar. The famous \emph{Bernstein's conjecture} predicts that this is always true for minimal graphs $\R^{n-1}\to\R$, which are automatically (locally) area-minimizing hypersurfaces.
            In view of the improvement of flatness, this question quickly reduces to understanding whether any blow-down is necessarily a hyperplane. Bernstein's question was answered affirmatively
            by the works of Fleming, De Giorgi, Almgren, and Simons for $n\le 8$, while Bombieri--De Giorgi--Giusti produced a counterexample for $n=9$, whose blow-down corresponds to the Simons cone, in \cite{BDDG}.

            By analogy, De Giorgi conjectured that
            critical points $u:\R^n\to\R$ of the Allen--Cahn energy
            with $\frac{\de u}{\de x_n}>0$ (so that level sets are graphs)
            are just rotations of a one-dimensional solution $u=u(x_n)$, at least when \(n \le 8\).
            The question has been solved  by  Ghoussoub--Gui for \(n=2\), in \cite{GG}, by  Ambrosio--Cabr\'e for \(n=3\), in \cite{AC}, and by Barlow--Bass--Gui  under additional regularity for the level sets, in \cite{BBG}. Finally, in \cite{Savin-1} Savin settled the conjecture  for all $n\le 8$ under the  assumption that $u(x',x_n)\to\pm1$ as $x_n\to\pm\infty$, for any fixed $x'\in\R^{n-1}$. In fact, his main contribution could be phrased as follows.

            \begin{thm}[Savin's theorem]\label{savin.cheat}
                A local minimizer $u$ for Allen--Cahn enjoys improvement of flatness. In particular, if any blow-down is a hyperplane,
                then the blow-down is unique.
            \end{thm}

            Here the blow-downs can be understood in terms of energy concentration, or by looking at the blow-downs of the zero set $\{u=0\}$ with respect to the (local) Hausdorff convergence of sets.

            The previous statement implies the resolution of De Giorgi's conjecture for $n\le 8$, with the extra assumption mentioned above.
            Indeed, it is known that this condition,  together with $\frac{\de u}{\de x_n}>0$,  implies that $u$ is a local minimizer; moreover, any blow-down gives a vertical area-minimizing cone in $\R^n$, hence an area-minimizing cone in $\R^{n-1}$, which is known to be necessarily a hyperplane for $n\le8$.

            In other words, uniqueness of the blow-down relies on two ingredients: improvement of flatness and a classification of blow-downs. While the second one can be directly exported from the setting of minimal hypersurfaces, the first one needs to be proved  \emph{before} passing to the limit $\epsilon\to0$,  and this is the difficult part settled by Savin.
            
            Finally, using the maximum principle (see, e.g., \cite{Farina,Berestycki}),  one can deduce  the following.

            \begin{corollary}
                Under the previous assumptions, $u$ is one-dimensional.
            \end{corollary}
                
            As for minimal graphs, De Giorgi's conjecture (even with the extra assumption used by Savin) is false for $n\ge9$:
            a counterexample has been constructed by Del Pino--Kowalczyk--Wei, in \cite{DPKW}.

            Savin's approach uses viscosity techniques, resembling the Krylov--Safanov theory in spirit. In particular,
            while his groundbreaking methods have a wide range of applicability, even beyond variational equations, it is not always clear how one can extend these techniques to the  vectorial setting, where the maximum principle does not apply; see however \cite{Savin-system, DGS-23}.
            
           Recently, Wang \cite{kelei-wang} obtained a variational proof of Savin's theorem, following   the strategy of  Allard's proof of excess decay for stationary varifolds. Wang's paper has been the starting point for our investigation of the regularity properties of the zero set of solutions of the Yang--Mills--Higgs equations.

            \subsection{Main results}
            We consider the energy
            $$E_\epsilon:=\int e_\epsilon(u,\nabla), \quad
            e_\epsilon(u,\nabla):=|\nabla u|^2 + \frac{(1-|u|^2)^2}{4\epsilon^2} + \epsilon^2|F_\nabla|^2. $$
            Note that $E_\epsilon$ is just a rescaling of $E_1$, for $\epsilon>0$.            
            The main result of the paper could be summarized as follows.

            \begin{thm}
                Savin's result, as stated in \cref{savin.cheat},
                holds for critical pairs $(u,\nabla)$ for $E_1$, in any dimension $n\ge2$.
            \end{thm}

            The following is the precise statement of the excess decay for critical points.

            \begin{thm}[Tilt-excess decay]\label{tilt-excess-decay-statement}
                For any $n \geq 3$ and small enough $0<\rho\leq\rho_0(n)$, there exist constants $\epsilon_0(n,\rho),\tau_0(n,\rho)$ such that the following holds. Let $(u,\nabla)$ be a critical point for $E_\epsilon$ on the unit ball $B_1^n\subset\R^n$,
                with $\epsilon \leq \epsilon_0$, $u(0)=0$, and the energy bound
                \begin{align}\label{energy-bound}
                    \frac{1}{|B^{n-2}_1|}\int_{B^n_1} e_\epsilon(u,\nabla) \leq 2\pi+\tau_0.
                \end{align}
                Then at least one of the following statements is true:
                {either}
                \begin{align*}\E_1(u,\nabla,B_\rho^n,\bar{S}) \leq C\rho^2\E_1(u,\nabla,B_1^n,S),
                \end{align*}
                for some $(n-2)$-plane $\bar{S}$ with $\|P_{\bar{S}} - P_{S}\| \leq C\sqrt{\E_1(u,\nabla,B_1^n,S)}$, where $P_S$ is the orthogonal projection onto $S$, the plane minimizing $\E(u,\nabla,B_\rho^n,\cdot)$, {or}
                \begin{align*}
                    \E_1(u,\nabla,B_1^n,S) \leq \max\{C\epsilon^2 |\log\E|^2\sqrt{\E},e^{-{K}/{\epsilon}}\},
                \end{align*}
                where $\E=\E(u,\nabla,B_1^n,S)$ and $C=C(n)$, $K=K(n)$  are independent of $\rho$.
            \end{thm}

            \begin{rmk}
                To be precise, we assume also the pointwise bounds \cref{u-le-1}--\cref{modica-bounds}, which are automatically true if $(u,\nabla)$ is a critical pair on $\R^n$ with energy growth $O(R^{n-2})$ on $B_R^n$.
            \end{rmk}

            Here $\E$ is the \emph{excess}, defined in \cref{excess-definition} below, which naturally splits into two parts, $\E_1$ and $\E_2$, measuring  how far a solution is from being two-dimensional and from solving the first order \emph{vortex equations}, respectively. We also note that $\E_1$ parallels the notion of excess in the theory of varifolds and does not depend on the orientation, while $\E$ sees the orientation and should be thought of as the stronger notion of excess in the setting of currents. While in principle the previous result establishes a quantitative decay only for $\E_1$, it is enough to obtain the following.

            \begin{corollary}\label{coroll.intro}
                If $(u,\nabla)$ is an entire critical point on $\R^n$, with
                $$0<\lim_{R\to\infty}\frac{1}{|B^{n-2}_R|}\int_{B^n_R} e_\epsilon(u,\nabla) \leq 2\pi+\tau_0(n),$$
                then this limit is $2\pi$ and the blow-down is a unique plane.
            \end{corollary}

            The previous limit always exists by the monotonicity formula for  $E_\epsilon$, see \cite{Pigati-1}. By a simple compactness argument and Allard's theorem, it is easy to see that the assumption guarantees that any blow-down is an $(n-2)$-dimensional plane. The key assertion is that, in view of improvement of flatness, the blow-down is \emph{unique}.

            Another simple consequence of the techniques is the following fact, a diffuse version of the $C^{1,\alpha}$ regularity of minimal graphs.

            \begin{thm}\label{Main-result-epsilon-neighbourhood-of-graph}
                Let $(u,\nabla)$ be a critical point for $E_\epsilon$ as above.
                Given $\alpha\in[0,1)$ and $\gamma>0$, if $\epsilon\le\epsilon_0(n,\alpha,\gamma)$ and $\tau_0\le\tau_0(n,\alpha,\gamma)$ then the vorticity set $\{|u|\le\frac34\}\cap B_{1/2}^n$ is contained in a $C(n,\alpha,\gamma)\epsilon^{1/(1+\alpha)}$-neighborhood of the graph of a function
                $$f:B_1^{n-2}\to\R^2$$
                with $\|f\|_{C^{1,\alpha}}\le \gamma$.
            \end{thm}

            Differently from the co-dimension one setting,
            where uniqueness of the blow-down (with multiplicity one)
            implies via the maximum principle that $u$ is one-dimensional, at the present time we are not able to conclude that, in the setting of \cref{coroll.intro},
            the solution $u$ is two-dimensional. Here we formulate the following variant of the \emph{Gibbons conjecture}.

            \begin{conj}
            An entire critical point $(u,\nabla)$ on $\R^n$ satisfying
                $$\lim_{R\to\infty}\frac{1}{|B^{n-2}_R|}\int_{B^n_R} e_\epsilon(u,\nabla) = 2\pi$$
                and, writing any $x\in\R^n$ as $x=(y,z)\in\R^2\times\R^{n-2}$, also
            $$\lim_{|y|\to\infty}|u(y,z)|=1,\quad\text{uniformly in $z$},$$
            is necessarily two-dimensional, i.e., it is the pullback through the projection $\R^n\to\R^{2}$ of the standard solution in $\R^2$ with degree $\pm1$, up to translation and change of gauge.
            \end{conj}

            It is interesting to note that, if we allow for multiplicity higher than one in the blow-down,
            this conjecture (with the appropriate energy assumption) does \emph{not} hold for the non-magnetic Ginzburg--Landau energy mentioned before, see \cite{DP}. It is not clear if such rigidity with higher multiplicity should be expected for the energy considered in the present work.
            On the other hand, our excess decay is strong enough to give an affirmative answer up to dimension $4$. With a more involved argument, we are able to settle it also for local minimizers in all dimensions $n\geq 2$, thus obtaining a full analogue of Savin's theorem. 
            
            \begin{thm}\label{main-result-global-n4}
                The previous conjecture holds for critical points in dimension $2\le n\le 4$, as well as for local minimizers in all dimensions $n\ge2$,
                even without the second assumption that $\lim_{|y|\to\infty}|u(y,z)|=1$ uniformly in $z$: the pair $(u,\nabla)$ is two-dimensional, up to rotation and change of gauge.
            \end{thm}

            The techniques used in this paper resemble those of \cite{kelei-wang} at several places.
            However, there are several  key  differences which require substantially new ideas. For instance, in order to construct the Lipschitz approximation, Wang uses a generic level set of $u$.
            The fact that typical level sets share effectively properties of minimal hypersurfaces is often used in \cite{kelei-wang}, as well as in \cite{Ilmanen,TW,CM} and many other works in the Allen--Cahn setting.
            For the abelian Higgs model, level sets of $u$ can be arbitrarily irregular, due to gauge invariance; while we can always pass to a local Coulomb gauge, we do not expect such effective properties of typical preimages of $u$.

            Rather, in the present setting, we rely on the results from \cite{Halavati-stability} in order to control in a fine way the behavior of $u$ on many (but not all) two-dimensional slices perpendicular to the reference plane. For instance, we are able to bound the distance of the actual zero set from a certain function giving the ``center of mass'' of each slice, which is used as a Lipschitz approximation and allows to derive a Caccioppoli-type inequality.
            
            In the case of minimizers, this refined control also allows us to deform a nearly flat minimizing pair $(u,\nabla)$ in the interior to gain a \emph{stronger} decay of the excess. This deformation process also requires a very involved gauge fixing argument, since generically $(u,\nabla)$ could be very irregular in an arbitrary gauge. The following theorem is the precise statement of the improved tilt-excess decay that we obtain for minimizers. Note that in the statement below,  \(\beta\) can indeed be any power. This is fundamental for proving \cref{main-result-global-n4} for minimizers, where we need to take \(\beta=n-2\).
        
            \begin{thm}\label{tilt-excess-decay-minimizer}
               For any $\beta>0$ and small enough $0<\rho\le\rho_0(n,\beta)$ there exist $\tau_0(n,\beta,\rho)>0$, $\epsilon_0(n,\beta,\rho)> 0$ with the following property. Let $(u,\nabla)$ be a local minimizer of $E_\epsilon$ in $B^n_1$ with $\epsilon \leq \epsilon_0$ and $u(0)=0$ such that
               \begin{align*}
                   \frac{1}{|B^{n-2}_1|}\int_{B^n_1}e_\epsilon(u,\nabla) \leq 2\pi + \tau_0,
               \end{align*}
               and let $S$ minimize $\E(u,\nabla, B_1^n,S)$.
               Then, after a suitable rotation, at least one of the following statements is true: either
               \begin{align*}
                   \E(u,\nabla,B_\rho^n,\bar S) \leq C\rho^2\E(u,\nabla,B_1^n,S),
               \end{align*}
               for some new oriented $(n-2)$-plane $\bar S$ with $\|P_{\bar S}-P_S\| \leq C\sqrt{\E}$, or
               \begin{align*}
                   \E(u,\nabla,B^n_1,\R^{n-2}) \leq \epsilon^\beta,
               \end{align*}
               where $C=C(n,\beta)$ is independent of $\rho$.
            \end{thm}
            Then, by taking $\beta\ge n-2$, we obtain a direct proof of \cref{main-result-global-n4} in the case of minimizers.

            \begin{ack*}
            This article is part of the PhD thesis of A.H.\ and he would like to thank Fanghua Lin for his inspiring lectures on harmonic maps. The authors would like to thank  Robert V.\ Kohn, Zhengjiang Lin, Sylvia Serfaty, and Daniel Stern for their interest and related discussions.
            The work of G.D.P.\ and A.H.\ has been supported by the NSF grant DMS-2055686 and the Simons Foundation.
            \end{ack*}

\tableofcontents
            
            \section{Basic definitions}
                While we work on the trivial Hermitian line bundle over the Euclidean space $\R^n$, it is worth to recall the definition of Hermitian line bundle over a general manifold.
                \begin{definition}
                    A \emph{Hermitian line bundle} over a smooth manifold $M$ is a complex line bundle $L\rightarrow M$ (i.e., a complex vector bundle with typical fiber $\C$) equipped with a \emph{Hermitian metric}, whose real part will be denoted by $\ang{\cdot,\cdot}$;
                    thus, for any two smooth sections $s,t\in\Gamma(L)$,
                    the function $p\mapsto\ang{s(p),t(p)}$ is smooth and real-valued, and satisfies $\ang{is(p),it(p)}=\ang{s(p),t(p)}=\ang{t(p),s(p)}$.
                    %
                \end{definition}

                \begin{definition}
                    A \emph{metric connection} is a map $\nabla$ which assigns to each vector field $\xi\in\Gamma(TM)$ an endomorphism $\nabla_\xi: \Gamma(L)\rightarrow \Gamma(L)$ with the following properties:
                    \begin{enumerate}[label=(\roman*)]
                        \item $\nabla_{\xi+\eta} s = \nabla_{\xi} s+ \nabla_\eta s$;
                        \item $\nabla_{\phi\xi} s = \phi \nabla_\xi s$;
                        \item $\nabla_{\xi}(\phi s) = (\xi\phi)s + \phi\nabla_{\xi} s$;
                        \item $\xi(\ang{s,t})=\ang{\nabla_\xi s,t}+\ang{s,\nabla_\xi t}$,
                    \end{enumerate}
                    for any sections $s,t\in\Gamma(L)$, vector fields $\xi,\eta \in \Gamma(TM)$, and function $\phi \in C^{\infty}(M)$.
                    %
                \end{definition}

                On the trivial bundle $L=\C\times M$,
                we can always write a metric connection $\nabla$ as
                $$\nabla=d-i\alpha,$$
                for a real-valued one-form,
                meaning that $\nabla_\xi s=ds(\xi)-i\alpha(\xi)s$.
                
                In general, for two vector fields $\xi$ and $\eta$, typically $\nabla_\xi$ and $\nabla_\eta$ do not commute, meaning that the connection has nontrivial \emph{curvature}. Formally, the curvature $F_\nabla$ is given by
                \begin{align}\label{curvature-definition}
                    F_\nabla(\xi,\eta) (s) = [\nabla_\xi, \nabla_\eta] s - \nabla_{[\xi,\eta]} s.
                \end{align}
                A simple computation shows that $F_\nabla$ is a two-form with values in the Lie algebra of $U(1)$, i.e., in imaginary numbers; we will sometimes use the real-valued two-form $\omega$ given by
                \begin{align}\label{omega}
                    F_\nabla(\xi,\eta) (s)=:-i\omega(\xi,\eta)s.
                \end{align}
                On the trivial bundle, if $\nabla=d-i\alpha$ then we simply have
                $$ \omega=d\alpha. $$
                
                We will use the inner product on two-forms induced by the following quadratic form:
                                \begin{align*}
                    |\omega|^2 = \sum_{1\leq j < k \leq n} |\omega(e_j,e_k)|^2,
                \end{align*}
                where $\{e_k\}_{k=1}^n$ is a local orthonormal frame for $TM$.
                
            \section{The $U(1)$-Yang--Mills--Higgs equations}
            For a section $u\in \Gamma(L)$ and a (metric) connection $\nabla$ on a Hermitian line bundle $L\rightarrow M$ over a smooth Riemannian manifold $(M,g)$, given a parameter $\epsilon>0$, we define the $U(1)$-Yang--Mills--Higgs energy as
            \begin{align}\label{energy-definition}
                E_\epsilon(u,\nabla) := \int_{M} \left[|\nabla u|^2 + \epsilon^2|F_\nabla|^2 + \frac{(1-|u|^2)^2}{4\epsilon^2}\right],
            \end{align}
            where $F_\nabla$ is the curvature of $\nabla$ and $|F_\nabla|$ is defined to be $|\omega|$ (with $\omega$ as in \cref{omega}). Equivalently, on the trivial bundle, for any section $u$ (viewed as a function $M\to\C$) and connection $\nabla=d-i\alpha$ we have
            \begin{align*}
                E_\epsilon(u,\nabla=d-i\alpha) = \int_{M} \left[|du - iu\alpha|^2 + \epsilon^2|d\alpha|^2 + \frac{(1-|u|^2)^2}{4\epsilon^2}\right].
            \end{align*}
            
            A smooth pair $(u,\nabla)$ gives a critical point for the Yang--Mills--Higgs energy if and only if it satisfies the system of partial differential equations:
            \begin{align}
                \nabla^*\nabla u &= \frac{1}{2\epsilon^2}(1-|u|^2)u,\label{ymh-pde-1}\\
                \epsilon^2d^*\omega &= \ang{\nabla u , iu},\label{ymh-pde-2}
            \end{align}
            where $\nabla^*$ is the adjoint of $\nabla$, while $d^*$ is the adjoint of $d:\Omega^1(M)\rightarrow \Omega^2(M)$, given by
            \begin{align*}
                (d^*\omega)(e_k) = -\sum_{j=1}^n (\nabla_{e_j}\omega)(e_j,e_k)
            \end{align*}
for some (and hence any) orthonormal frame \(\{e_j\}\).

            We now recall some Bochner-type identities from \cite[Sections 2--3]{Pigati-1}.
            Since $\omega$ is a closed two-form, after taking the exterior derivative in \cref{ymh-pde-2} we get
            \begin{align}\label{curvature-pde}
                \epsilon^2\Delta_H \omega  +|u|^2\omega = \psi(u),
            \end{align}
            where $\Delta_H = dd^* + d^*d$ is the Hodge Laplacian and
            \begin{align}\label{psi-definition}
                \psi(u) (e_j,e_k) := 2\ang{i\nabla_{e_j}u,\nabla_{e_k}u}.
            \end{align}
            One easily sees that the modulus $|u|^2$ satisfies the equation
            \begin{align}\label{modulus-pde}
                \Delta \frac{1}{2}|u|^2 = |\nabla u|^2 - \frac{|u|^2}{2\epsilon^2}(1-|u|^2).
            \end{align}
            We also recall the following Bochner identity for $|\nabla u|^2$:
            \begin{align}\label{Bochner-1}
                \Delta\frac12|\nabla u|^2 = |\nabla^2 u|^2 +\frac{1}{2\epsilon^2} (3|u|^2 - 1)|\nabla u|^2 - 2\ang{\omega,\psi(u)} + \mathcal{R}_1(\nabla u,\nabla u),
            \end{align}
            where $\mathcal{R}_1 = \operatorname{Ric}(e_j,e_k) \ang{\nabla_{e_j}u,\nabla_{e_k}u}$ and $\nabla^2_{e_j,e_k}u = \nabla_{e_j}(\nabla_{e_k}u)$.
            
            Next, we define the gauge-invariant Jacobian, which plays an important role
            in the $\Gamma$-convergence theory \cite{Pigati-2}, similar to the classical Jacobian in the $\Gamma$-convergence for the Ginzburg--Landau energy with no magnetic field, see \cite{giovanni-orlandi-baldo, BBO,Jerrard-Soner}. It is the two-form given by
            \begin{align}\label{jacobian-definition}
                J(u,\nabla) := \psi(u) + (1-|u|^2)\omega.
            \end{align}
            We have the trivial pointwise bound
            \begin{align}\label{Jacobian-energy-bound}
                |J(u,\nabla)| \leq e_\epsilon(u,\nabla),
            \end{align}
            where $e_\epsilon(u,\nabla)$ is the integrand in \cref{energy-definition}.
            
            We define $\Gamma_\epsilon$ to be the dual current to the Jacobian, formally identified by the duality formula
            \begin{align}
                \ang{\Gamma_\epsilon,\xi} = \frac{1}{2\pi}\int_{M} J(u,\nabla) \wedge \xi,
            \end{align}
            for any $(n-2)$-form $\xi\in\Omega^{n-2}(M)$. Note that by \cref{curvature-pde} the Jacobian can be written as
            \begin{align*}
                J(u,\nabla) = \omega + \epsilon^2\Delta_H\omega.
            \end{align*}
            In particular, this shows that the gauge-invariant Jacobian is a closed two-form. This, in duality, implies that $\de\Gamma_\epsilon = 0$, i.e., $\Gamma_\epsilon$ is an $(n-2)$-dimensional cycle.
            
            \section{Preliminary estimates}
            \subsection{The energy concentration set}
            It was proved by the third-named author and Stern in \cite{Pigati-1} that, for a sequence $(u_\epsilon,\nabla_\epsilon)$ with $\epsilon\to 0$, one can extract a subsequence such that the energy density converges to (the weight of) a stationary integer-rectifiable $(n-2)$-varifold. We restate the main result of \cite{Pigati-1} in the following theorem, see \cite[eq. (6.35)]{Pigati-1} for the conclusion on the Jacobian.
            
            \begin{thm}[The varifold limit]\label{varifold-limit-theorem}
                Let $L\rightarrow M$ be a Hermitian line bundle over a closed, oriented Riemannian manifold $(M^n,g)$ of dimension $n\geq 2$ and let $(u_\epsilon,\nabla_\epsilon)$ be a family of critical points of $E_\epsilon$, satisfying the uniform energy bound
                \begin{align*}
                    E_{\epsilon}(u_\epsilon,\nabla_\epsilon) \leq \Lambda < \infty.
                \end{align*}
                Then, as $\epsilon\rightarrow 0$, the energy measures
                \begin{align*}
                    \mu_\epsilon = \frac{1}{2\pi}e_\epsilon(u_\epsilon,\nabla_\epsilon)\,\vol_g,
                \end{align*}
                converge subsequentially, in duality with $C^0(M)$, to a measure $\mu$ which is the weight of a stationary integral $(n-2)$-varifold $V$. Also, for all $0\leq\delta < 1$,
                \begin{align*}
                    \spt(V) = \lim_{\epsilon\rightarrow 0} \{|u_\epsilon|\leq\delta\},
                \end{align*}
                in the Hausdorff topology. The $(n-2)$-currents dual to the curvature forms $\frac{1}{2\pi} \omega_\epsilon$ and Jacobians $\frac{1}{2\pi}J(u_\epsilon,\nabla_\epsilon)$ converge subsequentially to the same limit, an integral cycle $\Gamma$ with $|\Gamma| \leq \mu$.
            \end{thm}

            \begin{rmk}
                The previous result admits a local version, proved in the same way (assuming the bounds \cref{u-le-1} and \cref{modica-bounds} below, which in the closed case follow from the maximum principle): assume that we have an increasing sequence of open sets $U_\epsilon\subseteq\R^n$ and a sequence of smooth pairs $(u_\epsilon,\nabla_\epsilon)$, each defined on the trivial bundle $\C\times U_\epsilon$ and critical for $E_\epsilon$; if we have
                $$\limsup_{\epsilon\to0}\int_K e_\epsilon(u_\epsilon,\nabla_\epsilon)<\infty$$
                for any compact subset $K\subset U:=\bigcup_\epsilon U_\epsilon$, as well as \cref{u-le-1}--\cref{modica-bounds},
                then there exist a limiting varifold $V$ and a limiting cycle $\Gamma$ satisfying the same conclusions as above (up to a subsequence).
            \end{rmk}
            We will use the above theorem (in its local version) in several soft arguments by compactness and contradiction; in particular, we will use it to obtain information for any blow-down limit of an entire solution. 
            
            \subsection{Modica-type bounds and exponential decay}
            Actually, \cite{Pigati-1} contains some additional information which will be used frequently in the paper, including a Modica-type bound which was first proved in dimension two in \cite[Theorem III.8.1]{Taubes-2}. We record the following propositions in the non-compact case of $M = \R^{n}$, with the trivial bundle $L = \mathbb{C}\times\R^{n}$.

            \begin{proposition}\label{p:le1}
                A critical point $(u,\nabla)$ for $E_\epsilon$,
                on the trivial bundle on $\R^n$, satisfies
                \begin{equation}\label{u-le-1}
                    |u|\le1
                \end{equation}
                everywhere.
            \end{proposition}

            \begin{proof}
                The bound $|u|\le 1+C(n)\epsilon^2$ on the unit ball $B_1(0)$
                can be shown as in \cite[Proposition A.2]{Pigati-3},
                and the claim follows by scaling.
            \end{proof}
            
            \begin{proposition}[Modica-type bounds]\label{Modica-bound-proposition}
                Assuming also that the energy on a ball $B_R$ is $O(R^{n-2})$ for $R$ large enough, we have the pointwise bounds
                \begin{align}\label{modica-bounds}
                    \epsilon|F_\nabla| \leq \frac{1-|u|^2}{2\epsilon},\quad |\nabla u| \leq \frac{1-|u|^2}{\epsilon}.
                \end{align}
                \begin{proof}
                    The proof is essentially the same as in \cite{Pigati-1}; however, in the Euclidean space, the Modica-type bound has no error terms. First, define $\xi_\epsilon$ to be the \emph{discrepancy}:
                    \begin{align}
                        \xi := \epsilon|F_\nabla|-\frac{1-|u|^2}{2\epsilon}.
                    \end{align}
                    Arguing as in \cite[Section 3]{Pigati-1}, we see that
                    \begin{align}\label{disrepency-pde-preliminary}
                        \Delta \xi \geq \frac{|u|^2}{\epsilon^2}\xi.
                    \end{align}
                    For the positive part $\xi^+$, this immediately implies that
                    $$\Delta\xi^+\ge0$$
                    in the distributional sense, i.e., $\xi$ is subharmonic.
                    Under the energy growth assumption, we have
                    $$\int_{B_R(0)}|\xi|=O(R^{n-1}),$$
                    which gives $\xi^+\equiv0$, as claimed.
                    
                    For the second bound,
                    proceeding as in \cite[eqs. (5.5)--(5.6)]{Pigati-1},
                    we check that
                    \begin{align*}
                        w := |\nabla u| - \frac{1-|u|^2}{\epsilon}
                    \end{align*}
                    satisfies
                    \begin{align*}
                        \Delta w \geq \frac{|u|^2}{\epsilon^2}w + \frac{1}{\epsilon}\left(w + \frac{1-|u|^2}{\epsilon}\right)\left(2w + \frac{1-|u|^2}{2\epsilon}\right).
                    \end{align*}
                    Again, this implies that $w^+$ is subharmonic,
                    and hence $w^+\equiv0$.
                \end{proof}
            \end{proposition}
            We also record the following exponential decay
            of energy, which plays a key role in the paper.
            \begin{proposition}[Exponential decay away from the vorticity set]\label{exponential-decay-proposition}
                There exist constants $K(n)>0$ and $C(n)>0$ such that, defining $Z := \{ |u| \leq \frac34\}$ and $r(p) := \operatorname{dist}(p,Z)$, we have
                \begin{align}\label{exponential-decay-estimate}
                    e_\epsilon(u,\nabla) \leq C \frac{e^{-Kr(p)/\epsilon}}{\epsilon^2}.
                \end{align}
            \end{proposition}

            \begin{proof}
                As in \cite[Corollary 5.2]{Pigati-1},
                we compute that on $\R^n\setminus Z$ we have
                $$\Delta\frac{1-|u|^2}{2}\ge\frac{1-|u|^2}{4\epsilon^2}.$$
                Exponential decay now follows as in \cite[Proposition 5.3]{Pigati-1}, using also the previous Modica-type bounds.
            \end{proof}
            
            \subsection{Inner variations and monotonicity}
            In this section we recall the inner variation formulas for critical points. With respect to any orthonormal basis $\{e_k\}_{k=1}^n$ for $TM$, we define the $(0,2)$-tensors $\nabla u^*\nabla u$ and $\omega^*\omega$ by
            \begin{align}
                (\nabla u^*\nabla u)(e_j,e_k) &:= \ang{\nabla_{e_j} u,\nabla_{e_k}u},\label{nabla-tensor-definition}\\
                \omega^*\omega(e_i,e_j) &:= \sum_{k=1}^{n} \omega(e_i,e_k)\omega(e_j,e_k).\label{curvature-tensor-definition}
            \end{align}
            We define the \emph{stress-energy tensor} to be
            \begin{align}\label{stress-energy-tensor-definition}
                T_\epsilon(u,\nabla) := e_\epsilon(u,\nabla) - 2\nabla u^*\nabla u - 2\epsilon^2\omega^*\omega.
            \end{align}
            Then, for any pair $(u,\nabla)$ satisfying \cref{ymh-pde-1}--\cref{ymh-pde-2}, the inner variation formula then reads
            \begin{align}\label{div-free-stress-energy-tensor}
                \operatorname{div}(T_\epsilon(u,\nabla))=0,
            \end{align}
            meaning that, for any compactly supported vector field $X$,
            \begin{align}\label{testing-stress-energy-tensor}
                \int_{M} \ang{T_\epsilon(u,\nabla),DX} = 0.
            \end{align}
            A core tool in the proof of \cref{varifold-limit-theorem} is the \emph{monotonicity formula} from \cite[Theorem 4.3]{Pigati-1}, which is cleaner in the case of the trivial line bundle $L= \C\times\R^n$ over the flat Euclidean space $M=\R^n$. We state this version of the theorem for convenience and give a short proof.
            \begin{proposition}[Monotonicity formula]\label{monotonicity-prop}
                Let $(u,\nabla)$ be a critical point for $E_\epsilon$ on the trivial line bundle $L=\C\times\R^n\rightarrow \R^n$. Then the normalized energy
                \begin{align*}
                    \tilde{E}_\epsilon(p,r) := r^{2-n}\int_{B_r(p)} e_\epsilon(u,\nabla)
                \end{align*}
                satisfies
                \begin{align}\label{energy-n-2-monotonicity}
                    \frac{d}{dr} \tilde{E}_\epsilon(p,r) = 2r^{1-n}\int_{B_r(p)}\left( \frac{(1-|u|^2)^2}{4\epsilon^2} - \epsilon^2|\omega|^2\right) + 2r^{2-n}\int_{\de B_r(p)} (|\nabla_\nu u|^2 + |\iota_\nu\omega|^2).
                \end{align}
                \begin{proof}
                    Without loss of generality, assume that $p=0$. By approximation we can take $X(x)= \textbf{1}_{B_r(0)} \sum_{k=1}^{n} x_k e_k$ in \cref{testing-stress-energy-tensor}, obtaining
                    \begin{align*}
                        r\int_{\de B_r} e_\epsilon(u,\nabla)
                        &=\int_{B_r}(n-2)e_\epsilon(u,\nabla) + 2\int_{B_r}\left(\frac{(1-|u|^2)^2}{4\epsilon^2} - \epsilon^2|\omega|^2\right) \\
                        &\quad+ 2r\int_{\de B_r}\left(|\nabla_{\nu}u|^2
                        + \epsilon^2|\iota_\nu \omega|^2\right).
                    \end{align*}
              Since 
                    \begin{align*}
                        \frac{d}{dr} \tilde{E}_\epsilon(x,r) 
                        &= (2-n)r^{1-n}\int_{B_r} e_\epsilon(u,\nabla) + r^{2-n}\int_{\de B_r} e_\epsilon(u,\nabla)\\
                        &= 2r^{1-n}\int_{B_r}\left( \frac{(1-|u|^2)^2}{4\epsilon^2} - \epsilon^2|\omega|^2\right) + 2r^{2-n}\int_{\de B_r} (|\nabla_\nu u|^2 + |\iota_\nu \omega|^2),
                    \end{align*}
                    we obtain  the desired conclusion.
                \end{proof}
            \end{proposition}
            \subsection{Quantitative stability in two dimensions and the vortex equations}
            In this section we record some results regarding the existence, uniqueness, and quantitative stability of critical points for \cref{energy-definition} in $\R^2$. First of all note that for  $\epsilon=1$ the energy $E_1$ of  \emph{any} pair $(u,\nabla)$ can be written as follows:
            \begin{align}
                E(u,\nabla) &= \int_{\R^2} \left[|\nabla u|^2 + |F_\nabla|^2 + \frac{(1-|u|^2)^2}{4}\right] \nonumber\\
                &= 2\pi |N| + \int_{\R^2} |\nabla_{1} u \pm i \nabla_2 u |^2 + \left|\star \omega \mp \frac{1-|u|^2}{2}\right|^2,\label{Bogomolny-completion-of-squares}
            \end{align}
            where $N$ is the vortex number of $(u,\nabla)$, given by
            \begin{align*}
                N:=\frac{1}{2\pi}\int_{\R^2}\star \omega.
            \end{align*}
            Thus  $(u,\nabla)$ is a minimizer of the total energy among pairs with the same vortex number if and only if it satisfies the first-order system of \textit{vortex equations}:
            \begin{align}\label{vortex-equations}
                \nabla_1 u \pm i \nabla_2 u = 0\text{ and } \star \omega = \pm \frac{1-|u|^2}{2\epsilon}.
            \end{align}
            These are also called \emph{Bogomol'nyi equations} (after \cite{Bogomolny}) or \emph{self-dual equations}, and arise in many self-dual gauge theories. Taubes, in \cite{Taubes},  proved that we can prescribe the zero set $u^{-1}(0) = \{a_1,\dots,a_k\}$: given any finite collection of $k\ge0$ points, counted with multiplicity, there exists a solution $(u,\nabla)$ to the vortex equations (with either choice of signs, corresponding to vortex number $N=k$ and $N=-k$, respectively) with this prescribed zero set;
            moreover, the solution is unique up to change of gauge.
            
            In \cite{Halavati-stability} the second-named author improved the previous results by proving a (sharp) quantitative stability for  critical points of $E_1$. We record these results in the following theorem.
            
            \begin{thm}[Uniqueness and stability in two dimensions]\label{ymh-dim-2-thm}
                On the trivial line bundle over $\R^2$, any critical point  $(u,\nabla)$ of finite energy for $E_1$ is actually a minimizer with $E_1(u,\nabla) = 2\pi |N|$. Moreover, up to change of gauge, any minimizer is uniquely characterized by its zero set $u^{-1}(0) = \{a_1,\dots,a_k\}$ (counted with multiplicity, according to the local degree of $u$ around any zero) and orientation. Letting $\mathcal{F}$ be the moduli space of all minimizers, the following quantitative stability estimates hold:
                \begin{align}\label{ymh-dim-2-stability-estimate}
                    \inf_{(u_0,\nabla_0)\in\mathcal{F}}(\|u-u_0\|_{L^2(\R^2)}^2 + \|F_\nabla - F_{\nabla_0}\|^2_{L^2(\R^2)})
                    \leq C_{|N|} \left(E_1(u,\nabla) - 2\pi |N|\right),
                \end{align}
                for some constant $C_{|N|}>0$ and all pairs such that the discrepancy $E_1(u,\nabla)-2\pi |N|\le\delta_{|N|}$ is small enough.
                \begin{proof}
                    Existence and uniqueness were proved in \cite{Taubes,Taubes-1},
                    while quantitative stability was obtained in \cite{Halavati-stability}.
                \end{proof}
            \end{thm}
            The proof of the above theorem uses weighted estimates developed in \cite{Halavati-inequality}. Vaguely, \cref{ymh-dim-2-thm} tells us that in the vanishing $\epsilon$ limit, two-dimensional slices perpendicular to the energy concentration set resemble minimizing vortex solutions in $\R^2$. In the special case of pairs $(u,\nabla)$ with vortex number $\pm1$, we also have the stability of the Jacobian and the energy density, given by the following proposition.
            \begin{proposition}\label{Jacobian-energy-stability-prop}
                Given a pair $(u,\nabla=d-i\alpha)$ on $\R^2$ with vortex number $\pm1$, obeying the bound
                $$\|d|u|\|_{L^\infty(\R^2)} + \|u\alpha\|_{L^\infty(\R^2)} \leq \Lambda,$$
                if the discrepancy $E_1(u,\nabla) - 2\pi$ is small enough then
                \begin{align}\label{Jacobian-energy-stability-estimate}
                    \begin{aligned}
                        &\inf_{(u_0,\nabla_0)\in\mathcal{F}} \left(\|J(u,\nabla) - J(u_0,\nabla_0)\|_{L^1(\R^2)} + \|e_1(u,\nabla) - e_1(u_0,\nabla_0)\|_{L^1(\R^2)}\right) \\
                        &\leq C(\Lambda)\sqrt{E_1(u,\nabla) - 2\pi}.
                    \end{aligned}
                \end{align}
                \normalsize
                \begin{proof}
                    See \cite[Appendix C]{Halavati-stability}.
                \end{proof}
            \end{proposition}
            
            \section{Quantifying flatness and the excess}
            We assume that $n\ge 3$ throughout the rest of the paper,
            unless otherwise stated.
            \subsection{Excess definitions}
            In this section we introduce a way to measure \textit{flatness} of a pair $(u,\nabla)$. Inspired by the definition of tilt-excess by De Giorgi \cite{Degriogi-1}, we define the \emph{Yang--Mills--Higgs excess} as
            \begin{align}\label{excess-definition}
            \begin{aligned}
                \E(u,\nabla,B_r(x),S) &:= \frac{r^{2-n}}{2\pi}\int_{B_r(x)} [e_\epsilon(u,\nabla) - J(u,\nabla)\wedge e_S^*] \\
                &\phantom{:}= \mu_{\epsilon}(B_r(x)) - \ang{\Gamma_\epsilon,\textbf{1}_{B_r(x)}e_S^*}
                \end{aligned}
            \end{align}
            for any \emph{oriented} $(n-2)$-plane $S$ in $\R^n$ with the associated $(n-2)$-vector $e_S$ and $(n-2)$-covector $e_S^*$. Take an oriented orthonormal basis of $S = \operatorname{span}\{e_3,\dots,e_n\}$ and extend it to an orthonormal basis $\{e_1,\dots,e_n\}$ of $\R^n$. Then by a completion of squares we see that the excess splits into two terms:
            \begin{align}
                \E &= \E_1 + \E_2\nonumber,\\
                 \E_1 (u,\nabla,B_r(x),S) &:= \frac{r^{2-n}}{2\pi} \int_{B_r(x)} \left[\sum_{k=3}^n |\nabla_{e_k}u|^2 + \epsilon^2\sum_{(j,k)\neq(1,2)} \omega(e_j,e_k)^2\right]\label{excess-1-definition},\\
                 \E_2(u,\nabla,B_r(x),S) &:= \frac{r^{2-n}}{2\pi}\int_{B_r(x)} \left[|\nabla_{e_1}u + i\nabla_{e_2}u|^2 + \left|\epsilon\omega(e_1,e_2) - \frac{1-|u|^2}{2\epsilon}\right|^2\right]\label{excess-2-definition}.
            \end{align}
            Note that $\E_1$ quantifies how flat the solution is in the directions tangent to $S$, while $\E_2$ quantifies the error in the vortex equations on perpendicular slices.
            Moreover, $\E_1$ does \emph{not} depend on the orientation of $S$ (while $\E$ and $\E_2$ do).
            
            The Yang--Mills--Higgs excess is a key tool in our analysis. 
            For $S:=\{0\}\times\R^{n-2}$, with a slight abuse of notation, we define
            \begin{align*}
                \E_z = \frac{1}{2\pi}\int_{B^2_1\times \{z\}} [e_\epsilon(u,\nabla) - J(u,\nabla)(e_1,e_2)]
            \end{align*}
            for $z\in\R^{n-2}$, and similarly
            \begin{align*}
                (\E_1)_z &:= \frac{1}{2\pi}\int_{B_1^2\times \{z\}} \left[\sum_{k=3}^n |\nabla_{e_k}u|^2 + \epsilon^2\sum_{(j,k)\neq(1,2)} \omega(e_j,e_k)^2\right],\\
                 (\E_2)_z &:= \frac{1}{2\pi}\int_{B_1^2 \times \{z\}} \left[|\nabla_{e_1}u + i\nabla_{e_2}u|^2 + \left|\epsilon\omega(e_1,e_2) - \frac{1-|u|^2}{2\epsilon}\right|^2\right].
            \end{align*}
            
            \subsection{The tilt-excess decay statement}
            Parallel to De Giorgi's \cite{Degriogi-1} and Allard's \cite{Allard-1} regularity theorems, we aim to prove a  \emph{decay of the excess } up to scale $\epsilon$, compare with \cite[Theorem 3.3]{kelei-wang}. More precisely, our goal is to show \cref{tilt-excess-decay-statement},
            which is one of the main results of the present work. For convenience, we recall its statement here.
            
            \begin{thm}\label{thm:local}
                For any $n \geq 3$ and small enough $0<\rho\le\rho_0(n)$ there exist constants $C(n)>0$ and $\epsilon_0(n,\rho),\tau_0(n,\rho)$ such that the following holds.
                Let $(u,\nabla)$ be a critical point for the energy $E_\epsilon$, given by \cref{energy-definition}, with $\epsilon \leq \epsilon_0$. Assume that \(u\) satisfies the bounds \cref{u-le-1} and \cref{modica-bounds}, that  $u(0)=0$, and the energy bound
                \begin{align*}
                    \frac{1}{|B^{n-2}_1|}\int_{B^n_1} e_\epsilon(u,\nabla) \leq 2\pi+\tau_0.
                \end{align*}
                Then at least one of the following statements is true:
                {either}
                \begin{align}\E_1(u,\nabla,B_\rho^n,\bar{S}) \leq C(n)\rho^2\E_1(u,\nabla,B_1^n,S),
                \end{align}
                for some $(n-2)$-plane $\bar{S}$ with $\|P_{\bar{S}} - P_{S}\| \leq C(n)\sqrt{\E_1(u,\nabla,B_1^n,S)}$, where $P_S$ is the orthogonal projection onto $S$, the plane minimizing $\E(u,\nabla,B_\rho^n,\cdot)$, and $\|\cdot\|$ is the Hilbert--Schmidt norm, {or}
                \begin{align}
                    \E_1(u,\nabla,B_1^n,S) \leq \max\{C(n)\epsilon^2 |\log\E|^2\sqrt{\E},e^{-{K(n)}/{\epsilon}}\},
                \end{align}
                where $\E=\E(u,\nabla,B_1^n,S)$.
            \end{thm}
            Note that thanks to \cref{p:le1} and \cref{Modica-bound-proposition}, if \(u\) is an entire solution such that $\int_{B_R^n}e_\epsilon(u,\nabla)=O(R^{n-2})$, then \cref{u-le-1} and \cref{modica-bounds} are statisfied. In particular by scaling we deduce the following.

            \begin{thm}\label{global-tilt-excess-decay-thm}
                For any small enough $0<\rho\le\rho_0(n)$, there exist constants $C(n),R_0(n,\rho)>0$ and $\tau_0(n,\rho)$ with the following property. Let $(u,\nabla)$ be an entire critical point for $E_1$, with the energy bound
                \begin{align*}
                    \lim_{R\rightarrow \infty} \frac{1}{|B^{n-2}_R|}\int_{B_R^{n}} e_1(u,\nabla) \leq 2\pi+\tau_0.
                \end{align*}
                Then for all $R \ge R_0$ at least one of the following statements is true:
                {either}
                \begin{align}\E_1(u,\nabla,B_{\rho R}^n,\bar{S}) \leq C(n)\rho^2\E_1(u,\nabla,B_R^n,S),
                \end{align}
                for some $(n-2)$-plane $\bar{S}$ with $\|P_{\bar{S}} - P_{S}\| \leq C(n)\sqrt{\E_1(u,\nabla,B_R^n,S)}$ and $S$ minimizing $\E(u,\nabla,B_R^n,\cdot)$, {or}
                \begin{align}
                    \E_1(u,\nabla,B_R^n,S) \leq \max\{C(n)R^{-2} |\log\E|^2\sqrt{\E},e^{-{K(n)R}}\},
                \end{align}
                where $\E=\E(u,\nabla,B_R^n,S)$.
            \end{thm}
            \subsection{Blow-up at multiplicity one points}
            Allard's regularity theorem \cite{Allard-1} asserts that the energy concentration set in \cref{varifold-limit-theorem} is locally a $C^{1,\alpha}$ submanifold around points of multiplicity one. We use this to show that, 
            for any blow-down, the energy concentration set is a flat $(n-2)$-plane.
            \begin{proposition}[Multiplicity one and vanishing of excess]\label{excess-vanishes}
                For any $\delta>0$ there exist $\tau_0(n,\delta)>0$ and $\epsilon_0(n,\delta)>0$ small enough with the following property. Let $(u,\nabla)$ be a critical point for $E_\epsilon$ on the unit ball $B_1^n$, with $u(0) = 0$ and $\epsilon \leq \epsilon_0$, as well as the energy bound
                \begin{align*}
                    \frac{1}{|B_1^{n-2}|}\int_{B_1^n} e_\epsilon(u,\nabla) \leq 2\pi+\tau_0
                \end{align*}
                and \cref{u-le-1}--\cref{modica-bounds}.
                Then, after a suitable rotation and, possibly, a conjugation of $(u,\nabla)$,
                \begin{align*}
                    \E(u,\nabla,B_{1/2}^n,\R^{n-2}) \leq \delta,
                \end{align*}
                where we write $\R^{n-2}$ to mean $\{0\}\times\R^{n-2}$.
                As a consequence, given an entire critical point $(\tilde u,\tilde\nabla)$ for $E_1$, with $u(0) = 0$ and the energy bound
                \begin{align*}
                    \lim_{R\rightarrow \infty }\frac{1}{|B_R^{n-2}|}\int_{B_R^n} e_1(\tilde{u},\tilde\nabla) \leq 2\pi+\tau_0(n),
                \end{align*}
                then the previous limit is $2\pi$ and
                we can find oriented $(n-2)$-planes $S(R)$ such that
                \begin{align*}
                    \lim_{R\rightarrow \infty} \E(\tilde{u},\tilde\nabla,B_R^n,S(R)) = 0.
                \end{align*}
            \end{proposition}
                
                \begin{proof}
                    The proof is a standard argument by compactness and contradiction.

                    \fbox{\textit{Local case.}}
                    Assume that there are sequences $(u_\epsilon,\nabla_\epsilon)$ and $\tau_\epsilon\to0$ (as $\epsilon\rightarrow0$) such that
                    \begin{align*}
                        \int_{B_1^n} e_\epsilon(u_\epsilon,\nabla_\epsilon) \leq (2\pi+\tau_\epsilon)|B_1^{n-2}|
                    \end{align*}
                    and, on the other hand,
                    \begin{align}
                    \liminf_{\epsilon\rightarrow0}\E(u_\epsilon,\nabla_\epsilon,B_{1/2}^n,S(\epsilon)) > 0,
                    \end{align}
                    for any choice of oriented $(n-2)$-planes $S(\epsilon)$
                    (where, with abuse of notation, we write $\epsilon$ to mean a sequence $\epsilon_k\to0$). We apply \cref{varifold-limit-theorem}: up to extracting a subsequence, we have
                    $$e_{\epsilon}(u_{\epsilon},\nabla_{\epsilon})\,dx\weakstarto 2\pi\, d\mu_V$$
                    in duality with $C^0_c$, where $V$ is a stationary integral $(n-2)$-varifold whose weight $\mu_V$ obeys the bound
                    \begin{align}\label{liminfbound-1}
                        \mu_V(B_1^n) \leq \liminf_{\epsilon\to0} \frac{1}{2\pi} \int_{B_1^n}e_\epsilon(u_{\epsilon},\nabla_{\epsilon}) \leq |B_1^{n-2}|.
                    \end{align}
                    Moreover there exists an integral $(n-2)$-cycle $\Gamma$ such that $J(u_{\epsilon},\nabla_{\epsilon}) \weakto 2\pi\Gamma $ as currents and $|\Gamma|\leq \mu_V$.
                    
                    Since $u_{\epsilon}(0)=0$, by the clearing-out lemma \cite[Corollary 4.4]{Pigati-1} we get that $0 \in \spt(\mu_V)$, so that $\Theta^{n-2}(\mu_V,0) \geq 1$
                    since $V$ is an integral stationary varifold. Because of \cref{liminfbound-1}, the monotonicity formula for stationary varifolds is saturated, showing that $V$ must be a cone with respect to the origin; we extend it to a stationary cone $\tilde V$ on $\R^n$.
                    Since $\Theta^{n-2}(\mu_{\tilde V},x)\ge 1=\Theta^{n-2}(\mu_{\tilde V},0)$
                    for all $x\in\spt(\mu_{\tilde V})$,
                    we see that $\tilde V$ is a cone with respect to any $x\in\spt(\mu_{\tilde V})$, and hence a plane (since the tangent plane exists for a.e.\ point $x$).
                    Thus, up to a rotation, $V$ is the multiplicity-one varifold associated to $\{0\}\times\R^{n-2}$.

                    Moreover, the argument used in \cite[Section 6.2]{Pigati-1} to show integrality of $V$
                    actually reveals that the limiting density is the sum of the absolute values of the degrees of
                    $$\frac{u_\epsilon}{|u_\epsilon|}\Big|_{\de D_i\times\{z\}}$$
                    along typical slices $B_1^2\times\{z\}$ with $|z|<\frac12$,
                    where $D_1,\dots,D_N\subset B_{1/2}^2$
                    are suitable disjoint disks (depending on $z$)
                    such that $u_\epsilon(\cdot,z)\neq0$ on $B_{1/2}^2\setminus\bigcup_i D_i$ (see in particular the proof of \cite[Proposition 6.6]{Pigati-1} and the conclusion of \cite[Proposition 6.7]{Pigati-1}).
                    Since the limiting density is $1$ and, eventually, $u_\epsilon(y,z)\neq0$
                    for $y\in\de B_{1/2}^2$ and $z\in B_{1/2}^{n-2}$,
                    we see that
                    $$\operatorname{deg}\frac{u_\epsilon}{|u_\epsilon|}(\cdot,z)=1\quad\text{from $\de B_{1/2}^2$ to $S^1$}$$
                    eventually. As in \cite[Lemma 6.11]{Pigati-1},
                    we deduce that $\Gamma=\pm\llbracket \{0\}\times B^{n-2}_1 \rrbracket$.
                    We conclude that, after possibly replacing $(u,\nabla)$
                    with the conjugate pair, we have
                    \begin{align*}
                    0 &< \lim_{\epsilon\to0} \E(u_{\epsilon},\nabla_\epsilon,B_{1/2}^n,\R^{n-2}) \\
                    &= \lim_{\epsilon\to0} \frac{1}{2\pi}\int_{B_1^n} [e_{\epsilon}(u_{\epsilon},\nabla_{\epsilon}) - J(u_{\epsilon},\nabla_{\epsilon}) \wedge e_3^*\wedge\dots\wedge e_n^*] \\
                    &= \mu_V(B_{1/2}^n) - \ang{\Gamma,\uno_{B_{1/2}^n}e_3^*\wedge\dots\wedge e_n^*}\\
                    &= 0,
                    \end{align*}
                    which is the desired contradiction.

                    \fbox{\textit{Entire case.}}
                    For the case of an entire solution $(u,\nabla)$, we perform a rescaling: writing $\nabla=d-i\tilde\alpha$, let
                    \begin{align*}
                        u_\epsilon(x) := u(\epsilon^{-1}x),
                        \quad \nabla_\epsilon:=d-i\alpha_\epsilon\text{ with }\alpha_\epsilon(x):=\epsilon^{-1}\alpha(\epsilon^{-1}x).
                    \end{align*}
                    Again, by applying \cref{varifold-limit-theorem}, up to extracting a subsequence we have $e_{\epsilon}(u_{\epsilon},\nabla_{\epsilon})\,dx \weakstarto d\mu_V$, for a stationary integral $(n-2)$-varifold $V$, and the Jacobians $J(u_{\epsilon},\nabla_{\epsilon}) \weakto \Gamma$ for an integral $(n-2)$-cycle $\Gamma$ with the pointwise bound $|\Gamma|\leq\mu_V$. Using the monotonicity formula for $E_\epsilon$, we see that
                    $$\mu_V(B_R^n)=\lim_{\epsilon\to0}\frac{\epsilon^{n-2}}{2\pi}\int_{B_{R/\epsilon}^n}e_1(u,\nabla)$$
                    is a constant multiple of $R^{n-2}$, and hence
                    $V$ is a cone around the origin with $\mu_V(B^n_1) \leq 1+\frac{\tau_0}{2\pi}$. Then, by Allard's regularity theorem \cite{Allard-1}, we see that for $\tau_0(n)$ small enough, after a suitable rotation, $V$ is the varifold associated to $\{0\}\times\R^{n-2}$.
                    In particular, this shows the conclusion on the energy limit. 
                    
                    As before, we also have $\Gamma=\pm\llbracket\{0\}\times\R^{n-2}\rrbracket$, concluding that for $R:=\epsilon_k^{-1}$ (where $\epsilon_k\to0$ is our subsequence)
                    the statement holds for the plane $S(R):=\{0\}\times\R^{n-2}$, either for $(u,\nabla)$ or the conjugate pair $(\bar u,\bar\nabla=d+i\alpha)$ (depending on $R$). Since the initial sequence $\epsilon_k\to0$ was arbitrary, we deduce that
                    \begin{equation}\label{excess-R}
                        \min_{S\in\Gr(n,n-2)}\min\{\E(u,\nabla,B_R^n,S),\E(\bar u,\bar\nabla,B_R^n,S)\}\to0
                    \end{equation}
                    as $R\to\infty$. Finally, letting $S(R)$ realize the minimum over $S\in\Gr(n,n-2)$, since $\E_1\le\E$ does not distinguish between $(u,\nabla)$ and $(\bar u,\bar\nabla)$ we have
                    $$\E_1(u,\nabla,B_R^n,S(R))\to0.$$
                    As a consequence, we must have
                    \begin{equation}\label{planes-align}
                        \sup_{R'\in[R,2R]}\|P_{S(R')}-P_{S(R)}\|\to0,
                    \end{equation}
                    since otherwise we would find sequences $S(R_k)\to\hat S$ and $S(R_k')\to \hat S'\neq \hat S$ (with $R_k\le R_k'\le2R_k$) for which
                    \begin{equation*}
                        \E_1(u,\nabla,B_R^n,\hat S)+\E_1(u,\nabla,B_R^n,\hat S')\to0\quad\text{as $R=R_k\to\infty$},
                    \end{equation*}
                    thanks to the assumption $\int_{B_R^n}e_1(u,\nabla)=O(R^{n-2})$. If $\hat S\cup\hat S'$ spans $\R^n$, this immediately gives $\int_{B_R^n}e_1(u,\nabla)=o(R^{n-2})$, contradicting the assumption $u(0)=0$ and the clearing-out lemma. Otherwise,
                    their span is $(n-1)$-dimensional; letting $e_1$ be a unit vector orthogonal to it and completing to an orthonormal basis $\{e_1,\dots,e_n\}$ such that $e_2\perp \hat S$, we deduce that
                    $$\int_{B_R^n} \left[\sum_{j=2}^n|\nabla_{e_j}u|^2+|\omega|^2\right]=o(R^{n-2}).$$
                    Because of \cref{excess-R}, we also have
                    $$\min\{\E_2(u,\nabla,B_R^n,\hat S),\E_2(\bar u,\bar\nabla,B_R^n,\hat S)\}\to0,$$
                    giving
                    $$\int_{B_R^n}\left[||\nabla_{e_1}u|-|\nabla_{e_2}u||^2+\left||\omega(e_1,e_2)|-\frac{1-|u|^2}{2}\right|^2\right]\to0,$$
                    giving again the contradiction $\int_{B_R^n}e_1(u,\nabla)=o(R^{n-2})$.

                    Having established \cref{planes-align}, the claim follows by a straightforward continuity argument: for $R$ large enough we cannot have that
                    $$\E(u,\nabla,B_R^n,S(R)),\quad\E(\bar u,\bar\nabla,B_{R'}^n,S(R'))$$
                    are both small, for some $R'\in[R,2R]$, since this would imply that
                    $$\E_2(u,\nabla,B_R^n,S(R))+\E_2(\bar u,\bar\nabla,B_{R}^n,S(R))$$
                    is also small, which would give again small normalized energy on $B_R^n$; the same holds interchanging the roles of $R$ and $R'$, completing the proof.
                \end{proof}

            We also record the following consequence of the Hausdorff convergence of the vorticity set $Z = \{|u| \leq \frac34\}$.
            \begin{lemma}[Soft height bound]\label{soft-height-bound}
                For any $\sigma>0$ there exist $\tau_0(n,\sigma)>0$ and $\epsilon_0(n,\sigma)>0$ with the following property. Let $(u,\nabla)$ be a critical point for $E_\epsilon$ on $B_1^n$, with $\epsilon \leq \epsilon_0$ and $u(0)=0$, as well as the energy bound
                \begin{align*}
                    \frac{1}{|B_1^{n-2}|}\int_{B_1^n} e_\epsilon(u,\nabla) \leq 2\pi+\tau_0
                \end{align*}
                and \cref{u-le-1}--\cref{modica-bounds}.
                Then, after a suitable rotation, the zero set is contained in a small neighborhood of $\R^{n-2}$; more precisely,
                \begin{align*}
                    \{|u_\epsilon|\leq3/4\}\cap B_{1-\sigma}^n \subset B^2_{\sigma}\times B^{n-2}_1.
                \end{align*}
            \end{lemma}
            
                \begin{proof}
                    Following the same strategy as in the proof of \cref{excess-vanishes}, the statement follows from the Hausdorff convergence of the vorticity set in \cref{varifold-limit-theorem}.
                \end{proof}

            \begin{rmk}
                We will often use the following observation:
                if the excess $\E_1$ is suitably small on $B_1^n$, then the same conclusion holds without any rotation.
                The same holds under other assumptions forcing the vorticity set to concentrate on the plane $\R^{n-2}$ in the limit $\epsilon\to0$, such as energy close to $|B_1^{n-2}|\cdot2\pi$ on the cylinder $B_1^2\times B_1^{n-2}$, for a critical pair defined there (with $u(0)=0$).
            \end{rmk}
            
            In the following lemma we essentially show that if $\E_1$ is small in a ball of radius larger than $\epsilon$, then $\E$ is small as well.
            \begin{lemma}[$\E_1$ vanishing implies $\E$ vanishing]\label{excess-vanishes-when-excess-1-vanishes}
                For any $\delta,\Lambda>0$ there exist $\tau_0(n,\delta,\Lambda)>0$ and $\epsilon_0(n,\delta,\Lambda)>0$ small enough with the following property. Let $(u,\nabla)$ be a critical pair for $E_\epsilon$ on the unit ball $B_{1}^n$,
                with $u(0)=0$,
                $$E_\epsilon(u,\nabla)\le2\pi+\tau_0,$$
                and \cref{u-le-1}--\cref{modica-bounds}, as well as $\epsilon\le\epsilon_0$.
                Let $x\in B_{1-\delta}^n$ be a point such that
                \begin{align*}
                    \sup_{\epsilon\le s\leq1-|x|} \E_1(u,\nabla,B_s^n(x),\R^{n-2}) \leq \tau_0.
                \end{align*}
                Then, up to conjugating the pair,
                \begin{align*}
                    \sup_{\epsilon\le s\leq1-|x|} \E(u,\nabla,B_{s/2}^n(x),\R^{n-2}) \leq \delta.
                \end{align*}
                 \end{lemma}

                \begin{proof}
                    The proof of this lemma is basically the equivalence of the (second-order) Euler--Lagrange equations and the (first-order) vortex equations in two dimensions.
                    
                    By contradiction, assume we have a sequence $(u_k,\nabla_k)$ of critical points for $E_\epsilon$, with $\epsilon=\epsilon_k\to0$,
                    and a sequence of points $x_k\in B_{1-\delta}^n$ and radii $s_k\in[\epsilon_k,1-|x_k|]$ such that
                    \begin{align*}
                        \E_1(u_k,\nabla_k,B^{n}_{s_k}(x_k),\R^{n-2})\rightarrow 0,\quad\liminf_{k\to\infty}\E(u_k,\nabla_k, B^n_{s_k/2}(x_k),\R^{n-2}) \geq \delta.
                    \end{align*} 

We now distinguish a few cases depending on the behavior of the limit \(\epsilon_k/s_k\), which we can assume to exists and to belong to \([0,1]\), and on the distance of \(x_k\) from the vorticity set \(Z_k=\{x\in B_{s_k}^n(x_k): |u_k|\le 3/4\}\).
 
 \fbox{\textit{Case 1: \(\epsilon_k/s_k\to 0\) and \(\dist(x_k, Z_k)/s_k \to 0\).}} 
                    Since the energy concentration varifold is a plane with multiplicity 1 (as in the previous proof), recalling that $1-|x_k|\ge\delta$ and $x_k$ has vanishing distance from the vorticity set, we immediately see that
                    $$\frac{1}{|B_{1-|x_k|}^{n-2}|}\int_{B_{1-|x_k|}(x_k)}e_{\epsilon_k}(u_k,\nabla_k)\to2\pi.$$
                    Defining the map $\phi_k(x):=x_k+s_kx$, we consider the pullback pair
                    $$(\tilde{u}_k,\tilde\nabla_k) := \phi_k^*(u_k,\nabla_k),$$
                    which is critical for $E_{\tilde\epsilon_k}$,
                    where $\tilde\epsilon_k:=\epsilon_k/s_k$. Moreover,
                    we have
                    $$\limsup_{k\to\infty}\frac{1}{|B_{1}^{n-2}|}\int_{B_{1}^n}e_{\tilde\epsilon_k}(\tilde u_k,\tilde\nabla_k)\le2\pi$$
                    by monotonicity of the energy.
                    
                    Since $\tilde\epsilon_k\to0$ and $0$ has vanishing distance from $\{|\tilde u_k|\le\frac34\}\cap B_1^n$, as in the previous proof,
                    the energy concentration varifold $V$ is a plane $S$ passing through the origin, with multiplicity 1, while
                    the limiting cycle $\Gamma=\pm\llbracket S\rrbracket$. By possibly replacing \((\tilde u, \tilde{\nabla})\) with their conjugate, we can assume that \(\Gamma=\llbracket S\rrbracket\).
                    Also, the stress-energy tensors
                    $$T_{\tilde{\epsilon}_k}(\tilde u_k,\tilde\nabla_k),$$
                    viewed as matrix-valued measures,
                    converge (up to subsequences) to a limit $T$ such that $dT(x)=P_{T_xV}\,d\mu_V(x)$, where $P_{T_xV}$ is the orthogonal projection onto the tangent space $T_xV$
                    (cf.\ \cite[Section 6.1]{Pigati-1}).
                    Hence, the fact that $\E_1(\tilde u_k,\tilde\nabla_k,B_1^n,\R^{n-2})\to0$
                    implies $T_xV=\R^{n-2}$ a.e., giving $S=\R^{n-2}$. Since
                    $$\lim_{k\to\infty}\int_{B_{1/2}^n}[e_{\tilde\epsilon_k}(\tilde u_k,\tilde\nabla_k)-J(\tilde u_k,\tilde\nabla_k)\wedge e_S^*]=2\pi[\mu_V(B_{1/2}^{n})-\ang{\Gamma,\uno_{B_{1/2}^n}e_S^*}]=0$$
                    , we get the desired contradiction in this case.

 \fbox{\textit{Case 2: \(\epsilon_k/s_k\to 0\) and \(\dist(x_k, Z_k)/s_k \to 2d>0\).}} By performing the same scaling as in the previous step we get that \(|\tilde u_k|\) converges uniformly to \(1\) in \(B_d^n(0)\), which immediately implies that both excesses converges to \(0\) in \(B_{ds_k}^n(x_k)\) and thus the statement of the theorem with \(s/2\) replaced by \(ds\). A covering argument then allows to pass to \(s/2\).

 \fbox{\textit{Case 3: \(\epsilon_k/s_k\to \bar\epsilon >0\).}} Note  that this implies that \(s_k\to 0\). 
        
                    After passing to a local Coulomb gauge, for any $\ell\in\N$ we get local uniform $C^\ell$ bounds on $B_{R_k}^n$,
                    with $R_k:=\delta/s_k$, since by monotonicity we have local uniform bounds on the energy here, see \cite[Appendix]{Pigati-1}. By Arzelà--Ascoli we obtain a subsequential limit $(\tilde{u}_\infty,\tilde\nabla_\infty)$ in $C^\infty(\R^n)$.
                    By the definition of $\E_1$ (cf.\ \cref{excess-1-definition}), we see that $(\tilde\nabla_\infty)_{\de_k} \tilde{u}_\infty =0$ for all $3\leq k \leq n$ and $\tilde\omega_\infty(e_j,e_k) = 0$ for all $(j,k)\neq(1,2)$.
                    As in \cite[Proposition 6.7]{Pigati-1} (after \cite[eq. (6.30)]{Pigati-1}), up to a further change of gauge,
                    the limiting pair depends only on the first two coordinates.
                    By the equivalence of first-order and second-order vortex equations in $\R^2$ \cite{Taubes-1} (cf.\ also the end of the proof of \cite[Proposition 6.7]{Pigati-1}), we see that $(\tilde{u}_\infty,\tilde\nabla_\infty)$ solves the first-order vortex equations up to conjugation;
                    this yields a contradiction for $k$ large enough.
                \end{proof}

                \begin{lemma}\label{cone}
                For every  $\sigma>0$ there exist constants  $\eta (n,\sigma), C(n,\eta)>0$ such that if  $r\ge C\epsilon$ and  $(u,\nabla)$ is a critical pair on \(B_r(p)\) satisfying \cref{u-le-1}--\cref{modica-bounds} 
                and $\E_1(u,\nabla,B_r^n(p),\R^{n-2})\le \eta$ then
                $$\{|u|\le3/4\}\cap B_{(1-\sigma)r}^n(p)\subseteq B_{\sigma r}^2(y)\times\R^{n-2},$$
                provided that $|u(p)|\le\frac34$ at $p=(y,z)$ and the normalized energy is at most $2\pi+\eta$.
                
                Moreover, given $\sigma, \Lambda>0$ there are $\eta (n,\sigma,\Lambda), C(n,\eta,\Lambda)>0$ such that if 
                $$\E_1(u,\nabla,B_{C\epsilon}^n(p),\R^{n-2})\le \eta$$
                then $G:=\{u=0\}\cap B_{\Lambda\epsilon}^n(p)$ is a $\sigma$-Lipschitz graph,  we have the inclusion
                $$\{|u|\le3/4\}\cap B_{\Lambda\epsilon}^n(p)\subseteq B_{C(n)\epsilon}(G),$$
                and
                $\epsilon|u|$ is comparable with the distance from $S$ in this neighborhood $B_{C(n)\epsilon}(G)$.
            \end{lemma}
            
            \begin{proof}The first part follows by the very same arguments of \cref{excess-vanishes-when-excess-1-vanishes}. The second one is again showed by contradiction  after scaling by \(\epsilon\), noticing that in the Coulomb gauge the contradicting sequence $(u_k,\nabla_k)$ converges smoothly to a solution depending only on the two variables \((y_1,y_2)\). To infer the smooth convergence of the zero set (which is gauge invariant) one  notices that, by the explicit form of the Taubes solution, the Jacobian \(J u_k(e_1,e_2)\) is bounded away from zero. Convergence of the zero set then follows from the implicit function theorem. Compare also with the proof of \cref{Lipschitz-approximation-of-zero-loci-lemma}. 
            \end{proof}

            In the next lemma we show that
            the energy on each slice is approximately the excess on the slice plus the degree of $u$ on the boundary.
            \begin{lemma}\label{energy-identity-on-slice}
                Let $(u,\nabla)$ be an arbitrary smooth pair defined on $\bar B^2_1\times \bar B^{n-2}_1$ (not necessarily a critical point) with
                \begin{align*}
                    e_\epsilon(u,\nabla)(x) \leq e^{-K/\epsilon}\quad\text{for all }x\in\de B^2_1\times B^{n-2}_1
                \end{align*}
                and $|u(x)|\ge\mz$ for all $x$ in the same set.
                Then we have
                \begin{align*}
                    \left|\deg(u/|u|,\de B^2_1\times\{z\}) + \E_z - \frac{1}{2\pi}\int_{B^2_1\times\{z\}} e_\epsilon(u,\nabla) \right| \leq 4\epsilon e^{-K/\epsilon},
                \end{align*}
                for all $z\in B_1^{n-2}$, up to conjugating the pair.
            \end{lemma}
            
                \begin{proof}
                    First of all, by a completion of squares, since $$J(u,\nabla)(e_1,e_2) = 2\ang{i\nabla_1u,\nabla_2u} + (1-|u|^2)\omega(e_1,e_2),$$ we see that
                \begin{align}\label{display-1}
                \begin{aligned}
                &\frac{1}{2\pi}\int_{B^2_1\times\{z\}} e_\epsilon(u,\nabla)\\
                &=(\E_1)_z + \frac{1}{2\pi}\int_{B^2_1\times\{z\}} \left[|\nabla_1u|^2 + |\nabla_2u|^2 + \epsilon^2\omega(e_1,e_2)^2 + \frac{(1-|u|^2)^2}{4\epsilon^2}\right] \\
                &= (\E_1)_z + \frac{1}{2\pi}\int_{B^2_1\times x} \left[|i\nabla_1 u - \nabla_2 u|^2 + \left|\epsilon\omega(e_1,e_2) - \frac{1-|u|^2}{2\epsilon}\right|^2 + J(u,\nabla)(e_1,e_2)\right] \\
                &= \E_z + \frac{1}{2\pi}\int_{B^2_1\times\{z\}} J(u,\nabla)(e_1,e_2).
                \end{aligned}
                \end{align}
                We then define the modulus $r:B_1^2\rightarrow [0,\infty)$ and the phase $\theta:B^2_1\setminus\{r=0\}\rightarrow S^1$ by
                \begin{align*}
                    r(y) := |u(y,z)|,\quad\theta(y) := \frac{u}{|u|}(y,z).
                \end{align*}
                Writing $\nabla=d-i\alpha$, we also have
                $$r^2(d\theta-\alpha)(y) = \ang{\nabla u,iu}(y,z)$$
                (note that $\theta$ and $\alpha$ are not gauge-invariant). Recalling that $J(u,\nabla)=d\alpha+d\ang{\nabla u,iu}$, we compute
                \begin{align*}
                    \int_{B_1^2\times\{z\}} J(u,\nabla)
                    =\int_{\de B_1^2\times\{z\}} [(1-r^2)\alpha(\tau)+r^2\de_\tau\theta],
                \end{align*}
                where $\tau$ is the tangent vector to $\de B^2_1$.
                Hence, we have
                \begin{align*}
                    \int_{B_1^2\times\{z\}} J(u,\nabla)
                    =2\pi\operatorname{deg}(u/|u|,\de B_1^2\times\{z\})+\int_{\de B_1^2\times\{z\}} (1-r^2)[\alpha(\tau)-\de_\tau\theta],
                \end{align*}
                and the last integrand is bounded by
                $$(|u|^{-2}-1)|\ang{\nabla u,iu}|
                \le 4(1-|u|^2)|\nabla u|\le4\epsilon e_\epsilon(u,\nabla)$$
                in absolute value. Combining these bounds, the claim follows.
                \end{proof}
            
            \section{Slicing the current and Lipschitz approximation}
            In this section, inspired by  \cite{AK, Jerrard-Soner, Camillo-1}, we slice the currents $\Gamma_\epsilon$ dual to the Jacobians $J(u,\nabla)$. We get metric-space-valued functions of bounded variation (MBV) in the sense of Ambrosio \cite{Ambrosio-1}, with values in $0$-currents in $\R^2$. Then, by placing a threshold on the maximal function of $\E_1$, we construct a Lipschitz approximation of the barycenter of each slice with a uniform $W^{1,2}$ bound.
            
            \subsection{Slicing identities and BV estimates}
We start by defining vertical slices.
            \begin{definition}\label{slicing-definition}
                We define the \emph{vertical slices} of the current $\Gamma_\epsilon$,  $(\Gamma_\epsilon)_z=\ang{\Gamma_\epsilon,P,z}$, by the following identity:
                \begin{align*}
                        \int_{B^{n-2}_1(0)} \langle (\Gamma_\epsilon)_{z} , \psi \rangle \phi(z)\, dz = \langle \Gamma_\epsilon , \psi(y)\phi(z)dz \rangle,
                \end{align*}
                for any two functions $\psi \in C^{\infty}_c(B^2_1)$ and $\phi\in C^{\infty}_c(B^{n-2}_1)$, where $P:\R^2\times \R^{n-2} \rightarrow \R^{n-2}$ is the projection on the last $n-2$ coordinates.
            \end{definition}
            In the next lemma we derive BV estimates for the slices, given a smooth pair $(u,\nabla)$ defined on $B_1^2\times B_1^{n-2}$.
            \begin{lemma}[BV-type estimate]\label{BV-type-estimate}
                    Define the function $\Phi_{\psi}:B^{n-2}_1\rightarrow \R$ by 
                    \begin{align*}
                    	\Phi_{\psi}(z):=\ang{(\Gamma_\epsilon)_z,\psi}.
                    \end{align*}
                    Then, assuming $\int_{B_1^2\times B_1^{n-2}}e_\epsilon(u,\nabla)\le2\pi\Lambda$, the total variation of $\Phi_\psi(x)$ is bounded by $\E_1$ and $\E$ as follows:
                    \begin{align*}
                            \frac12|D\Phi_\psi|(B^{n-2}_1)^2 \leq \|d\psi\|_{L^{\infty}}^2 \Lambda \min\{C(n)\E_1, \E\},
                    \end{align*}
                    where $|D\Phi_\psi|$ denotes the total variation measure, and $\E$ and $\E_1$ are measured on $B_1^2\times B_1^{n-2}$
                    (without normalization).
                    \begin{proof}
                        The notation and line of argument is inspired from \cite[Lemma A.1]{Camillo-1}. For any $\phi \in C^{\infty}_c(B^{n-2}_1,\R^{n-2})$ we define the $(n-3)$-form $\alpha$ by
                        \begin{align*}
                            \alpha := \sum_{k=3}^{n}(-1)^{k-1}\phi_k(x_3,\dots,x_n)  dx_3\wedge\dots\wedge dx_{k-1}\wedge dx_{k+1}\wedge \dots \wedge dx_{n},
                        \end{align*}
                        so that
                        \begin{align*}
                            d\alpha = (\operatorname{div} \phi)(z)dz.
                        \end{align*}
                        Now, writing $x=(y,z)$, we have
                        \begin{align*}
                            \int_{B^{n-2}_1} \Phi_{\psi}(z) \div \phi(z)\, dz &=  \int_{B^{n-2}_1} \langle (\Gamma_{\epsilon})_z , \psi \rangle (\div \phi)(z)\,dz\\
                            & = \langle \Gamma_\epsilon , \psi(y) (\div \phi)(z) dz \rangle\\
                            & = \langle \Gamma_\epsilon , d(\psi\alpha) \rangle - \langle \Gamma_\epsilon , d\psi\wedge \alpha \rangle  \\
                            & = -\langle \Gamma_\epsilon , d\psi\wedge \alpha \rangle,
                        \end{align*}
                        where the last equality follows from the fact that $\partial \Gamma_{\epsilon} = 0$. Now notice that $d\psi\wedge\alpha$ is a linear combination of $(n-2)$-covectors of the form
                        \begin{align*}
                            &dx_j\wedge dx_3\wedge\dots\wedge dx_{k-1}\wedge dx_{k+1}\wedge \dots \wedge dx_{n}\quad\text{with }j=1,2,\ k=3,\dots,n.
                        \end{align*}
                        As a consequence,
                        \begin{align*}
                            | \langle \Gamma_\epsilon , d\psi\wedge \alpha \rangle | \leq \|d\psi\|_{L^\infty} \|\alpha\|_{L^\infty} \sum_{\substack{j=1,2\\k=3,\dots,n}}\int_{B_1^2\times B_1^{n-2}} [2|\ang{i\nabla_{e_j}u,\nabla_{e_k}u}|+ (1-|u|^2)|\omega(e_j,e_k)|],
                        \end{align*}
                        which, by Cauchy--Schwarz, is bounded by
                        \begin{align*}
                        \|d\psi\|_{L^\infty} \|\alpha\|_{L^\infty}\cdot C(n)\sqrt\Lambda\sqrt{\E_1}.
                        \end{align*}
                        Taking the supremum over the functions $\phi$ with $\|\phi\|_{L^\infty} \leq 1$, we get the BV bound
                        \begin{align*}
                            |D\Phi_\psi|(B^{n-2}_1) \leq C(n)\|d\psi\|_{L^{\infty}}\sqrt{\Lambda}\sqrt{\E_1}.
                        \end{align*}
                        
                        We can also estimate in the following way.  Set $B:=B_1^2\times B_1^{n-2}$ and
                        $$\vec e_{n-2}=e_3\wedge\dots\wedge e_{n},\quad e^*_{n-2} := dx_3\wedge\dots\wedge dx_n,$$
                        and let us write $d\Gamma_\epsilon=\overrightarrow{\Gamma}_\epsilon\,d|\Gamma_\epsilon|$ (viewing $\Gamma_\epsilon$ as a measure with values in $\Lambda_{n-2}\R^n$).
                        Since $d\psi\wedge \alpha$ does not have any $e^*_{n-2}$-component, if we write $\vec{\Gamma}_\epsilon = (\vec{\Gamma}_\epsilon\cdot \vec e_{n-2}) \vec e _{n-2}+ \vec{R}$
                        (where the dot denotes the scalar product in $\Lambda^{n-2}\R^n$), we get
                        \begin{align*}
                            \langle \Gamma_\epsilon , d\psi\wedge\alpha \rangle = \int_B \vec{R}\cdot(d\psi\wedge\alpha)\,d|\Gamma_\epsilon|,
                        \end{align*}
                        and moreover
                        \begin{align*}
                            \int_{B} |\vec{R}|^2\,d|\Gamma_\epsilon| &= \int_{B} (1-(\vec{\Gamma_\epsilon}\cdot\vec {e}_{n-2})^2)\, d|\Gamma_\epsilon| \\
                            &\leq 2 \int_{B} (1-\vec{\Gamma}_\epsilon\cdot\vec{e}_{n-2}) \,d|\Gamma_\epsilon| \\
                            &=2 e(\Gamma_\epsilon,B,\vec{e}_{n-2}),
                        \end{align*}
                         where $e(\Gamma_\epsilon,B,\vec{e}_{n-2})$ is the \emph{current excess} defined by
                    \begin{align*}
                        e(\Gamma_\epsilon,B,\vec{e}_{n-2}) := \mz\int_{B} |\overrightarrow{\Gamma}_\epsilon-\vec{e}_{n-2}|^2 \,d|\Gamma_\epsilon|.
                    \end{align*}
                        Hence,
                        \begin{align*}
                            | \langle \Gamma_\epsilon , d\psi\wedge \alpha \rangle | &= \left|\int_B \overrightarrow{R}\cdot(d\psi\wedge\alpha)\,d|\Gamma_\epsilon|\right|\\
                            & \leq |d\psi \wedge \alpha| \int_{B} |\vec{R}|\,d|\Gamma_\epsilon| \\
                            & \leq \|d\psi\|_{L^{\infty}} \|\alpha\|_{L^\infty} \sqrt{2e(\Gamma_\epsilon,B,\vec{e}_{n-2})}\sqrt{|\Gamma_\epsilon|(B)}\,.
                        \end{align*}
                        Again, taking the supremum over the functions $\phi$ with $\|\phi\|_{L^\infty} \leq 1$, we get
                        \begin{align*}
                            |D\Phi_\psi|(B^{n-2}_1) \leq \|d\psi\|_{L^{\infty}} \sqrt{2e(\Gamma_\epsilon,B,\overrightarrow{e}_{n-2})} \sqrt{|\Gamma_{\epsilon}|(B)}.
                        \end{align*}
                        From the pointwise bound of the Jacobian $|\Gamma_\epsilon| \leq \frac{1}{2\pi}e_\epsilon(u,\nabla)$ we see that
                        \begin{align*}
                        e(\Gamma_\epsilon,B,\overrightarrow{e}_{n-2})
                        =|\Gamma_\epsilon|(B)-\ang{\Gamma_\epsilon,\uno_{B}\overrightarrow{e}_{n-2}}
                        \leq \E.
                        \end{align*}
                        The previous bounds, together with $|\Gamma_\epsilon|(B) \leq \Lambda$, give the conclusion.
                    \end{proof}
                \end{lemma}
                \begin{rmk}\label{gradient-of-slice-identity}
                    The Jerrard--Soner-type computations in \cref{BV-type-estimate} are valid for any current without boundary (formally dual to a closed form). In the case of the Yang--Mills--Higgs Jacobian, we record the following identity for convenience (it will be used in \cref{Lipschitz-approximation-thm} and \cref{harmonic-approx-prop}):
                    \begin{align}
                        \ang{d\Phi_\psi,\phi} = \frac{1}{2\pi}\int_{B^2_1\times B_1^{n-2}} \sum_{j=1,2}\sum_{k=3}^{n} (-1)^{j}[\ang{2i\nabla_{e_j}u,\nabla_{e_k}u} + (1-|u|^2)\omega(e_j,e_k)]\de_{e_{3-j}}\psi\,\phi_k,
                    \end{align}
                    for any $\phi \in C^1_c(B^{n-2}_1,\R^2)$.
                \end{rmk}
                \subsection{Lipschitz approximation of the barycenter}
                Parallel to the regularity theory of minimal currents, we define a Lipschitz approximation of the barycenter of the slices of $\Gamma_\epsilon$ (see for instance \cite[Lemma A.2]{Camillo-1}). First we fix some notation which will be used frequently:
                \begin{itemize}
                    \item we use $\E_1$ as shorthand for $\E_1(u,\nabla,B^2_1\times B^{n-2}_1,\R^{n-2})$,
                    and similarly for $\E$;
                    \item as already mentioned, for any $z\in B^{n-2}_1$ we denote the excess on the slice $B^2_1\times\{z\}$ by $(\E_1)_z$,
                    and similarly for $\E_z$;
                    \item we write $M_{\E_1}(z)$ to denote the maximal function of $(\E_1)_z$;
                    \item we fix a cut-off function $\chi\in C^\infty_c(B_{3/4}^2)$ such that $0\le\chi\le1$ and $\chi=1$ on $B_{1/2}^2$.
                \end{itemize}
                \begin{proposition}[Lipschitz approximation]\label{Lipschitz-approximation-thm}
                    Given $0<\eta\le\eta_0(n)$ small enough, there exist
                    $\tau_0(n,\eta)>0$ and $\epsilon_0(n,\eta)>0$
                    such that the following holds.
                    Let $(u,\nabla)$ be a critical pair for $E_\epsilon$, defined on $B_1^2\times B_1^{n-2}$,
                    satisfying $u(0)=0$, \cref{u-le-1}--\cref{modica-bounds}, and the energy bound
                    $$\frac{1}{|B^{n-2}_1|}\int_{B_1^2\times B_1^{n-2}}e_\epsilon(u,\nabla)\le2\pi+\tau_0.$$
                    Then, up to conjugating $(u,\nabla)$, for $0<\eta\le\eta_0(n)$ small enough there exists a Lipschitz approximation $h:B^{n-2}_{3/4} \rightarrow \R^2$ with the following properties:
                    \begin{enumerate}[label=(\roman*)]
                        \item $\operatorname{Lip}(h)\leq C\eta$ and $\int_{B^{n-2}_{3/4}}|dh|^2 \leq C\E_1$;
                        \item $h|_{\G^\eta} = \Phi_{\chi(x_1,x_2)}$ for a set $\G^\eta \subseteq B_{3/4}^{n-2}$ such that $|B^{n-2}_{3/4}\setminus\G^\eta| \leq C \frac{\E_1}{\eta^2}$;
                        \item $\int_{B_{3/4}^2\times(B_{3/4}^{n-2}\setminus\G^\eta)} e_\epsilon(u,\nabla) \leq C\frac{\E_1}{\eta^2} + e^{-{K}/\epsilon}$;
                        \item $\int_{\G^\eta} \frac{|dh|^2}{2} \leq (1+\delta)\int_{\G^\eta}\E_{z}\,dz +e^{-K/\epsilon}$  with $\delta(n,\eta)>0$ such that $\lim_{\eta\rightarrow0}\delta(n,\eta) = 0$.
                    \end{enumerate}
                    Here  $C=C(n)>0$ and $K=K(n)>0$, provided that $\epsilon\le\epsilon_0$.
                    \begin{proof}
                        We define the \emph{good set} to be
                    \begin{align}\label{good-set-definition}
                        \G^\eta := \{z \in B^{n-2}_{3/4}\,:\,M_{\E_1}(z) \leq \eta^2\}.
                    \end{align}
                    By the weak $L^1$ bound and Vitali's covering lemma, we can bound the measure of the complement of the good set, namely the \textit{bad set}, by
                \begin{align}\label{bad-set-measure-bound}
                    \hau^{n-2}(B^{n-2}_{3/4} \setminus \G^\eta) \leq C(n)\frac{\E_1}{\eta^2}.
                \end{align}
                
                \fbox{\textit{Bounding energy on the bad set.}}
                To check that the third conclusion holds, we introduce another \emph{bad set} $\mathcal{B}^\eta$, defined on the $n$-dimensional space:
                it is the set of points $x=(y,z)\in B_{3/4}^2\times B_{3/4}^{n-2}$ such that, for some radius $r\in(0,\frac{1}{50})$, we have $\E_1(u,\nabla,B_r^n(x),\R^{n-2})>\eta^2$.
                
                By Vitali's covering lemma, we can cover $\mathcal{B}^\eta$ with balls $B_{5r_i}(x_i)$
                such that the balls $B_{r_i}(x_i)$ are disjoint and $\E_1(u,\nabla,B_{r_i}(x_i),\R^{n-2})>\eta^2$.
                By monotonicity of the energy, the energy on each dilated ball $B_{5r_i}(x_i)$ is at most $C(n)r_i^{n-2}$, giving
                $$\int_{\mathcal{B^\eta}}e_\epsilon(u,\nabla)\le\sum_i C(n)r_i^{n-2}
                \le\sum_i\frac{C(n)}{\eta^2}r_i^{n-2}\E_1(u,\nabla,B_{r_i}(x_i),\R^{n-2})\le\frac{C(n)}{\eta^2}\E_1$$
                (recall that the excess on a ball $B_r(x)$ is normalized by a factor $r^{2-n}$).
                Since the measure of $B_{3/4}^{n-2}\setminus\G^\eta$ obeys the same bound,
                it is enough to show that
                $$\int_{S}e_\epsilon(u,\nabla)\le C(n)$$
                for $\eta$ small enough, where $S:=(B_{3/4}^2\times\{z\})\setminus\mathcal{B}^\eta$.
                
                We denote by $d_Z$ the distance from the vorticity set $Z=\{|u|\le\frac34\}$.
                As we can see from the proof of \cref{soft-height-bound}, its conclusion holds without any rotation in the present situation (as necessarily energy concentrates along $\R^{n-2}$ as $\tau_0,\epsilon_0\to0$). Hence, we can assume that,
                for any $(y,z)\in S$ on this slice, we have $d_Z(y,z)\ge\frac{1}{200}$ unless $|y|<\frac{1}{100}$.
                                
                Given $s\ge\epsilon$, by \cref{cone} we know that if $$d_Z(y,z),d_Z(y',z)<s,$$
                for two points $(y,z),(y',z)\in S$ with $|y|,|y'|<\frac{1}{100}$,
                then
                $$|y-y'|\le Cs,$$
                provided that $\eta$, $\epsilon$ and $\epsilon/s$ are small enough.
                %
                With this observation in hand, we can apply \cref{exponential-decay-proposition} (giving $e_\epsilon(u,\nabla)(y,z)\le Ce^{-K\min\{d_Z(y,z),1/10\}/\epsilon}$ on $S$) and the coarea formula to write
                $$\int_{S}e_\epsilon(u,\nabla)\le \frac{C}{\epsilon^3}\int_0^2 e^{-Kt/\epsilon}|\{y\in B_{1/100}^2\,:\,d_Z(y,z)<t\}|\,dt+\frac{C}{\epsilon^3}e^{-K/(200\epsilon)}.$$
                The previous observation says that $\{y\in B_{1/100}^2\,:\,d_Z(y,z)<t\}$ is included in a ball of radius $C\max\{t,\epsilon\}$.
                We deduce that the last integral is bounded by
                $$\frac{C}{\epsilon^3}\int_0^2 e^{-Kt/\epsilon}C\max\{t^2,\epsilon^2\}\,dt\le C,$$
                giving the desired bound
                $$\int_{S}e_\epsilon(u,\nabla)\le C(n).$$                
                
                \fbox{\textit{Bounds in terms of $(\E_1)_z$.}}
                Now we establish Dirichlet energy bounds for $\Phi_\psi$ on the good set, for $\psi\in C^1_c(B_{3/4}^2)$. Given $z\in \G^{\eta}$, we can use \cref{gradient-of-slice-identity} to bound
                    \begin{align*}
                        |d\Phi_\psi|^2(z) &\leq C\sum_{j=1,2}\sum_{k=3}^{n} \left[\int_{B^2_1\times \{z\}}  (-1)^{j}\left(\ang{2i\nabla_{e_j}u,\nabla_{e_k}u} + (1-|u|^2)\omega(e_j,e_k)\right)\de_{e_{3-j}}\psi\right]^2\\
                        &\leq C \|d\psi\|_{L^\infty}^2\left[\int_{B^2_{3/4}\times\{z\}} |\nabla_{e_1} u|^2 + |\nabla_{e_2} u|^2 + \frac{(1-|u|^2)^2}{\epsilon^2}\right](\E_1)_z.
                    \end{align*}
                    
                    Since $z\in\G^\eta$, we have $S=B_{3/4}^2\times\{z\}$ in the previous argument. Thus, the last integral is bounded by $C(n)$.
                    As a consequence,                    
                    \begin{align*}
                        |d\Phi_\psi|^2(z) \leq C(n)\|d\psi\|_{L^\infty}^2(\E_1)_z\quad\text{for all }z\in \G^\eta.
                    \end{align*}
                    
                    \fbox{\textit{Bounds in terms of $\E_z$.}}
                    Also, we can use \cref{BV-type-estimate} (cf.\ \cite[Lemma A.2]{Camillo-1}) to conclude that
                    $$|d\Phi_{\chi(x_1,x_2)}|^2(z)
                    \le 2\E_z\lim_{r\to0}\frac{|\Gamma_\epsilon|(B_{1/2}^2\times B_r^{n-2}(z))}{|B_r^{n-2}|}+Ce^{-K/\epsilon}.$$
                    (indeed, this bound follows by applying \cref{BV-type-estimate} and its proof
                    with $\psi:=\chi(ax_1+bx_2)$ for an arbitrary $(a,b)\in S^1$ and using the fact that this $\psi$ is 1-Lipschitz on $B_{1/2}^2$, outside of which the energy density is exponentially small).
                    
                    To conclude, we have
                    $$|\Gamma_\epsilon|(B_{1/2}^2\times B_r^{n-2}(z))\le|B_r^{n-2}|+\int_{B_r^{n-2}(z)}\E_z+|B_r^{n-2}|e^{-K/\epsilon}$$
                    by \cref{energy-identity-on-slice}, giving
                    $$|d\Phi_{\chi(x_1,x_2)}|^2(z)\le 2\E_z(1+\E_z)+Ce^{-K/\epsilon},$$
                    where we can actually replace $\E_z$ with the excess on the smaller disk $B_{1/2}^2\times\{z\}$,
                    denoted by $\E_z(B_{1/2}^2)$.
                    Now, fixing $L>1$ large, by an obvious variant of \cref{excess-vanishes-when-excess-1-vanishes}
                    we have that $\E_z(B_{\epsilon}^2(y))$ is small for all $y\in B_{1/2}^2$ (see also the remark below). Since
                    we can cover the set $\{y\in B_{1/2}^2\,:\,d_Z(y,z)\le L\epsilon\}$
                    with $C(n,L)$ such disks, we infer that
                    $$\E_z(B_{1/2}^2)\le\delta(n,L,\eta)+\frac{C(n)}{\epsilon^3}\int_{L\epsilon}^2 e^{-Kt/\epsilon}|\{y\in B_{1/2}^2\,:\,d_Z(y,z)<t\}|\,dt+\frac{C(n)}{\epsilon^3}e^{-K/(4\epsilon)},$$
                    for some quantity $\delta(n,L,\eta)$ vanishing as $\eta\to0$. Choosing $L$ suitably large,
                    we deduce that
                    $$\E_z(B_{1/2}^{n-2})\le\delta(n,\eta)+e^{-K/\epsilon}$$
                    for some quantity $\delta(n,\eta)$ vanishing as $\eta\to0$. The statement follows by extending
                    $h:=\Phi_{\chi(x_1,x_2)}|_{\G_\eta}$ to a function with Lipschitz constant $C(n)\eta$.
                \end{proof}
                
                \begin{rmk}
                As a technical remark, a simple continuity argument as in \cref{excess-vanishes}
                shows that the possible need of conjugating the pair $(u,\nabla)$ in \cref{excess-vanishes-when-excess-1-vanishes}
                happens precisely when the degree of $u/|u|$ along the circle $\de B_{1/2}^2(0)\times\{0\}$ is $-1$ instead of $1$.
                \end{rmk}
                
                \end{proposition}
                \subsection{Lipschitz approximation of the zero set} In this part we collect information about the Lipschitz approximation of the zero set. We use compactness arguments similar to \cite[Sectoin 5]{kelei-wang}.
                \begin{proposition}[Zero set is Lipschitz on the good set]\label{Lipschitz-approximation-of-zero-loci-lemma} For any $\sigma,\delta>0$, there exists $\eta_0(n,\sigma,\delta)$ small enough with the following property. For $(u,\nabla)$ as in the previous statement, for any $\eta \le \eta_0(\delta,\sigma)$, the set $u^{-1}(0) \cap (B_{3/4}^2\times\G^{\eta})$ is included in a $\delta$-Lipschitz graph $h_0:B^{n-2}_{3/4} \rightarrow B_\sigma^2$ with the following estimate:
                \begin{align*}
                    \int_{B^{n-2}_{3/4}} |h_0-h|^2 \leq C\sigma^2\frac{\E_1}{\eta^2} + C\epsilon^2|\log(\E_2)|^2\E_2+Ce^{-K/\epsilon},
                \end{align*}
                for $C=C(n)$ (provided that $\epsilon\le\epsilon_0(n,\sigma,\delta)$ and the energy is $\le2\pi+\tau_0(n,\sigma,\delta)$).
                    \begin{proof}
                        The proof is similar to \cite[Lemma 5.3]{kelei-wang}.
        
                        \fbox{\textit{Lipschitz approximation at scale $\epsilon$.}}
                        This is essentially the second part of \cref{cone}, but we present a detailed argument here.
                        Notice that, locally at scale $\epsilon$, critical points enjoy uniform $C^k$ estimates in the Coulomb gauge (and thus $C^k$ bounds for gauge-invariant quantities): see \cite[Appendix]{Pigati-1}. Then around any $x_0=(y_0,z_0)\in B^{2}_{3/4}\times \G^\eta$ with $u(x_0)=0$ we rescale as follows:
                        \begin{align*}
                            \tilde{u}(x) := u(x_0 + \epsilon x ),\quad \tilde{\nabla} := \phi_{x_0,\epsilon}^*(\nabla),
                        \end{align*}
                        where $\phi_{x_0,\epsilon}$ is the map $x\mapsto x_0+ \epsilon x$. The resulting pair $(\tilde{u},\tilde\nabla)$ satisfies
                        \begin{align*}
                            \sup_{r\leq 1/(4\epsilon)} \E_1(\tilde{u},\tilde\nabla,B^2_{1/(4\epsilon)}\times B^{n-2}_r,\R^{n-2}) \leq \eta^2
                        \end{align*}
                        (where the excess is normalized by a factor $r^{2-n}$).
                        By Arzelà--Ascoli we conclude that, for small enough $\eta$,
                        $(\tilde{u},\tilde\nabla)$ is $C^1$-close to a pair $(u_0,\nabla_0)$ that satisfies the Yang--Mills--Higgs equations \cref{ymh-pde-1}--\cref{ymh-pde-2} and depends only on the variables $x_1,x_2$
                        (as in the proof of \cref{excess-vanishes-when-excess-1-vanishes}).
                        As noted in the proof of \cref{excess-vanishes}, $u_0/|u_0|$ has degree $\pm1$ on large circles,
                        and $u(\cdot,z_0)|_{B_{3/4}^2}$ 
                        
                        By the main result of \cite{Taubes,Taubes-1}, we deduce that $(u_0,\nabla_0)$ is the standard entire solution of degree $\pm1$, centered to vanish just at the origin. For this solution, we have
                        \begin{align}
                            |Ju_0|(0)>0,
                        \end{align}
                        where $Ju_0$ is the Jacobian of $u_0$ in the local Coulomb gauge in $B_1^n$. It then follows that, for small enough $\eta>0$, we have $|J\tilde{u}(e_1,e_2)| \geq c>0$. Then, by an application of the implicit function theorem and the fact that $\{\tilde{u}=0\}$ is a gauge-invariant set, we see that $\{\tilde{u}=0\}$ is locally a Lipschitz graph with a (qualitatively) small Lipschitz constant. The fact that the zero set
                        intersects the slice only at $x_0$ follows from \cref{cone}, which says that there is no zero outside a $C\epsilon$-neighborhood of $x_0$,
                        while in this neighborhood uniqueness follows from the fact that it holds for $u_0$ (see also a similar argument in the proof of \cite[Theorem 4.1]{Halavati-stability}). Hence, for small enough $\eta$, we can define a function $h_0:\G^{\eta} \rightarrow \R^2$ such that
                        \begin{align*}
                            \{u=0\} \cap (B_{3/4}^2\times \G^\eta) = \operatorname{graph}(h_0).
                        \end{align*}
                        
                        \fbox{\textit{Lipschitz approximation at larger scales.}}
                        Using the first part of \cref{cone}, we see that given two points
                        $(y,z),(y',z')\in\{u=0\} \cap (B_{3/4}^2\times \G^\eta)$ we have
                        $$|z-z'|\le\delta|y-y'|\quad\text{if }|y-y'|\ge C(n,\delta)\epsilon,$$
                        for a constant $\delta=\delta(\eta)>0$ such that
                        $$\delta(\eta)\to0\quad\text{as }\eta\to0.$$
                        Together with the previous control at scales comparable with $\epsilon$, this tells us that $h_0$ is indeed Lipschitz, with $\operatorname{Lip}(h_0)$ vanishing as $\eta\to0$. We apply the classical extension theorem to build a Lipschitz extension of $h_0$ defined on $B_s^{n-2}$.
                        
                        \fbox{\textit{$L^2$ estimates.}}
                        Using the soft height bound of \cref{soft-height-bound} (note that no rotation is needed in the present situation, as necessarily the energy concentrates along $\R^{n-2}$), we have
                        $$|h|+|h_0|\le\sigma$$
                        for $\eta$ (and hence $\epsilon$) small enough.
                        Using the estimates of \cref{zero-set-distance-barycenter-slice} on the good set $\G^\eta$ (see also \cref{concave}) and the measure bound for the bad set $B^{n-2}_{3/4}\setminus\G^\eta$ we see that
                        \begin{align*}
                            \int_{B^{n-2}_{3/4}} |h_0 - h|^2 &\leq \int_{\G^{\eta}} |h_0-h|^2 + \int_{B^{n-2}_{3/4}\setminus\G^\eta} |h_0-h|^2\\
                            &\leq C\epsilon^2|\log\E_2|^2 \E_2 + C\sigma^{2} \frac{\E_1}{\eta^2}+Ce^{-K/\epsilon}.
                        \end{align*}
                        We thus get the desired conclusion.
                    \end{proof}
                \end{proposition}
                \begin{rmk}
                    We remark that the function $h_0$ is well-behaved under small rotations, since the construction also rotates. However, the Lipschitz approximation of the slice barycenters, a priori, might not behave well under rotations.
                \end{rmk}
                
                \section{Harmonic approximation and a Caccioppoli-type estimate}
                \subsection{Harmonic approximation}
                In this section we show that the Lipschitz approximation of \cref{Lipschitz-approximation-thm} nearly satisfies the Laplace equation. We achieve this by relating the stress-energy tensor to the slices of $\Gamma_\epsilon$ using the self-dual discrepancy excess $\E_2$. Then we use this with uniform $W^{1,2}$ bounds to show that the Lipschitz approximation is well approximated in $L^2$ by a harmonic function. To begin with, we state a very well-known lemma.
                \begin{lemma}\label{harmonic-approximation-lemma-1}
                        For any $\nu>0$ small there exists $\tau(n,\nu)>0$ with the following property. Let $f$ be a function in $ W^{1,2}(B_1^n)$ such that
                        \begin{align*}
                            \int_{B_1^n}|\nabla f|^2 \leq 1,\quad\left|\int_{B_1^n} \ang{d f, d \phi}\right| \leq \tau\|d\phi\|_{L^\infty},
                        \end{align*}
                        for any $\phi\in C^1_c(B_1^n)$. Then there exists a harmonic function $w:B_1^n\to\R$ such that
                        \begin{align*}
                            \int_{B_1^n}|dw|^2\le1,\quad\int_{B_1^n} |w-f|^2 \leq \nu.
                        \end{align*}
                        Moreover, if $f$ has zero average, we can choose \(w\) so that  $w(0)=0$.
                              \end{lemma}

                        \begin{proof}
                            The claim follows easily from Rellich's compact embedding theorem: see for instance \cite[Lemma 6.1]{DeLellis2017ALLARDSIR}.
                            For the second part,  by the mean value property of harmonic functions and \(\int f=0\) one gets that
                            \[
                            |B_1^n|\lvert w(0)\rvert=\left\lvert\int_{B_1} w-\int_{B_1^n}f\right\rvert \le C(n)\|w-f\|_{L^2} \le \sqrt{\nu}.
                            \]
                            The function \(w-w(0)\) satisfies the conclusion of the lemma.
                        \end{proof}
                \begin{proposition}[Harmonic approximation]\label{harmonic-approx-prop}
                    Let $(u,\nabla)$ be a critical point of $E_\epsilon$ as in the previous section and let $h:B_{3/4}^{n-2}$ be the Lipschitz approximation built in \cref{Lipschitz-approximation-thm} for $\eta$. Then there exist constants $C(n),K(n)>0$ such that, for any test function $\phi \in C^\infty_c(B^{n-2}_{3/4},\R^2)$, we have
                    \begin{align*}
                        \left|\int_{B^{n-2}_{3/4}} \ang{d h, d \phi} \right| \leq C(\eta^{-1}\E_1+\sqrt{\E\E_1}+e^{-K/\epsilon})\|d\phi\|_{L^\infty}.
                    \end{align*}
                    Moreover, given any $\nu>0$, if $e^{-K/\epsilon}\le\E_1$ and $\E$ is small enough (depending on $n,\eta,\nu$),
                    there exists a harmonic function $w:B^{n-2}_{3/4}\rightarrow\R^2$ with $w(0)=0$ such that
                    \begin{align*}
                        \int_{B_{3/4}^{n-2}} |dw|^2 \leq C,\quad\int_{B_{3/4}^{n-2}} |(\E_1)^{-1/2}(h-c)-w|^2 \leq \nu,
                    \end{align*}
                    where $c$ is the average of $h$.
                    \begin{proof}
                        First, we define the vector field $X := \phi(x_3,\dots,x_n)e_1$ for any compactly supported test function $\phi \in C^\infty_c(B_{3/4}^{n-2})$, and we test \cref{div-free-stress-energy-tensor} with $\psi(x_1,x_2)X$, where $\psi$ is a smooth cut-off function such that $\psi=1$ on $B^2_{1/2}$ and $\psi=0$ outside of $B_{3/4}^2$. We obtain
                        \begin{align*}
                            \left|\int_{B_{1/2}^2 \times B_{3/4}^{n-2}} \ang{T_\epsilon(u,\nabla), DX}\right| \leq C e^{-K/\epsilon}\|d\phi\|_{L^\infty},
                        \end{align*}
                        thanks to the fact that $d\psi$ is supported in the annulus $B^2_{3/4}\setminus B^2_{1/2}$ and the exponential decay away from the vorticity set $Z$, which intersects $B_{3/4}^2\times B_{3/4}^{n-2}$
                        only inside $B_{1/4}^2\times B_{3/4}^{n-2}$. Then, since $DX$ is traceless, we compute
                        \begin{align*}
                            &\int_{B_{1/2}^2 \times B_{3/4}^{n-2}} \ang{T_\epsilon(u,\nabla), DX} \\
                            &= -2\int_{B_{1/2}^2 \times B_{3/4}^{n-2}} \sum_{k=3}^n\left[\ang{\nabla_{e_1} u , \nabla_{e_k} u} + \epsilon^2\sum_{j=1}^{n}\omega(e_1,e_j)\omega(e_k,e_j)\right]\de_{e_k}\phi.
                        \end{align*}
                        Except for $j=2$, the integral of the terms involving the curvature $\omega$ is bounded by $C(n)\E_1$, giving
                        \begin{align}\label{expr.from.t}
                        \left|\int_{B_{1/2}^2 \times B_{3/4}^{n-2}} \sum_{k=3}^n[\ang{\nabla_{e_1} u , \nabla_{e_k} u}
                        +\epsilon^2\omega(e_1,e_2)\omega(e_k,e_2)]\de_{e_k}\phi\right|
                        \le C(\E_1+e^{-K/\epsilon})\|d\phi\|_{L^\infty}.
                        \end{align}
                        
                        We now want to relate the expression in the left-hand side with the identity
                        for $d\Phi_{\chi x_1}$ obtained in \cref{gradient-of-slice-identity}, which in particular gives
                        \begin{align*}
                        \left|\int_{B_{3/4}^{n-2}}\ang{d\Phi_{\chi x_1},d\phi}\right|
                        &\le C\left|\int_{B^2_{1/2}\times B_{3/4}^{n-2}}\sum_{k=3}^n[\ang{2i\nabla_{e_2} u , \nabla_{e_k} u} +(1-|u|^2)\omega(e_2,e_k)]\de_{e_k}\phi\right|\\
                        &\quad+Ce^{-K/\epsilon}\|d\phi\|_{L^\infty},
                        \end{align*}
                        or equivalently
                        \begin{align}\label{expr.from.js}\begin{aligned}
                        \left|\int_{B_{3/4}^{n-2}}\ang{d\Phi_{\chi x_1},d\phi}\right|
                        &\le C\left|\int_{B^2_{1/2}\times B_{3/4}^{n-2}}\sum_{k=3}^n\left[-\ang{i\nabla_{e_2} u , \nabla_{e_k} u} +\frac{1-|u|^2}{2}\omega(e_k,e_2)\right]\de_{e_k}\phi\right|\\
                        &\quad+Ce^{-K/\epsilon}\|d\phi\|_{L^\infty}.\end{aligned}
                        \end{align}
                        We observe that the two integrals in \cref{expr.from.t} and \cref{expr.from.js}
                        differ by the integral of
                        $$\sum_{k=3}^n\left[\ang{\nabla_{e_1}u+i\nabla_{e_2} u , \nabla_{e_k} u} +\left(\omega(e_1,e_2)-\frac{1-|u|^2}{2}\right)\omega(e_k,e_2)\right]\de_{e_k}\phi.$$
                        Hence, recalling the definition of $\E_2$ and using Cauchy--Schwarz, we conclude that
                        $$\left|\int_{B_{3/4}^{n-2}}\ang{d\Phi_{\chi x_1},d\phi}\right|
                        \le C(n)(\E_1+\sqrt{\E_1\E}+e^{-K/\epsilon})\|d\phi\|_{L^\infty}.$$
                        
                        Repeating the same for $\Phi_{\chi x_2}$, we arrive at the same conclusion for $\Phi_{\chi(x_1,x_2)}$,
                        integrated against any $\phi\in C^\infty_c(B_{3/4}^{n-2},\R^2)$.
                        To conclude we note that thanks to  items (i) and (ii) of \cref{Lipschitz-approximation-thm},
                        $$\int_{B_{3/4}^{n-2}\setminus\G^\eta}|dh|\le C\eta|B_{3/4}^{n-2}\setminus\G^\eta|\le C\frac{\E_1}{\eta}$$
                        and, in view of \cref{gradient-of-slice-identity}, Cauchy--Schwarz, item (iii) of \cref{Lipschitz-approximation-thm} and the assumption \(e^{-K/\epsilon}\le \E_1\),
                        $$\int_{B_{3/4}^{n-2}\setminus\G^\eta}|d\Phi_{\chi(x_1,x_2)}|
                        \le C\sqrt{\E_1}\left(\int_{B_{3/4}^2\times B_{3/4}^{n-2}\setminus \G^\eta}e_\epsilon(u,\nabla)\right)^{1/2}
                        \le C\frac{\E_1}{\eta}.$$
                        
                        The second part follows from \cref{harmonic-approximation-lemma-1},
                        noting that the normalized function $\tilde h:=(\E_1)^{-1/2}h$
                        has Dirichlet energy bounded by $C(n)$ by item (i) of \cref{Lipschitz-approximation-thm}.
                \end{proof}
                \end{proposition}
                \subsection{Caccioppoli-type estimates}\label{Caccioppoli-decay-section}
                The starting point in the regularity theory of elliptic partial differential equations is the Caccioppoli--Leray bound, obtained by testing the equation with $\phi^2 u$,
                where $\phi$ is a cut-off function and $u$ is the solution. We aim to do something similar in spirit. Here the \textit{function} that we deal with is the barycenter of the energy measure at any slice. This suggests that testing the stress-energy tensor with $\phi^2 x_1 e_1$ and $\phi^2 x_2e_2$ is an appropriate choice.
                \begin{proposition}\label{Caccioppoli-prop-varifold}
                	Let $(u,\nabla)$ be a critical point of $E_\epsilon$ as in the previous section.
                    For any $\sigma>0$ there exist $\epsilon_0(n,\sigma)$ and $\tau_0(n,\sigma)$ small enough such that the following Caccioppoli-type estimate holds:
                    \begin{align*}
                        \int_{B^{n-2}_{3/4}} \phi^2(z) (\E_1)_z\,dz \leq C \int_{B_{1/2}^2\times B^{n-2}_{3/4}} (x_1^2 + x_2^2) e_\epsilon(u,\nabla)\Delta \phi^2 + C(\sigma^2 \E_1+e^{-{K}/{\epsilon}})\|D^2 \phi\|_{\infty},
                    \end{align*}
                    for any test function $\phi\in C^{\infty}_c(B^{n-2}_{3/4})$, where $C=C(n)$ and $K=K(n,\sigma)$.
                    \begin{proof}
                        First, we define the vector fields
                        \begin{align*}
                            X := \sum_{k=3}^n\de_k\phi^2(x_3,\dots,x_n)\frac{x_1^2 + x_2^2}{2} e_k,\quad Y := \phi^2(x_3,\dots,x_n) (x_1 e_1 +  x_2 e_2)
                        \end{align*}
                        and we calculate their derivatives:
                        \begin{align*}
                            DX &= \frac12\sum_{3\leq j,k\leq n} \de_{e_j,e_k}^2 \phi^2(x_1^2+x_2^2) e_j\otimes e_k^* + \sum_{k=3}^n \de_{e_k}\phi^2(x_1e_k\otimes e_1^*+ x_2e_k\otimes e_2^*)\\
                            DY &= \phi^2 (e_1 \otimes e_1^*+ e_2\otimes e_2^*) +\sum_{k=3}^n \de_{e_k}\phi^2(x_1e_1\otimes e_k^* + x_2e_2\otimes e_k^*).
                        \end{align*}
                        Then we test \cref{div-free-stress-energy-tensor} with $\chi X$ and $\chi Y$, where $\chi=\chi(x_1,x_2)$ is a smooth cut-off function such that $\chi=1$ on $B_{1/2}^2$ and $\chi=0 $ on $B^2_1\setminus B_{3/4}^2$. We note that the terms containing $d\chi$ are supported in $(B^2_{3/4}\setminus B^2_{1/2})\times B_{3/4}^{n-2}$, where $|T_\epsilon(u,\nabla)|\le C(n)e_\epsilon(u,\nabla)$ is very small by the exponential decay. Hence,
                        \begin{align*}
                            \left| \int_{B^2_{1/2}\times B^{n-2}_{3/4}} \ang{T_\epsilon(u,\nabla),DY} - \ang{T_\epsilon(u,\nabla),DX} \right| \leq C\|D^2\phi\|_{L^\infty} e^{-K/\epsilon}.
                        \end{align*}
                        Using the previous expansion of $DX$ and $DY$, together with the symmetry of $T_\epsilon(u,\nabla)$, we see that the above integrand equals
                        \begin{align*}
                        	&\frac12\sum_{3\leq j,k\leq n} T_\epsilon(u,\nabla)(e_j,e_k)\de_{e_j,e_k}^2 \phi^2(x_1^2+x_2^2)
                        	-\sum_{j=1,2}T_\epsilon(u,\nabla)(e_j,e_j)\phi^2 \\
                        	&=\frac{x_1^2+x_2^2}{2}\left[e_\epsilon(u,\nabla)\Delta\phi^2
                        	-2\sum_{3\le j,k\le n}(\nabla u^*\nabla u+\epsilon^2\omega^*\omega)(e_j,e_k)\de_{e_j,e_k}^2 \phi^2\right]\\
                        	&\quad-2\left[e_\epsilon(u,\nabla)-|\nabla_{e_1}u|^2-|\nabla_{e_2}u|^2-\sum_{j=1,2}\sum_{k=1}^n\epsilon^2\omega(e_j,e_k)^2\right]\phi^2.
                        \end{align*}
                        By the Modica-type inequality \cref{modica-bounds}, the last expression multiplying $-2\phi^2$ is bounded below by
                        $$\sum_{k=3}^n|\nabla_{e_k}u|^2+\frac{(1-|u|^2)^2}{4\epsilon^2}-\epsilon^2\omega(e_1,e_2)^2
                        \ge \sum_{k=3}^n|\nabla_{e_k}u|^2+\epsilon^2\sum_{(j,k)\neq(1,2)}\omega(e_j,e_k)^2$$
                        (where the last sum is over all pairs $(j,k)\neq(1,2)$ with $j<k$), which is the integrand in the definition of $\E_1$.
                        Hence, combining the previous bounds, we arrive at
                        \begin{align*}\int_{B^2_{1/2}\times B^{n-2}_{3/4}} \phi^2(z)(\E_1)_z\,dz
                        &\le\int_{B^2_{1/2}\times B^{n-2}_{3/4}}\frac{x_1^2+x_2^2}{2}e_\epsilon(u,\nabla)\Delta\phi^2 \\
                        &\quad+C\int_{B^2_{1/2}\times B^{n-2}_{3/4}} \frac{x_1^2+x_2^2}{2}\sum_{3\le j,k\le n}(\nabla u^*\nabla u+\epsilon^2\omega^*\omega)(e_j,e_k)\de_{e_j,e_k}^2 \phi^2 \\
                        &\quad+C\|D^2\phi\|_{L^\infty}e^{-K/\epsilon}.
                        \end{align*}
                        Now, by the soft height bound, we can assume that the vorticity set $Z$
                        intersects $B_{1/2}^2\times B_{3/4}^{n-2}$ in a small cylinder
                        $B_\sigma^2\times B_{3/4}^{n-2}$; the conclusion follows by exponential decay, up to replacing $K$ with another constant $K(n,\sigma)$.
                    \end{proof}
                \end{proposition}
                \begin{rmk}\label{rmk:cacciopoli}
                    In the statement of \cref{Caccioppoli-prop-varifold} we can replace the first term of the right-hand side as follows:
                    \begin{align*}
                        \int_{B^{n-2}_{3/4}} \phi^2(z) (\E_1)_z\,dz &\leq C \int_{B_{1/2}^2\times B^{n-2}_{3/4}} [(x_1-c_1)^2 + (x_2-c_2)^2] e_\epsilon(u,\nabla)\Delta \phi^2 \\ \quad + C(\sigma^2  \E_1+e^{-{K}/{\epsilon}})\|D^2 \phi\|_{\infty},
                    \end{align*}
                    provided that $|c|\leq C\sigma$ for $C=C(n)$. The proof is essentially the same.
                \end{rmk}
                    \section{Proof of decay of the tilt-excess}
                    In this section we prove \cref{tilt-excess-decay-statement}:
                    roughly speaking, we prove that $\E_1$, the first part of the excess, decays up to scales where it becomes comparable with $\epsilon^2$. We will deduce this from the excess decay property of harmonic functions, stated in the next elementary lemma.
                    \begin{lemma}\label{decay-lemma-harmonic-functions}
                        Given a harmonic function $w:B^n_1(0) \rightarrow \R$, we have the decay estimate
                        \begin{align}
                            \sup_{x\in B_\rho^n(0)}|w(x)-w(0)-d w(0)[x]| \leq C(n)\rho^2\|dw\|_{L^2},
                        \end{align}
                        for $\rho\in(0,\mz)$.
                        \begin{proof}
                            By a Taylor expansion, the left-hand side is bounded by $\frac{\rho^2}{2}\|D^2w\|_{L^{\infty}(B_{1/2}^n)}$,
                            which is bounded by $C(n)\rho^2\|dw\|_{L^2}$ by the mean-value property of harmonic functions.
                            \end{proof}
                    \end{lemma}
                    \subsection{Proof of the excess decay in the case of small $|d w(0)|$}
                    First, we prove \cref{tilt-excess-decay-statement} when the harmonic approximation has $|d w(0)|\le\delta$, for a small $\delta>0$ to be chosen later.
                    We dilate the ball $B_1^n$ to $B_{\sqrt2}^n$ (and replace $\epsilon$ with $\epsilon/\sqrt2$),
                    in such a way that it includes $B_1^2\times B_1^{n-2}$; we also assume that $S=\R^{n-2}$ in the statement.
                    
                    Let $c$ be the average of $h$ on the ball $B_{3/4}^{n-2}$. The construction of $h$ shows that
                    $$|c|\le C\sigma+Ce^{-K/\epsilon}\le C\sigma$$
                    for $\epsilon$ small enough (depending on $n,\sigma$).
                    
                    We apply the Caccioppoli-type estimates in \cref{Caccioppoli-prop-varifold}, with $x_1-c_1$ and $x_2-c_2$ in place of $x_1$ and $x_2$, see \cref{rmk:cacciopoli}. Taking $\phi\in C^{\infty}_c(B_{2\rho}^{n-2})$ to be a cut-off function with $\phi=1$ on $B_\rho^{n-2}$ and $|D^2\phi| \leq C(n)\rho^{-2}$ we get 
                        \begin{align*}
                            &\int_{B^{n-2}_{3/4}} \phi^2(z)(\E_1)_z\,dz \\
                            &\leq C\int_{B^2_{1/2}\times B^{n-2}_{3/4}} e_\epsilon(u,\nabla)[(x_1-c_1)^2+(x_2-c_2)^2]\Delta\phi^2 + C\rho^{-2}(\sigma^2\E_1 + e^{-{K}/{\epsilon}}).
                        \end{align*}
                        The contribution of the bad set $B_{3/4}^{n-2}\setminus\G^\eta$ can be bounded using the soft height bound of \cref{soft-height-bound} and energy estimate on the bad set (item (iii) in \cref{Lipschitz-approximation-thm}), obtaining
                        \begin{align*}
                        \int_{B^2_{1/2}\times (B^{n-2}_{3/4}\setminus\G^\eta)} e_\epsilon(u,\nabla)[(x_1-c_1)^2+(x_2-c_2)^2]\Delta\phi^2
                       \le C\rho^{-2}(\sigma^2+e^{-K/\epsilon})\frac{\E_1}{\eta^2}.
                        \end{align*}
                        
                        On the good set $\G^\eta$,
                        we apply \cref{variance-of-good-slices} to estimate the \textit{second moment} of good slices as follows:
                        \begin{align}\label{good-set-temp-identity-1}\begin{aligned}
                            &\left|\int_{B^2_{1/2}\times \G^{\eta}}e_\epsilon(u,\nabla)[(x_1-c_1)^2+(x_2-c_2)^2]\Delta\phi^2 - \int_{\G^\eta} \epsilon^2v_0\,\Delta \phi^2\right|\\
                            &\le C\rho^{-2}(\epsilon^2|\log\E_2|^2 \E_2^{1/2}+ \sigma^2 \E_1 +\int_{\G^\eta} |h-c|^2 (\E_2)_z^{1/2} +e^{-K\sigma/\epsilon})
                        \end{aligned}\end{align}
                        (see also \cref{concave}),
                        where $h$ is the Lipschitz approximation obtained in \cref{Lipschitz-approximation-thm}.
                        Note that the term containing $v_0$ disappears once integrated on $B_{2\rho}^{n-2}$, as $v_0$ is a constant and $\Delta\phi^2$ has zero integral.
                        
                        Combining the previous bounds, we arrive at
                        \begin{align*}
                        &\left|\int_{B^{n-2}_{3/4}} \phi^2(z)(\E_1)_z\,dz\right|\\
                        &\le C\rho^{-2}\int_{B_{2\rho}^{n-2}}|h-c|^2
                        +C\rho^{-2}\left[(\sigma^2+\epsilon^2)\frac{\E_1}{\eta^2} + \left(1+\frac{\E_1}{\eta^2}\right)e^{-{K}/{\epsilon}}
                        +\epsilon^2|\log\E_2|^2 \E_2^{1/2}\right].\end{align*}
                        Assuming $e^{-K/\epsilon}\le\E_1$,
                        we now apply \cref{harmonic-approx-prop} and \cref{decay-lemma-harmonic-functions}.
                        Since $\|(h-c)-\sqrt{\E_1}w\|_{L^2}^2\le\nu\E_1$, we have 
                        $$\int_{B_{2\rho}^{n-2}}|h-c|^2\le 2\nu\E_1+2\E_1\int_{B_{2\rho}^{n-2}}|w|^2
                        \le 2\nu\E_1+C\E_1(\rho^{4+(n-2)}+\delta^2\rho^{2+(n-2)}).$$
                        Thus, for some $C=C(n)$ and $K=K(n,\sigma)$, we get
                        \begin{align*}&\rho^{2-n}\int_{B_\rho^{n-2}}\E_1(z)\\
                        &\le C\E_1(\rho^{-n}\nu+\rho^2+\delta^2)\\
                        &\quad+C\rho^{-n}\left[(\sigma^2+\epsilon^2)\frac{\E_1}{\eta^2} +\left(1+\frac{\E_1}{\eta^2}\right)e^{-{K}/{\epsilon}}+\epsilon^2|\log\E_2|^2 \E_2^{1/2}\right].\end{align*}
                        We now choose $\eta,\rho$ and, \emph{subsequently}, $\delta,\sigma,\nu$ to be small enough.
                        The claim follows (with the same plane $\bar S=\R^{n-2}$) once we assume that $\E_1$ is small enough.

                    \subsection{Tilting the picture}\label{Tilt-subsection-label} 
                    In the general case, before applying the Caccioppoli-type estimate, we need to tilt the picture slightly to ensure that $|dw|(0)$ is small enough.
                    We assume that $\R^{n-2}$ minimizes $\E_1(u,\nabla,B_1^n,\cdot)$.
                    
                    Consider a rotation $R\in SO(n)$ bringing $\R^{n-2}$ to the graph of the linear map $\sqrt{\E_1}dw(0)$.
                    Since $w$ is harmonic with the bound $\|dw\|_{L^2} \leq C$, we have $|\sqrt{\E_1}dw(0)|\le C\sqrt{\E_1}$. Hence, we can find $R$ such that
                    \begin{align}\label{rotation-estimate}
                        \|R-I\| \leq C\E_1^{1/2}, \quad \|(P_{\R^{n-2}}\circ R-I)\circ P_{\R^{n-2}}\| \leq C\E_1,
                    \end{align}
                    for a dimensional constant $C=C(n)$:
                    indeed, calling $S$ the graph of $\sqrt{\E_1}dw(0)$, using the spectral theorem we can find an orthonormal basis $\{v_3,\dots,v_{n}\}$ of $\R^{n-2}$
                    such that the vectors $P_S(v_i)$ form an orthogonal basis of $S$, so that $\ang{P_S(v_i),v_j}=\ang{P_S(v_i),P_S(v_j)}=0$
                    for $i\neq j$. Thus, $P_{\R^{n-2}}\circ P_S(v_i)$ is parallel to $v_i$ and
                    $$\frac{P_S(v_i)}{|P_S(v_i)|}=\frac{(\sqrt{\E_1}dw(0)[v_i],v_i)}{\sqrt{1+\E_1|dw(0)[v_i]|^2}}$$
                    (under the identification $\R^{n-2}=\{0\}\times\R^{n-2}$).

                    We extend $\{v_3,\dots,v_n\}$ to an orthonormal basis $\{v_1,\dots,v_n\}$ of $\R^n$.
                    The desired rotation is
                    obtained by sending $v_i$ to $\frac{P_S(v_i)}{|P_S(v_i)|}$ for $i\ge3$ and $v_1,v_2$ to suitable unit vectors $v_1+O(\sqrt{\E_1})$ and $v_2+O(\sqrt{\E_1})$, obtained for instance via the Gram--Schmidt algorithm on the collection $\{\frac{P_S(v_3)}{|P_S(v_3)|},\dots,\frac{P_S(v_n)}{|P_S(v_n)|},v_1,v_2\}$.
                    For $i\ge3$, since $|P_{S^\perp}v_i|\le C\sqrt{\E_1}$    we have $|P_{S}v_i|\ge1-C\E_1$, and hence the previous formula gives
                    \begin{align}\label{R.explicit}
                        R(v_i)=R(0,v_i)=(\sqrt{\E_1}dw(0)[v_i],v_i)+O(\E_1)\quad\text{for }i\ge3.
                    \end{align}
                    
                    Then we define the rotated pair $(\tilde{u},\tilde{\nabla})$ as follows:
                    \begin{align}\label{rotated-pair}
                        \tilde{u} := R^*u, \quad \tilde\nabla := R^*\nabla.
                    \end{align}
                    First we prove that the excess changes proportionally after this rotation.
                    \begin{lemma}[Tilted excess estimate]
                        There exists a dimensional constant $C(n)$ such that, for a pair $(u,\nabla)$ as in \cref{tilt-excess-decay-statement} with small enough $\tau_0,\epsilon_0>0$ and a rotation $R$ as above, the tilted excess is bounded by the initial excess; more precisely,
                        \begin{align}\label{tilted-excess-estimate}
                        \E_1(\tilde{u},\tilde\nabla,B_1^n,\R^{n-2}) \leq C \E_1,\quad\E_2(\tilde{u},\tilde\nabla,B_1^n,\R^{n-2}) \leq C\E.
                        \end{align}
                        \begin{proof}
                            Take an orthonormal basis ${e_1,e_2,\dots,e_n}$ for $\R^n$ such that $\{e_3,\dots,e_n\}$ is an orthonormal basis for $\R^{n-2}$. Then, recalling the definition of the excess $\E_1$, we have
                            \begin{align*}
                                \E_1(\tilde{u},\tilde{\nabla},B_1^n,\R^{n-2}) &= \int_{B_1^n} \left[\sum_{k=3}^{n}|\nabla_{Re_k}u|^2 + \sum_{(j,k)\neq(1,2)} \epsilon^2\omega(Re_j,Re_k)^2\right]\\
                                &\leq C\E_1 + C\|R-I\|^2E_\epsilon(u,\nabla)\\
                                &\leq C\E_1.
                            \end{align*}
                            The second line above follows from the elementary bounds
                            $$|\nabla_{Re_k} - \nabla_{e_k} u | \leq \|R-I\||\nabla u|$$
                            and
                            $$|\omega(Re_j,Re_k) - \omega(e_j,e_k)| \leq 2\|R-I\||\omega|.$$
                            We estimate $\E_2$ in a similar way:
                            \begin{align*}
                                \E_2(\tilde{u},\tilde\nabla,B_1^n,\R^{n-2}) &= \int_{B^n_1} \left[|\nabla_{Re_1}u + i\nabla_{Re_2} u|^2 + \left|\epsilon\omega(Re_1,Re_2) - \frac{1-|u|^2}{2\epsilon}\right|^2\right] \\
                                &\leq C\E_2 + C\|R-I\|^2 E_\epsilon(u,\nabla)\\
                                &\leq C(\E_1 + \E_2).
                            \end{align*}
                            This is indeed the desired conclusion.
                        \end{proof}
                    \end{lemma}
                    Then we claim that the Lipschitz approximations $h$ and $\tilde h$ are approximately a rotation of one another. To do this, we first notice that the Lipschitz approximation $h_0$ of the zero set in \cref{Lipschitz-approximation-of-zero-loci-lemma} (applied with $\delta=\sigma$) behaves well under rotations: take $\tilde{h}_0$ to be the function whose graph is obtained by rotating of the graph of $h_0$ by $R^{-1}$ (cf.\ \cite[Section 8.2]{kelei-wang}).
                    For $z$ in the domain of $\tilde h_0$, there exists $z'\in B_{3/4}^{n-2}$ such that
                    $$(\tilde h_0(z),z) = R^{-1} ({h}_0(z'),z').$$
                    Since $\|(P_{\R^{n-2}}\circ R - I)\circ P_{\R^{n-2}}\| \leq C\E_1$ and $|h_0|\le\sigma$, we have $|z'-z|\le C\E_1+C\sqrt{\E_1}\sigma$.
                    Moreover, we have $\operatorname{Lip}(h_0)\le\sigma$, giving $|h_0(z')-h_0(z)|\le C\sqrt{\E_1}\sigma$.
                    Thus, assuming $\sqrt{\E_1}\le\sigma$,
                    $$(\tilde h_0(z),z) = R^{-1} ({h}_0(z),z)+O(\sqrt{\E_1}\sigma);$$
                    recalling \cref{R.explicit}, we see that $R(0,z)=(\sqrt{\E_1}dw(0)[z],z)+O(\E_1|z|)$, so that
                    \begin{align}\label{Lipschitz-tilt-estimate}
                        \tilde{h}_0(z) = h_0(z) - \sqrt{\E_1}d w(0)[z] + O(\sqrt{\E_1}\sigma),
                    \end{align}
                    with an implicit constant $C(n)$. Note that $\tilde h_0$ can be taken as a Lipschitz approximation of the zero set of the tilted pair: in order to have the conclusion of \cref{Lipschitz-approximation-of-zero-loci-lemma}, the only property that we care about is that its graph covers the zeros of $\tilde u$, except some exceptional ones projecting on a set of measure at most $C(n)\frac{\E_1}{\eta^2}$, and this holds for the rotated graph.
                    \subsection{Proof of the excess decay in the general case}
                    Now we can use \cref{Lipschitz-tilt-estimate} and the $L^2$ bound from \cref{Lipschitz-approximation-of-zero-loci-lemma} to conclude the proof of the tilt-excess decay theorem in the general case.
                    \begin{proof}[Proof of \cref{tilt-excess-decay-statement}]
                        Recall that $\R^{n-2}$ minimizes $\E_1(u,\nabla,B_1^n,\cdot)$.
                        Let $\tilde{h}$ be the Lipschitz approximation of the barycenter (built in \cref{Lipschitz-approximation-thm}) for the tilted pair $(\tilde{u},\tilde\nabla)$. We have
                        \begin{align*}
                            |\tilde{h}(z) - (h(z)-\sqrt{\E_1}dw(0)[z])| \leq |\tilde{h} - \tilde{h}_0| + |h - h_0| + |\tilde{h}_0 - (h_0-\sqrt{\E_1}d w(0)[z])|.
                        \end{align*}
                        We combine the main estimate of \cref{Lipschitz-approximation-of-zero-loci-lemma} and \cref{tilted-excess-estimate}--\cref{Lipschitz-tilt-estimate} to see that
                        \begin{align}\label{compare.with.rotation}
                            \int_{B^{n-2}_{1/2}} |\tilde{h}(z) - (h(z)-\E_1^{1/2}dw(0)[z])|^2 \\
                            &\leq C\left(\frac{\sigma^2}{\eta^2} + \sigma^2\right)\E_1 + C\epsilon^2|\log\E|^2\E+Ce^{-K/\epsilon}.
                        \end{align}
                        We assume in the sequel that
                        \begin{align*}
                            \epsilon^2|\log\E|^2\E,e^{-K/\epsilon}\le\sigma^2\E_1,
                        \end{align*}
                        so that
                        \begin{align*}
                            \int_{B^{n-2}_{1/2}} |\tilde{h}(z) - (h(z)-\E_1^{1/2}dw(0)[z])|^2 \leq C\left(\frac{\sigma^2}{\eta^2} + \sigma^2\right)\E_1.
                        \end{align*}
                        Now, taking the harmonic approximation for the tilted pair to be $\tilde{w}$, we can see that
                        \begin{align*}
                            &\int_{B^{n-2}_{1/2}} | \tilde\E_1^{1/2}\tilde{w}(z) - \E_1^{1/2}(w(z)-dw(0)[z])|^2\,dz \\
                            &\leq C\int_{B^{n-2}_{1/2}} [| h - \E_1^{1/2}w|^2 + |\tilde{h}-\tilde\E_1^{1/2}\tilde{w}|^2 + |\tilde{h}(z) - (h(z)-\E_1^{1/2}dw(0)[z])|^2]\\
                            &\le C\left(\nu+\frac{\sigma^2}{\eta^2} + \sigma^2\right)\E_1
                        \end{align*}
                        (the last line follows from $\tilde\E_1 \leq C\E_1$, as we saw in \cref{tilted-excess-estimate}).
                        Since $$\tilde\E_1^{1/2}\tilde{w}(z) - \E_1^{1/2}(w(z)-dw(0)[z])$$ is harmonic,
                        its differential at the origin
                        \begin{align*}
                            |\tilde\E_1^{1/2}d\tilde{w}(0)|^2\le C\left(\nu+\frac{\sigma^2}{\eta^2} + \sigma^2\right)\E_1.
                        \end{align*}
                        Since $\tilde\E_1\ge\E_1$,
                        this tells us that $|d\tilde{w}(0)|$ can be made arbitrarily small, reducing to the previous situation.
                    \end{proof}

                    \begin{rmk}\label{Z.instead.of.0}
                    In all the results obtained so far we were assuming that the center of the ball (or cylinder) belongs to the zero set, but actually they also hold if it belongs to the vorticity set $Z=\{|u|\le\frac34\}$, since this is enough to guarantee that it belongs to the support of the energy concentration measure in compactness arguments.
                    \end{rmk}
                    
                    \section{Iteration arguments and Morrey-type bounds}
                    \subsection{Proof of \cref{main-result-global-n4}: the case of critical pairs for $2\leq n \leq 4$}
                    We prove the following theorem, which is the first part of \cref{main-result-global-n4}.
                    \begin{thm}\label{iteration.pf}
                    For $2\leq n\leq4$, there exists $\tau_0(n)>0$ such that the following holds. If $(u,\nabla)$ is an entire critical point for the energy $E_1$, given by \cref{energy-definition} for $\epsilon=1$, with $u(0)=0$ and the energy bound
                    \begin{align}
                        &\lim_{R\to \infty} \frac{1}{|B^{n-2}_R|}\int_{B_R^n} \left[|\nabla u|^2 + |F_\nabla|^2 + \frac14(1-|u|^2)^2\right] \leq 2\pi+\tau_0,
                    \end{align}
                    then $(u,\nabla)$ is two-dimensional. More precisely,
                    we have $(u,\nabla) = P^*(u_0,\nabla_0)$ up to a change of gauge, where $P$ is the orthogonal projection onto a two-dimensional subspace and $(u_0,\nabla_0)$ is the standard degree-one solution of Taubes \cite{Taubes} (or its conjugate), centered at the origin.
                \end{thm}
                
                \begin{proof}
                    We can assume $n\in\{3,4\}$.
                	First, we claim that it is enough to show that
                	$$\lim_{R\rightarrow \infty}R^{2}\min_S\E_1(u,\nabla,B_R^n,S)=0.$$
                	Indeed, once this is done, we have
                	$$R^{4-n}\int_{B_R^n}\left[\sum_{a=3}^n |\nabla_{e_a^R}u|^2 + \sum_{(a,b)\neq(1,2)} \omega(e_a^R,e_b^R)^2\right]\to0$$
                	as $R\to\infty$, for a suitable choice of planes $S(R)$, where $\{e_1^R,\dots,e_n^R\}$
                	is an orthonormal basis such that $S(R)$ is spanned by $\{e_3^R,\dots,e_n^R\}$. Extracting a limit $S(R)\to S$ along a subsequence and assuming without loss of generality that $S=\R^{n-2}$, the fact that $n\le 4$ and Fatou's lemma give
                	$$\int_{\R^n}\left[\sum_{a=3}^n |\nabla_{e_a}u|^2 + \sum_{(a,b)\neq(1,2)} \omega(e_a,e_b)^2\right]=0.$$
                	As in the proof of \cref{excess-vanishes-when-excess-1-vanishes},
                	this implies that $(u,\nabla)$ depends only on the first two coordinates up to a change of gauge,
                	and the conclusion follows from the classification of planar solutions by Taubes \cite{Taubes}.
                	
                	We now turn to the previous claim.
                	By \cref{excess-vanishes}, we have $\frac{1}{|B_R^{n-2}|}\int_{B_R^n}e_\epsilon(u,\nabla)\to2\pi$
                	as $R\to\infty$, as well as
                	$$\E(u,\nabla,B_R^n,S(R))\to0$$
                	for suitable oriented planes $S(R)$, up to conjugating the pair. Arguing as in the proof of \cref{excess-vanishes},
                	we see that $S(R)$, viewed as an unoriented plane, has vanishing distance from any unoriented plane $S$ minimizing $\E_1(u,\nabla,B_R^n,S)$; hence, we can assume that $S(R)$ minimizes $\E_1$ on $B_R^n$.
                	
                	The proof now becomes an elementary iteration argument. In \cref{tilt-excess-decay-statement} we first fix \(\rho\in (0,1)\) such that \(C\rho^2 \le \rho\) and then \(\tau\) and \(\epsilon_0\) accordingly.
                    Let  $C'>\frac{1}{\epsilon_0}$.
                    Without loss of generality we can also assume that
                    $$\E_1(u,\nabla,B_R^n,S(R))>0,\quad \E(u,\nabla,B_R^n,S(R))\in(0,1)$$
                    are small enough to allow applying \cref{tilt-excess-decay-statement} on $B_R^n$
                    (by rescaling our pair),
                    for all $R\ge C'$. For every $k\in\N$ let us define
                    the minimum excess on each ball $B_{C'\rho^{-k}}$:
                    \begin{align*}
                        \E_1(k) := \E_1(u,\nabla,B_{C'\rho^{-k}},S(C'\rho^{-k})).
                    \end{align*}
                    Then \cref{tilt-excess-decay-statement} gives
                    \begin{align}
                        \begin{aligned}\label{decay-2}
                        \text{either }\E_1(k) &\leq {\rho} \bar\E_1(k+1)\\
                        \text{or }\E_1(k) &\leq \max\{\rho^{2k}|\log\E(k+1)|^2\sqrt{\E(k+1)},e^{-K\rho^{-2k}}\},
                        \end{aligned}
                    \end{align}
                    where $\E(k) := \E(u,\nabla,B_{C'\rho^{-k}},S(C'\rho^{-k}))$. By \cref{excess-vanishes}, we have
                    \begin{align}\label{decay-3}
                        \lim_{k\rightarrow\infty }\E_1(k) = 0,
                        \quad\lim_{k\rightarrow\infty}\E(k) = 0.
                    \end{align}
                    In order to iterate \cref{decay-2}, we define
                    \begin{align*}
                        f(k) := \log\E_1(k) + 2k \log\rho^{-1}
                    \end{align*}
                    and
                    \begin{align*}
                        \quad g(k) := \max\left\{2\log|\log\E(k+1)| + \frac12\log\E(k+1),-K\rho^{-2k} + 2k \log\rho^{-1}\right\}.
                    \end{align*}
                    Then \cref{decay-2} can be rewritten in terms of the functions $f,g:\mathbb{N}\rightarrow \R$ as
                    \begin{align}\label{decay-4}
                        f(k) \leq f(k+1) - \lambda\quad\text{or}\quad f(k) \leq g(k),
                    \end{align}
                    where $\lambda := 3\log\rho^{-1}$. Condition \cref{decay-3} also means that
                    \begin{align}\label{decay-5}
                        \lim_{k\rightarrow \infty} f(k) - 2k\log\rho^{-1} = -\infty,
                        \quad\lim_{k\rightarrow\infty} g(k) = -\infty.
                    \end{align}
                    We claim that if $f,g$ satisfy \cref{decay-4,decay-5} then
                    \begin{align*}
                        f(k) \leq \sup_{m\geq k} [g(m) - \lambda(m-k)].
                    \end{align*}
                    We prove this by contradiction: assume that there is some index $k_0$ such that
                    \begin{align}\label{decay-6}
                        f(k_0) + \lambda(m-k_0) > g(m)\quad\text{for all }m\geq k_0.
                    \end{align}
                    In particular we have $f(k_0) > g(k_0)$, so that \cref{decay-4} and \cref{decay-6} give
                    \begin{align*}
                        f(k_0+1) \geq f(k_0) + \lambda > g(k_0+1).
                    \end{align*}
                    By induction, we see that for all $m\geq k_0$
                    \begin{align*}
                        f(m) \geq f(k_0) +\lambda(m-k_0).
                    \end{align*}
                    Taking the limit $m\to\infty$ and noting that $\lambda > 2\log\rho^{-1}$, we obtain
                    \begin{align*}
                        f(k_0) \leq \lim_{m\rightarrow\infty} [f(m) - \lambda(m-k_0)] \leq \lim_{m\rightarrow\infty} [f(m) - 2m\log\rho^{-1}]
                        +2k_0\log\rho^{-1}=-\infty,
                    \end{align*}
                    where we used \cref{decay-5} in the last equality.
                    This is a contradiction, proving our claim.
                    
                    As a consequence, we have
                    \begin{align*}
                        f(k) \leq \sup_{m\geq k} [g(m) - \lambda (m-k)] \leq \sup_{m\geq k} g(m).
                    \end{align*}
                    Since $\lim_{k\rightarrow\infty} g(k) = -\infty$ by \cref{decay-5}, we deduce that
                    \begin{align*}
                        \lim_{k\rightarrow\infty} f(k) = -\infty.
                    \end{align*}
                    In other words, we have $\rho^{-2k}\E_1(k)\to0$, as desired.
                \end{proof}
                
            \subsection{Proof of \cref{coroll.intro} and \cref{Main-result-epsilon-neighbourhood-of-graph}}
            Given any $n\ge3$ and $(u,\nabla)$ as in \cref{Main-result-epsilon-neighbourhood-of-graph}, for any $\tau_0'>0$
            a standard compactness argument shows that
            $$\frac{1}{|B_r^{n-2}|}\int_{B_r(x)}e_\epsilon(u,\nabla)\le 2\pi+\tau_0'$$
            for all $x\in Z\cap B_{3/4}^n$ and $r=\frac18$, provided that $\tau_0$ and $\epsilon_0$ are taken small enough,
            and hence also for $r\le\frac17$ by energy monotonicity.
            
            This, together with \cref{excess-vanishes}, implies that,  for some oriented planes $S(x,r)$,
            $$\E(u,\nabla,B_r(x),S(x,r))\le\delta$$
            for some $\delta>0$ to be chosen momentarily and $C(n,\delta)\epsilon\le r\le\frac{1}{8}$.
            As in the previous proof, we can assume that $S(x,r)$ minimizes $\E_1$ on the ball $B_r(x)$.
            
            Given  $\alpha\in[0,1)$, we first fix \(\rho\) such that \(C\rho^2 \le \rho^{2\alpha}\) where \(C\) is the constant appearing in  \cref{tilt-excess-decay-statement}. We now choose   $\delta,\tau_0$ small such that \cref{tilt-excess-decay-statement} applies on each  ball $B_r(x)$ with $x\in Z\cap B_{3/4}^n$, compare with \cref{Z.instead.of.0}. We then consider 
            \begin{equation}\label{e:radii}
            \max\{M\epsilon, \epsilon^{1/(1+\alpha)}\}\le r\le\frac18
            \end{equation}
            where \(M\) chosen large enough to ensure that 
            \[
            e^{-Kr/\epsilon}\le\frac{\epsilon^2}{r^2}
            \]
            if \(\epsilon/r \le 1/M\). Applying the scaled version of  \cref{tilt-excess-decay-statement} (with \(\epsilon\) replaced by \(\epsilon/r\)), and noticing that  $\sup_{0<s\le\delta}|\log\delta|^2\delta^{1/2}\le1$, we finally obtain that either
            \[
            \E_1(u,\nabla,B_{\rho r}(x),S(x,\rho r))\le \rho^{2\alpha}\E_1(u,\nabla,B_{r}(x),S(x,r))
            \]
            or
            \[
            \E_1(x,r):=\E_1(u,\nabla,B_{r}(x),S(x,r))\le \frac{\epsilon^2}{r^2}\le r^{2\alpha}.
            \]
   This immediately implies
            \[
            \E_1(x,r)\le C(n,\alpha)r^{2\alpha}\quad\text{for all $x\in Z\cap B_{3/4}^n$ and any radii satisfying \cref{e:radii}.}
            \]           
            Moreover, if $S(x,r)$ is different from $S(x,r')$, for some $r'\in[r,2r]$,
            then we can find an orthonormal basis $\{e_1,\dots,e_n\}$
            such that $\{e_3,\dots,e_n\}$ spans $S(x,r)$ and $e_2$ belongs to the span of the two planes,
            with
            $$e_2=v+w,\quad v\in S(x,r)\text{ and }w\in S(x,r'),\quad |v|+|w|\le C(n)\|P_{S(x,r)}-P_{S(x,r')}\|^{-1},$$
            as the next simple lemma shows.
            
            \begin{lemma}
            Given two different planes $S,S'\in\Gr(n,k)$, there exists a unit vector $e\in (S+S')\cap S^\perp$
            such that $e=v+w$, with $v\in S$, $w\in S'$, and $|v|,|w|\le C(n)\|P_S-P_{S'}\|^{-1}$.
            \end{lemma}
            
            \begin{proof}
            We can assume that $S+S'=\R^n$ and $S\cap S'=\{0\}$ (otherwise we work on $(S\cap S')^\perp$), so that $n=2k$.
            We can also assume without loss of generality that $\|P_S-P_{S'}\|_{op}<c(n)$ for a constant $c(n)>0$ to be determined momentarily, since otherwise the
            statement follows from an immediate compactness argument (using the fact that, if $S_j\to S_\infty$ and $S_j'\to S_\infty'$, then each unit vector in $S_\infty+S_\infty'$ has vanishing distance from $S_j+S_j'$, even when the former sum has smaller dimension).
            
            It is elementary to check that the statement holds when $k=1$:
            in this case, calling $\theta\in(0,\frac{\pi}{2}]$ the angle between the lines $S$ and $S'$,
            we have $\|P_S-P_{S'}\|=\sqrt{2}\sin\theta$, and we can find vectors as in the statement with $|v|,|w|\le\frac{1}{\sin\theta}$.            
            
            Let $\tilde e$ be an eigenvector of $P_S-P_{S'}$, corresponding to an eigenvalue $\lambda$ with $0<|\lambda|=\|P_S-P_{S'}\|_{op}<c(n)$.
            Then
            $$P_S\tilde e-P_{S'}\tilde e=\lambda \tilde e,$$
            so that in particular $P_S\tilde e,P_{S'}\tilde e\neq0$ and
            $$P_SP_{S'}\tilde e=(1-\lambda)P_S\tilde e.$$
            Similarly we have
            $$P_{S'}P_S\tilde e=(1+\lambda)P_{S'}\tilde e.$$
            From the equation
            $$\ang{P_{S'}P_{S}\tilde e,P_{S'}\tilde e}=\ang{P_S\tilde e,P_{S'}\tilde e}=\ang{P_{S}\tilde e,P_{S}P_{S'}\tilde e}$$
            we deduce that
            $$|P_{S'}\tilde e|^2=\frac{1-\lambda}{1+\lambda}|P_S\tilde e|^2.$$
            In particular, calling $\theta\in(0,\frac{\pi}{2}]$ the angle between the vectors $P_S\tilde e$ and $P_{S'}\tilde e$, these identities easily give
            $$\sin^2\theta=1-\frac{\ang{P_S\tilde e,P_{S'}\tilde e}^2}{|P_S\tilde e|^2|P_{S'}\tilde e|^2}=\lambda^2.$$
            We now take
            $$Z:=\operatorname{span}\{P_S\tilde e,P_{S'}\tilde e\},$$
            which is a two-dimensional plane. By the case $k=1$, we can find
            $$e\in Z,\quad v\in \operatorname{span}\{P_S\tilde e\},\quad w\in \operatorname{span}\{P_{S'}\tilde e\}$$
            such that $e\perp P_S\tilde e$ is a unit vector and $|v|,|w|\le \lambda^{-1}$.
            In order to conclude, it suffices to check that $e\perp S$. Writing
            $$e=\alpha P_S\tilde e+\beta P_{S'}\tilde e,$$
            we have
            $$P_Se=\alpha P_S\tilde e+\beta P_SP_{S'}\tilde e=[\alpha+\beta(1-\lambda)]P_S\tilde e.$$
            Since $e\perp P_S\tilde e$, we have
            $$0=\ang{\alpha P_S\tilde e+\beta P_{S'}\tilde e,P_S\tilde e}
            =\alpha|P_S\tilde e|^2+\beta\ang{P_{S}P_{S'}\tilde e,P_S\tilde e}
            =[\alpha+\beta(1-\lambda)]|P_S\tilde e|^2.$$
            Since $P_S\tilde e\neq0$, we have $\alpha+\beta(1-\lambda)=0$, proving the claim.
            \end{proof}
            
            Since $\E(x,r)\le\delta$, we have
            $$r^{2-n}\int_{B_r(x)}e_\epsilon(u,\nabla)\le C(n)\delta+C(n)r^{2-n}\int_{B_r(x)}\left[\sum_{k=2}^n|\nabla_{e_k}u|^2+\sum_{(j,k)}\epsilon^2\omega(e_j,e_k)^2\right].$$
            Since the left-hand side is close to $2\pi$, and in particular larger than $\pi$
            (for $r\ge C\epsilon$), using the previous fact from linear algebra for the term \(\nabla_{e_2}u\) we obtain
            $$1\le C(n)[\E_1(x,r)+\E_1(x,r')](\|P_{S(x,r)}-P_{S(x,r')}\|^{-2}+1),$$
            and thus, since  \(\|P_{S(x,r)}-P_{S(x,r')}\|\le C(n)\),
            \begin{equation}\label{e:eppoi}
            \|P_{S(x,r)}-P_{S(x,r')}\|\le C(n)\sqrt{\E_1(x,r)+\E_1(x,r')}\le C(n,\alpha)r^\alpha.
            \end{equation}
            As a consequence, summing over dyadic scales, we have
            $$\|P_{S(x,r)}-P_{S(x,s)}\|\le C(n,\alpha)\max\{r,s\}^\alpha$$
            for $\max\{C(n,\alpha)\epsilon,\epsilon^{1/(1+\alpha)}\}\le r,s\le \frac18$.
                        
            A similar argument works varying the center: for two different points $x,x'\in Z\cap B_{3/4}^n$, looking at the balls $B_r(x)\subset B_{2r}(x')$ with $r:=|x-x'|$, we also have
            \begin{align}\label{holder.planes}\|P_{S(x,r)}-P_{S(x',r)}\|\le C(n,\alpha)r^\alpha,\end{align}
            provided that $r=|x-x'|\in[\max\{C(n,\alpha)\epsilon,\epsilon^{1/(1+\alpha)}\},\frac{1}{16}]$.
            
            Actually, the previous proof gives some extra information, which will be crucial in the sequel.
            We record it in the next proposition.
            
            \begin{proposition}\label{small.oscill.planes}
            Up to a rotation, we have $\|P_{S(x,r)}-P_{\R^{n-2}}\|\le\gamma$
            for any $\gamma>0$ fixed in advance (up to decreasing $\epsilon_0,\tau_0$),
            for all $x\in Z\cap B_{3/4}^n$ and $r\in[C(n,\gamma)\epsilon,\frac18]$.
            \end{proposition}
            
            \begin{proof}
            Let  $C$ be the constant in  the excess decay statement,  fix \(\rho\) such that \(C\rho^2\le \rho\) and fix \(\tau_0\) and \(\epsilon_0\) accordingly. Letting \(r_k:=\rho^k\), the first inequality of \cref{e:eppoi} gives that 
            $$\|P_{k+1}-P_k\|\le C\sqrt{\E_1(x,r_k)}.$$
            Iterating we get that 
            $$\|P_\ell-P_k\|\le C(n)\sum_{j=k}^{\ell-1}\sqrt{\E_1(x,r_j)}\le C(n)\sqrt{\E_1(x,r_k)}(1+\rho^{1/2}+\rho+\dots)
            \le C\sqrt{\E_1(x,r_k)}$$
            as long as $\E_1(x,r_j)>\frac{\epsilon^2}{r_j^2}$ for $j=0,\dots,\ell-1$ and  \(r_\ell\ge M\epsilon=:\bar r\) where  \(M\) is a large constant that we will fix at the end.
           Hence, if we  call  $r_{k_1}>\dots>r_{k_N}\ge\bar r$ the possible radii where $\E_1(x,r_{k_i})\le\frac{\epsilon^2}{r_{k_i}^2}$, we deduce that
           \begin{align*}
            \|P_{\ell}-P_0\|
            &\le C\max\{\sqrt{\E_1(x,r_0)},\sqrt{\E_1(x,r_{k_1})},\dots,\sqrt{\E_1(x,r_{k_N})}\}\\
           &\le C\left[\sqrt{\E_1(x,r_0)}+\frac{\epsilon}{\bar r}\right] \\
           &\le C\left[\sqrt{\E_1(x,r_0)}+\frac{1}{M}\right].
            \end{align*}
            Also, $P_0$ can be assumed arbitrarily close to $\R^{n-2}$ by a simple compactness argument
            (similar to the proof of \cref{excess-vanishes}). Since $\sqrt{\E_1(x,r_0)}$ and $1/M$
            can be taken arbitrarily small, the claim follows.
            \end{proof}
            
            The same proof gives the following.
            
            \begin{proposition}\label{ex.small.always}
            For any $x\in Z\cap B_{3/4}^n$ and $r\in[C(n)\epsilon,\frac{1}{8}]$, we have
            $$\E_1(u,\nabla,B_r(x),\R^{n-2})\le C(n)\E_1(u,\nabla,B_1(0),\R^{n-2})+C(n)\frac{\epsilon^2}{r^2}.$$
            \end{proposition}

            We now prove  \cref{coroll.intro}.
            \begin{proof}[Proof of \cref{coroll.intro}]
            We have already seen in \cref{excess-vanishes} that the energy on $B_R$ is asymptotic to $2\pi R^{n-2}$.
            We can then apply \cref{small.oscill.planes}: for any $\gamma>0$ we have
            $$\|P_{S(0,R)}-P_{S(0,R')}\|\le\gamma$$
            for $R<R'$ large enough
            (we use \cref{small.oscill.planes} after scaling the picture down by a factor $(R')^{-1}$).
            We deduce that the limit
            $$\lim_{R\to\infty}S(0,R)$$
            exists.
            \end{proof}
            
            
            \begin{proposition}\label{graph}
            Up to a rotation,
            the vorticity set $\tilde Z:=Z\cap[B_{1/2}^2\times B_{1/2}^{n-2}]$ is included in a $C(n,\gamma)\epsilon$-neighborhood of the graph of a $C^1$ map
            $$f:B_{1/2}^{n-2}\to B_\gamma^2$$
            with $\operatorname{Lip}(f)\le\gamma$, if we assume that $\tau_0$ and $\epsilon_0$ are small enough (depending on $n,\gamma$).
            \end{proposition}
            
            \begin{proof}
            Indeed, as seen in the proof of \cref{excess-vanishes}, for $\epsilon$ small enough
            we have $u(\cdot,z)\neq0$ on $\de B_{1/2}^2$, for all $z\in B_{1/2}^{n-2}$,
            and the degree of $(u/|u|)(\cdot,z)$ is $\pm1$ on this circle. Hence, each slice $B_{1/2}^2\times\{z\}$ intersects the zero set.
            
            Moreover, using \cref{soft-height-bound} on $B_1(0)$, we see that
            $\tilde Z\subseteq B_\gamma^2\times B_{1/2}^{n-2}$. Also, \cref{soft-height-bound}
            implies that  for all $x\in Z\cap B_{3/4}^n$ and $r\in[C(n,\gamma)\epsilon,\frac18]$
            we have
            \begin{align}\label{near.plane}
            Z\cap B_r(x)\subseteq B_{\gamma r}(x+S(x,r)),
            \end{align}
            where \(B_{\gamma r}(x+S(x,r))\) is the \(\gamma r\)-neighborhood of the affine plane \(x+S(x,r)\).
            We now take a collection of points $\{z_k\}\subset B_{1/2}^{n-2}$ with pairwise distance at least $C(n,\gamma)\epsilon$
            and $B_{1/2}^{n-2}\subseteq\bigcup_kB_{5C(n,\gamma)\epsilon}^{n-2}(z_k)$. For each $k$, we fix a point
            $x_k=(y_k,z_k)\in\tilde Z$. We then see that
            $$|y_k-y_j|\le C\gamma|x_k-x_j|,$$
            thanks to the previous observation applied with $r:=2|x_k-x_j|$ and the fact that $S(x,r)$ is $\gamma$-close to $\R^{n-2}$ (for $|x_k-x_j|>\frac{1}{16}$, this follows just from \cref{soft-height-bound}).
            Hence, the assignment $z_k\mapsto y_k$ defines a $C(n)\gamma$-Lipschitz function, which we can extend to a $C(n)\gamma$-Lipschitz function $f:B_{1/2}^{n-2}\to B_\gamma^2$. It is easy to check that (a regularization of) $f$ satisfies the desired conclusion, completing the proof.
            \end{proof}

We are now in position to prove \cref{Main-result-epsilon-neighbourhood-of-graph}.
\begin{proof}[Proof of \cref{Main-result-epsilon-neighbourhood-of-graph}]
            Let $\eta>0$ small such that \cref{Lipschitz-approximation-thm} applies.
            We first remark that the previous points $x_k$ can be taken such that $u(x_k)=0$
            and $z_k\in\mathcal{G}^\eta$.
            Indeed, by \cref{ex.small.always},
            we have
            $$\E_1(u,\nabla,B_r(x),\R^{n-2})\le c(n)\eta^2$$
            for all points $x\in Z\cap B_{3/4}^n$
            and radii $r\in[C(n)\epsilon,\frac18]$
            (by taking $\epsilon_0,\tau_0$ suitably small).
            We can apply this with $r:=M\epsilon$;
            by \cref{graph} and exponential decay of energy away from $Z$, we have
            $$r^{2-n}\int_{B_{r/2}^{n-2}(z)}(\E_1)_z
            \le c(n)\eta^2+e^{-K/M}\le 2c(n)\eta^2$$
            once we take $M=C(n)$ large enough, for any $z\in B_{1/2}^{n-2}$.
            Once we take $c(n)$ small enough, by the weak $L^1$ bound we can then find
            $$z'\in B_{r/2}^{n-2}(z)\cap\mathcal{G}^\eta$$
            (where we use slices of radius $\mz$ in the definition of $\mathcal{G}^\eta$), showing the claim.

            As a consequence of \cref{zero-set-distance-barycenter-slice}, we deduce that
            $$|h(z_k)-y_k|=|h(z_k)-h_0(z_k)|\le\epsilon.$$
            We immediately deduce that
            $\tilde Z$ is included in a $C(n)\epsilon$-neighborhood of the graph of $h$, which is the only consequence of the claim which we will use in the sequel.
            
            Now let $\bar\rho:=\max\{M\epsilon,\epsilon^{1/(1+\alpha)}\}$
            (with $M$ as in \cref{e:radii}) and
            consider \emph{another} finite collection of points $\{z_k\}\subset B_{1/2}^{n-2}$ such that the balls $B_{\bar\rho}^{n-2}(z_k)$ are disjoint and the dilated balls $B_{4{\bar\rho}}^{n-2}(z_k)$ cover $B_{1/2}^{n-2}$.
            Let $x_k=(y_k,z_k)$ be a point in $\tilde Z$ for each $k$.
            
            On $B_{10\bar\rho}^{n}(x_k)$ we consider the Lipschitz approximation $h_k$ built with respect to the rotated picture, obtained as a graph over $S_k:=x_k+S(x_k,10{\bar\rho})$. When viewed as a graph over $\R^{n-2}$, it becomes a function $\tilde h_k$ defined on the slightly smaller ball $B_{5{\bar\rho}}^{n-2}(z_k)$.
            
            By a scaled version of \cref{Lipschitz-approximation-thm}, we have
            $$\int_{B_{30{\bar\rho}/4}^n(x_k)\cap S_k}|dh_k|^2\le C{\bar\rho}^{n-2}\E_1(u,\nabla,B_{10{\bar\rho}}^n(x_k),S(x_k,10{\bar\rho}))\le C{\bar\rho}^{n-2+2\alpha}.$$
            In particular, by Poincaré,
            $$\int_{B_{30{\bar\rho}/4}^n(x_k)\cap S_k}|h_k-(h_k)|^2\le C{\bar\rho}^{n+2\alpha}.$$
            Now, as in \cref{Lipschitz-tilt-estimate}, we observe that
            $$|\tilde h_k(z)-h_k(z)-A_k(z)|\le C{\bar\rho}\sqrt{\E_1(u,\nabla,B_{10{\bar\rho}}^n(x_k),S(x_k,10{\bar\rho}))}\le C{\bar\rho}^{1+\alpha}$$
            for a suitable affine function $A_k$
            (where, with abuse of notation, $h_k(z)$ means $h_k$ composed with the isometry $\R^{n-2}\to S_k$). Combining these two bounds, we get
            $$\int_{B_{5{\bar\rho}}^{n-2}(z_k)}|\tilde h_k-A_k'|^2\le C{\bar\rho}^{n+2\alpha}$$
            for another affine function $A_k'$.

            We now take a partition of unity
            $\varphi_k$ subordinated to the cover $\{B_{4\bar\rho}^{n-2}(z_k)\}$
            and we let
            $$f:=\sum_k\varphi_k\tilde h_k.$$
            Since the zero set is within a $C\epsilon$-neighborhood of the graph of $\tilde h_k$
            (on the set $B_{1/2}^2\times B_{5{\bar\rho}}^{n-2}(z_k)$), we deduce that
            $$|\tilde h_k-\tilde h_{k'}|\le C\epsilon$$
            on $B_{5{\bar\rho}}^{n-2}(z_k)\cap B_{5{\bar\rho}}^{n-2}(z_{k'})$. Thus, we also have
            $$|A_k-A_{k'}|\le C\epsilon+C{\bar\rho}^{1+\alpha}$$
            whenever $B_{4{\bar\rho}}^{n-2}(z_k)\cap B_{4{\bar\rho}}^{n-2}(z_{k'})\neq\emptyset$.
            Since $\epsilon\le\bar\rho^{1+\alpha}$,
            this allows us to conclude that
            $$\int_{B_{2{\bar\rho}}^{n-2}(z)}|f-A_z|^2\le C{\bar\rho}^{n+2\alpha}$$
            for any $z\in B_{1/2}^{n-2}$, for a suitable affine function $A_z$ depending on $z$.
            
			Thus, taking a standard mollifier $\chi_{\bar \rho}$,
			setting
			$$g:=\chi_{\bar \rho}*f$$
			and using the previous bound, we deduce that
			$$|dg-dA_z|\le C\bar\rho^{\alpha}\quad\text{on }B_{\bar\rho}^{n-2}(z),$$
            and in fact
            $$[dg]_{C^{0,\alpha}(B_{\bar\rho}^{n-2}(z))}\le C.$$
            
			Finally, recalling that $dA_k$ is the slope of the plane $P_{S(x_k,10\bar\rho)}$, we also have
            $$|dA_k-dA_{k'}|\le C|z_k-z_{k'}|^\alpha$$
            by the H\"older continuity \cref{holder.planes}, while
            $$|dA_z-dA_k|\le C\bar\rho^{\alpha}\quad\text{for $z\in B_{4\bar\rho}^{n-2}(z_k)$.}$$
            From these bounds, we easily deduce that
            $$|dg(z)-dg(z')|\le C|z-z'|^\alpha\quad\text{for $|z-z'|\ge\bar\rho$},$$
            completing the proof of the $C^{1,\alpha}$ regularity of $g$. Since $|g-f|\le C\bar\rho$,
            it follows that the vorticity set is included in a $C\bar\rho$-neighborhood of the graph of $g$.
			
			It is clear from the proof that we can actually make $[dg]_{C^{0,\alpha}}$ arbitrarily small,
			up to decreasing $\tau_0$ and $\epsilon_0$.
			\end{proof}

            \section{Constructing competitors for local minimizers: a good gauge}
            In this section we prepare the ground to construct competitors for minimizers and to show that the full excess decays as long as it is above $\epsilon^\beta$, for any $\beta>0$, giving a proof of \cref{tilt-excess-decay-minimizer}. To investigate minimizers, we construct competitors with the same boundary conditions and compare the energies to show that the excess $\E$ is effectively approximated by the Dirichlet energy of the harmonic approximation.
            
            To do this, we need to construct competitors modeled on graphs in the interior and then glue them to the boundary condition, while controlling the error terms. We pullback the $\epsilon$-rescaled degree-one solution along the graph of the Lipschitz approximation, as obtained in \cref{Lipschitz-approximation-thm}. Then we gauge fix in balls of size $\epsilon|\ln\epsilon|$ and, using the estimates at that scale, we define a global gauge by a partition of unity. In this gauge we can interpolate between the intial pair and the new one with good estimates.
            \subsection{The pullback pair}
            Here we introduce the pullback pair $(u_f,\nabla_f)$ whose zero set is prescribed to be the graph of a Lipschitz function $f:B^{n-2}\rightarrow B^2_1$. We prove that the excess of these pairs are well approximated by the Dirichlet energy of $f$, as in the following proposition.
            \begin{proposition}[Constructing the pullback pair]\label{pull-back-proposition}
                There exist $\eta_0(n),\epsilon_0(n)>0$ small enough with the following property. Given any $\epsilon\leq \epsilon_0$ and a Lipschitz function $f: B^{n-2}_1\rightarrow B^2_{1/2}$ with $Lip(f)=\eta\leq\eta_0$, there exists a pair $(u_f,\nabla_f)$ obeying the following estimate:
                \begin{align*}
                    \frac{1}{2\pi}\int_{B^2_1\times B_1^{n-2}} e_\epsilon(u_f,\nabla_f) = |B^{n-2}_1| + (1+O(\eta^2))\int_{B^{n-2}_1} \frac{|df|^2}{2} + O(e^{-{K}/{\epsilon}}),
                \end{align*}
                with implicit constants $C(n)$.
                Moreover, we have that
                \begin{align*}
                    u_f^{-1}(0) = \operatorname{graph}(f).
                \end{align*}
                \begin{proof}
                    To construct $(u_f,\nabla_f)$ we pull back the planar degree-one solution of Taubes \cite{Taubes}, via the map $Q_\epsilon:B^2_1\times B^{n-2}_1\rightarrow \R^2$ given by
                    \begin{align*}
                        Q_\epsilon(x) = \frac{(x_1,x_2) - f(x_3,\dots,x_n)}{\epsilon}.
                    \end{align*}
                    Then we define $(u_f,\nabla_f)$ by
                    \begin{align}\label{pullback-definition}
                        (u_f,\nabla_f) := Q_\epsilon^*(u_0,\nabla_0),
                    \end{align}
                    where $(u_0,\nabla_0)$ is the degree-one solution in \cite{Taubes} with $u_0(0)=0$ (unique up to change of gauge). First, we note 
                    that, since $dQ_\epsilon(x)[e_k]=-\de_{e_k}f_1(x_3,\dots,x_n)e_1-\de_{e_k}f_2(x_3,\dots,x_n)e_2$, we have
                    the following identities for $k=3,\dots,n$:
                    \begin{align}\label{pull-back-identity-1}
                    \begin{aligned}
                        |(\nabla_f)_k u_f|^2(x) &= \epsilon^{-2}|\de_{e_k} f_1 (\nabla_0)_{e_1}u_0 + \de_{e_k} f_2 (\nabla_0)_{e_2}u_0|^2\left(Q_\epsilon(x)\right)\\
                        &= \epsilon^{-2}\frac{|\de_k f|^2}{2} |\nabla_0 u_0|^2\left(Q_\epsilon(x)\right),
                    \end{aligned}
                    \end{align}
                    where we omitted the argument of $f$ and we used the fact that $(\nabla_0)_{e_2}u_0 =i(\nabla_0)_{e_1}u_0$ for solutions of the vortex equations \cref{vortex-equations}. We also have
                    \begin{align}\label{pull-back-identity-2}
                        |(\nabla_f)_{e_1} u_f|^2(x) + |(\nabla_f)_{e_2} u_f|^2(x) = \epsilon^{-2}|\nabla_0u_0|^2\left(Q_\epsilon(x)\right).
                    \end{align}
                    We also compute for the curvature term $-i\omega_f:=F_{\nabla_f} = F_{Q_\epsilon^*(\nabla_0)} = Q_\epsilon^*(F_{\nabla_0})$ that
                    \begin{align}\label{pull-back-identity-3}
                        \sum_{j=1,2}\epsilon^2\omega_f(e_j,e_k)^2(x) = \epsilon^{-2}\omega_0(e_1,e_2)^2\left(Q_\epsilon(x)\right)|\de_{e_j} f|^2 \quad\text{for $j=1,2$ and $k=3,\dots,n$},
                    \end{align}
                    and moreover
                    \begin{align}
                        \epsilon^2\omega_f(e_1,e_2)^2(x) &= \epsilon^{-2}\omega_0^2(e_1,e_2)\left(Q_\epsilon(x)\right),\label{pull-back-identity-4}
                    \end{align}
                    as well as
                    \begin{align}\label{pull-back-identity-5}
                    \epsilon^2\omega_f(e_k,e_\ell)^2(x) &\leq \epsilon^{-2}|df|^4 \omega_0(e_1,e_2)^2\left(Q_\epsilon(x)\right)\quad\text{for $3\leq k<\ell\le n$}.
                    \end{align}
                    To compute $E_\epsilon(u_f,\nabla_f)$, we use \cref{pull-back-identity-1}--\cref{pull-back-identity-5} to see that
                    \begin{align*}
                        \int_{B^2_1\times B^{n-2}_1} e_\epsilon(u,\nabla)
                        &=\int_{B^{n-2}_1} \left[\int_{B^2_1} \epsilon^{-2}e_1(u_0,\nabla_0)(Q_\epsilon(x))\left(1+\frac{|df|^2}{2} + O(|df|^4)\right) + O(e^{-{K}/{\epsilon}})\right]\\
                        &= 2\pi \left[ |B^{n-2}_1|  + \int_{B_1^{n-2}} \frac{|df|^2}{2} + O(|df|^4) \right] + O(e^{-{K}/{\epsilon}}).
                    \end{align*}
                    In the above display we used the exponential decay from \cite[Chapter III, Theorem 8.1]{Taubes-2}:
                    \begin{align*}
                        \int_{\R^2\setminus B_{1/(2\epsilon)}^2} e_1(u_0,\nabla_0) = O(e^{-{K}/{\epsilon}}).
                    \end{align*}
                    Recalling that $|df| \leq \eta$ we get the desired estimate.
                \end{proof}
            \end{proposition}
            
            \subsection{Constructing the interpolation gauge} In this section we find a gauge transformation $(u,\nabla)\mapsto (e^{i\xi}u,\nabla -id\xi)$ for which the new pair is $L^2$-close to the pullback pair $(u_h,\nabla_h)$ constructed in \cref{pull-back-proposition}, where $h:B^{n-2}_1\rightarrow B^2_{1/2}$ is the Lipschitz approximation built in \cref{Lipschitz-approximation-thm}. Since this is the most technical part of the paper, we provide an overview of the arguments. 
            
            \fbox{\textit{Step 1.}} We cover the vortex set with cylinders $\{B_{5C|\epsilon\ln\epsilon|}^2(y_k)\times B^{n-2}_{5C\epsilon}(x_k)\}_{k=1}^N$ with $x_k = (y_k,z_k)\in \R^{2}\times \R^{n-2}$ such that $B^{n-2}_{5C\epsilon}(z_k)$ is a Vitali cover of $B^{n-2}_1$. Then we name a cylinder \textit{good} if the excess on it is small, and \textit{bad} otherwise. We also define a partition of unity, subordinate to this cover, with derivatives at most $C\epsilon^{-1}$.
            
            \fbox{\textit{Step 2.}} In each cylinder we pass to the Coulomb gauge, via a function $\xi_k$ with mean equal to the mean of $\theta_h-\theta$ on an appropriate annulus away from the vortex set (note that $\theta-\theta_h$ is well-defined far from the vortex set of both $u$ and $u_h$) Then we use Gaffney- and Poincaré-type inequalities from \cref{poincare-section} to derive estimates for $e^{i\xi_k}u-u_h$ and $(\alpha+d\xi_k)-\alpha_h$, where we write $\nabla=d-i\alpha$. Far from the vorticity set we modify the gauge so that $e^{i\xi_k}u$ and $u_h$ have the same phase. Then we use the exponential decay away from the vortex set, which is where the error $\epsilon^\beta$ comes from; however, this will be enough to show the classification result in all dimensions, since we are free to take $\beta$ arbitrarily large.

            \fbox{\textit{Step 3.}} We use the estimates on $(\alpha+d\xi_k)-\alpha_h$ (and the mean condition) and Poincaré--Gaffney-type inequalities from \cref{poincare-section} to bound $\xi_j - \xi_k$ on overlapping cylinders.
            
            \fbox{\textit{Step 4.}} We patch together the $\xi_k$'s with the partition of unity defined in the first step to obtain the function $\xi$. Then we use the bounds on $\xi_j-\xi_k$ to derive estimates on $(\alpha+d\xi)-\alpha_h$ and $e^{i\xi}u-u_h$.

            \begin{proposition}[The interpolation gauge]\label{Interpolation-gauge-proposition} For any $\beta>0$ there exist $\tau_0(n,\beta),\epsilon_0(n,\beta)>0$ and $C_0(n,\beta)>0$ with the following property. Let $(u,\nabla)$ be a critical pair for $E_\epsilon$ on $\R^n$, with $u(0)=0$, $\epsilon\leq\epsilon_0$, and the energy bound
            $$\frac{1}{|B^{n-2}_2|}\int_{B_2^n}e_\epsilon(u,\nabla)\le2\pi+\tau_0.$$
            Moreover, let $h:B_1^{n-2}\rightarrow B^2_{1/2}$ be the Lipschitz approximation defined in \cref{Lipschitz-approximation-thm} (for a suitable $\eta$ chosen later on) and let $(u_h,\nabla_h)$ be the pullback pair constructed in \cref{pull-back-proposition}. Then, on a given annulus $$\mathcal{A}_{s,\delta} := B^2_1\times (B^{n-2}_{s+\delta}\setminus\bar B^{n-2}_{s})$$ with $\delta \in [C_0\epsilon,\frac1{16}]$ and $s\leq \frac34$, we can find a gauge transformation
            \begin{align*}
                (u,\nabla) \mapsto (e^{i\xi}u, \nabla-id\xi),
            \end{align*}
                via a smooth function $\xi:\mathcal{A}_{s,\delta}\rightarrow \R$ such that:
            \begin{enumerate}[label=(\roman*)]
            \item letting $P:\R^{n}\rightarrow\R^{n-2}$ be the projection onto the last $n-2$ coordinates, we have
            \begin{align*}
                &\int_{\mathcal{A}_{s,\delta}} [\epsilon^{-2}|e^{i\xi}u-u_h|^2 + |(\alpha+d\xi)-\alpha_h|^2]\\
                & \leq C(n,\beta)|\ln\epsilon|^{8}\int_{P(\mathcal{A}_{s-\delta,3\delta})} [\E_z+\uno_{\mathcal{G}^\eta}\E_z|\log\E_z|^2]+ \epsilon^\beta;
            \end{align*}
            \item letting
            \begin{align*}
                Z := \{|u|\leq 3/4 \},\quad Z_{C_0\epsilon|\ln\epsilon|} := \{x=(y,z) \,:\, \operatorname{dist}(x,Z\cap(B_1^2\times\{z\})) \leq C_0\epsilon|\ln\epsilon|\},
            \end{align*}
            the function $e^{i\xi}u\to\C\setminus\{0\}$ has the same phase as $u_h$ far from the vortex set, i.e.,
            \begin{align*}
                \frac{e^{i\xi}u}{u_h} \in \R^+\quad \text{on } \mathcal{A}_{s,\delta} \setminus Z_{5C_0|\epsilon\ln\epsilon|}.
            \end{align*}
            \end{enumerate}
            \begin{proof}
            First we choose $\epsilon_0,\tau_0$ small enough so that \cref{Main-result-epsilon-neighbourhood-of-graph} applies (for $\alpha=0$). We divide the proof into several steps.
                
                \fbox{\textit{Covering arguments and a partition of unity.}} To begin with,
                by \cref{Main-result-epsilon-neighbourhood-of-graph}, we can find a collection of points $$\{x_k=(y_k,z_k)\}_{k=1}^N \subset Z\cap\mathcal{A}_{s,\delta}$$
                satisfying the following.
                \begin{enumerate}[label=(\roman*)]
                    \item The projected collection $\{z_k=P(x_k)\}_{k=1}^N\subset B^{n-2}_{s+\delta}\backslash B^{n-2}_{s}$ gives a Vitali covering of the projected annulus $P(\mathcal{A}_{s,\delta})$:
                    \begin{align}\label{covering-1}
                    \begin{aligned}
                        &P(\mathcal{A}_{s,\delta})= B^{n-2}_{s+\delta}\setminus\bar B^{n-2}_{s} \subseteq \bigcup_{k=1}^{N}B^{n-2}_{5C_0\epsilon}(z_k),\\
                        &B_{C_0\epsilon}^{n-2}(z_j)\cap B_{C_0\epsilon}^{n-2}(z_k) = \emptyset\quad\text{for all }j\neq k.
                    \end{aligned}
                    \end{align}
                    This last line shows in particular that $N\leq C(n)\epsilon^{2-n}$, where by now $C_0$ depends only on $n$.
                    \item As a direct consequence of \cref{Main-result-epsilon-neighbourhood-of-graph}, we can guarantee that
                    \begin{align}\label{covering-2}
                        Z_{C_0\epsilon|\ln\epsilon|}\cap \mathcal{A}_{s,\delta} \subseteq \bigcup_{k=1}^N B_{5C_0\epsilon|\ln\epsilon|}^2(y_k)\times B^{n-2}_{5C_0\epsilon}(z_k).
                    \end{align}
                    \item We say that a point $x_k=(y_k,z_k)$ is a \textit{good} point if
                    \begin{align*}
                        \sup_{\epsilon \le r \le 10C_0\epsilon}  r^{2-n}\E_1(u,\nabla,B^2_1\times B_r^{n-2}(z_k),\R^{n-2}) \leq \eta_0^2,
                    \end{align*}
                    where with a certain abuse of notation we have set 
                    \begin{align}\label{e:new-excess}
                    \E_1(u,\nabla,B^2_1\times B_r^{n-2}(z),\R^{n-2})=
                    \int_{B^2_1\times B_r^{n-2}(z)} \left[\sum_{k=3}^n |\nabla_{e_k}u|^2 + \epsilon^2\sum_{(j,k)\neq(1,2)} \omega(e_j,e_k)^2\right]  
                    \end{align}
                    (note the absence of normalization)
                    Let the set of \textit{good indices} be $G$. We also denote the set of bad ones by $B := \{1,\dots,N\}\setminus G$.
                    \item Again as a direct consequence of \cref{Main-result-epsilon-neighbourhood-of-graph} we get that
                    \begin{align}\label{covering-3}
                        |u_h|,|u|>\frac34\quad \text{on } \bigcup_{k=1}^N( B^{2}_{5C_0\epsilon|\ln\epsilon|}(y_k)\setminus B^2_{C_0\epsilon|\ln\epsilon|}(y_k))\times B^{n-2}_{5C_0\epsilon}(z_k).
                    \end{align}
                    Since both $u$ and $u_h$ have degree $1$ on each of the previous domains (up to conjugating $(u,\nabla)$ on $\R^n$),
                    the \emph{difference of phases} $\theta-\theta_h$ is well-defined on these domains.
                    \item We also define a partition of unity $\{\phi_k\}_{k=1}^N$, subordinate to the cylinders:
                    \begin{align}\label{partition-of-unity-1}
                        \phi_k\in C^{1}_c( B^{2}_{5C_0\epsilon|\ln\epsilon|}(y_k)\times B^{n-2}_{5C_0\epsilon}(z_k)),\quad 0\leq \phi_k \leq 1,
                    \end{align}
                    and
                    \begin{align*}
                        \sum_{k=1}^{N} \phi_k = 1 \quad \text{on }  Z_{C_0\epsilon|\ln\epsilon|} \cap \mathcal{A}_{s,\delta}.
                    \end{align*}
                    We also require that $|d\phi_k| \leq \tilde C\epsilon^{-1}$ for all $k=1,\dots,N$.
                    \item Up to modifying the $\phi_k$'s, we can define $\phi_0\in C^{1}_c(\mathcal{A}_{s,\delta}\setminus Z_{C_0|\epsilon\ln\epsilon|})$ with $0 \leq \phi_0 \leq 1$ and
                    \begin{align*}
                        &\phi_0=1\quad\text{on } \mathcal{A}_{s,\delta}\setminus Z_{5C_0|\epsilon\ln\epsilon|}, \\
                        &\phi_0 + \sum_{k=1}^{N}\phi_k = 1 \quad \text{on } \mathcal{A}_{s,\delta}.
                    \end{align*}
                    Lastly, $|d\phi_0| \leq \tilde C|\epsilon\ln\epsilon|^{-1}$.
                \end{enumerate}
                Note that $C_0>0$ is some large enough constant that we are still free to choose later on ($\tilde C$ above depends on $C_0$). In the sequel we also use the following notation for the excess on each ball:
                \begin{align*}
                    \E(k) := \E(u,\nabla, B^2_1 \times B^{n-2}_{10C_0\epsilon}(z_k),\R^{n-2}).
                \end{align*}
                We define $\E(k) = \E_1(k) + \E_2(k)$ similarly. Since the balls $B^{n-2}_{C_0\epsilon}(z_k)$ are disjoint,
                the dilated balls have bounded overlap (i.e., at most $C(n)$ of them intersect), giving
                \begin{align*}
                    \sum_{k=1}^{N}\E(k) \leq C\int_{P(\mathcal{A}_{s-\delta,3\delta})} \E_z,
                \end{align*}
                and the same holds for $\E_1$ and $\E_2$. With abuse of notation, we also write
                \begin{equation}\label{e:convention}
                    |\log\E|^2\E(k) := \begin{cases}\int_{B^{n-2}_{5C_0\epsilon}(z_k)} |\log\E_z|^2\E_z&\text{if }k\in G,\\
                    \int_{B^{n-2}_{5C_0\epsilon}(z_k)} \E_z&\text{if }k\in B.\end{cases}
                \end{equation}
                
                \begin{rmk}
                We will often use the following observation implicitly.
                For all $z\in B_1^{n-2}$, the results in the previous section show that
                $$r^{2-n}\E_1(u,\nabla,B_1^2\times B_{r}^{n-2}(z),\R^{n-2})\le \tilde\eta_0^2$$
                for an (arbitrarily) small $\tilde\eta_0$ and all radii $\Lambda\epsilon\le r\le\mz$, for some $\Lambda$
                depending on $n,\tilde\eta_0$.
                If $k\in G$ then, for all $z\in B_{5C_0\epsilon}^{n-2}(z_k)$,
                we have
                $$(C_{0}\epsilon)^{{2-n}}\E_1(u,\nabla,B_1^2\times B_{C_0\epsilon}^{n-2}(z),\R^{n-2})\le 10^{n-2}\eta_0^2.$$
                As a consequence of \cref{cone}, if we take $C_0\ge C(n)$ large and $\eta_0^2$ very small,
                we have
                $$r^{2-n}\E_1(u,\nabla,B_1^2\times B_{r}^{n-2}(z),\R^{n-2})\le \tilde\eta_0^2$$
                also for $r\in(0,\Lambda\epsilon)$, since we have $C^1$ control on the pair at this scale.                
                We now fix $\tilde\eta_0$ such that we actually apply \cref{Lipschitz-approximation-thm} for $\eta:=\tilde\eta_0$,
                so that \emph{every $z\in  B_{5C_0\epsilon}^{n-2}(z_k)$ gives a good slice}.
                \end{rmk}
                
                \fbox{\textit{Gauge fixing on each small cylinder with bounds.}} Fix $k\in\{1,\dots,n\}$ and consider 
                the unique solution $\xi_k$ to the following Neumann boundary value problem:
                    \begin{align}\label{gauge-transformation-definition}
                    \begin{cases}
                        \Delta \xi_k = d^*(\alpha-\alpha_h) & \text{in } \mathcal{C}_k:=B_{5C_0\epsilon|\ln\epsilon|}^2(y_k)\times B^{n-2}_{5C_{0}\epsilon}(z_k),\\
                        \de_\nu\xi_k = -(\alpha-\alpha_h)(\nu) & \text{at } \de\mathcal{C}_k,\\
                        \int_{\mathcal{A}_k} [(\theta + \xi_k)-\theta_h]=0&\text{where }\mathcal{A}_k:=(B^2_{5C_0\epsilon|\ln\epsilon|}(y_k)\setminus B^2_{C_0\epsilon|\ln\epsilon|}(y_k))\times B^{n-2}_{5C_0\epsilon}(z_k).
                    \end{cases}
                    \end{align}
                    
                    Recall that $\theta-\theta_h$ is well-defined on $\mathcal{A}_k$. Then we perform the following gauge transformation in $\mathcal{C}_k$:
                    writing $\nabla=d-i\alpha$, we transform
                    \begin{align*}
                        (u,\alpha) \mapsto (e^{i\xi_k}u , \alpha + d\xi_k).
                    \end{align*}
                    Since $d^*[(\alpha+d\xi_k)-\alpha_h] = 0$ and $[(\alpha+d\xi_k)-\alpha_h](\nu)=0$, we can use the Gaffney--Poincaré-type inequality in \cref{gaffney-poincare-thin-cylinder-lemma}:
                    \begin{align}\label{Gaffney-in-epsilon-ball}
                        \int_{\mathcal{C}_k} |(\alpha+d\xi_k) - \alpha_h|^2 \leq C(n,C_0)|\epsilon\ln\epsilon|^2 \int_{\mathcal{C}_k} |d\alpha - d\alpha_h|^2.
                    \end{align}
                    We bound separately the contributions of the good set and the bad set, using \cref{min.good.set} and \cref{min.bad.set} below,
                    obtaining
                    \begin{align}\label{gauge-estimate-on-balls}
                    \int_{\mathcal{C}_k} |(\alpha+d\xi_k) - \alpha_h|^2 \leq 
                    C|\ln\epsilon|^2|\log\E|^4\E{(k)} + C\epsilon^{\beta+3n}
                    \end{align}
                    for some $C=C(n,C_0,\beta)=C(n,\beta)$.

                    \fbox{\textit{Gauge fixing far from the vortex set with bounds.}} Far from the vortex set, in the set $\mathcal{A}_{s,\delta}\setminus Z_{C_0|\epsilon\ln\epsilon|}$, we gauge fix via a function $\xi_0$ such that $e^{i\xi_0}u/u_h$ is real-valued. Hence, we define
                    \begin{align}\label{far-away-gauge-definition}
                        \xi_0 := \theta_h-\theta\quad\text{on } \mathcal{A}_{s,\delta}\setminus Z_{C_0|\epsilon\ln\epsilon|};
                    \end{align}
                    note that a priori $\xi_0$ is well-defined only in the quotient $\R/2\pi\Z$, but this is enough to have a well-defined gauge transformation (in fact, since the vorticity set is included in a $C\epsilon$-neighborhood of a graph, we can check that $\theta_h-\theta$ is a well-defined real number).
                    
                    We can estimate $e^{i\xi_0}u-u_h$ and $(\alpha+d\xi_0)-\alpha_h$ in this domain using the exponential decay (\cref{exponential-decay-proposition}), as follows:
                    \begin{align*}
                        &\int_{\mathcal{A}_{s,\delta} \setminus Z_{C_0|\epsilon\ln\epsilon|}} [\epsilon^{-2}|e^{i\xi_0}u - u_h|^2 + |(\alpha+d\xi_0)-\alpha_h|^2] \\
                        &\le \int_{\mathcal{A}_{s,\delta} \setminus Z_{C_0|\epsilon\ln\epsilon|}} [\epsilon^{-2}||u|-|u_h||^2 + 2|\alpha-d\theta|^2 + 2|\alpha_h-d\theta_h|^2]\\
                        &\leq C(n)\epsilon^{-2} e^{-K(n)C_0|\ln\epsilon|}.
                    \end{align*}
                    In the last inequality, we used the following observation: since each slice intersects the zero set and $|\ln\epsilon|\ge C(n)$,
                    using \cref{cone} we see that on $B_1^2\times\{z\}$ the distance from $\{|u|\le3/4\}$ is comparable
                    with the distance from  $(B_1^2\times\{z\})\cap\{|u|\le3/4\}$.
                    
                    Taking $C_0$ large enough, we see that
                    \begin{align}\label{gauge-away-from-vrotex-set-estimate-1}
                        \int_{\mathcal{A}_{s,\delta} \setminus Z_{C_0|\epsilon\ln\epsilon|}} [\epsilon^{-2}|e^{i\xi_0}u - u_h|^2 + |(\alpha+d\xi_0)-\alpha_h|^2] \leq \epsilon^{\beta+3n}.
                    \end{align}
                    
                    \fbox{\textit{Difference of local gauges in the overlap regions.}} Fix $1\leq j<k \leq N$ such that
                    \begin{align*}
                        \Omega_{j,k}:=\mathcal{C}_j\cap\mathcal{C}_k \neq \emptyset.
                    \end{align*}
                    Notice that we can bound the $L^2$ norm of the difference $d\xi_k-d\xi_j$ as follows:
                    \begin{align*}
                        \int_{\Omega_{j,k}} |d\xi_j-d\xi_k|^2 
                        \le 2\int_{\mathcal{C}_j}|(\alpha+d\xi_j)-\alpha_h|^2+2\int_{\mathcal{C}_k}|(\alpha+d\xi_k)-\alpha_h|^2.
                    \end{align*}
                    By \cref{gauge-estimate-on-balls} we then have
                    \begin{align}\label{pre-poincare-overlap-estimate-1}
                        \int_{\Omega_{j,k}} |d(\xi_j-\xi_k)|^2 \leq C|\ln\epsilon|^4(|\log\E|^2\E{(j)} + |\log\E|^2\E{(k)}) + C\epsilon^{\beta+3n}.
                    \end{align}
                    
                    Our goal is to apply a Poincaré-type inequality on $\Omega_{j,k}$ to estimate $\xi_k-\xi_j$. To this aim, we first look at $\xi_j,\xi_k$ on an appropriate annulus. By the definition of $\mathcal{C}_j,\mathcal{C}_k$ and the structure of the vortex set in \cref{Main-result-epsilon-neighbourhood-of-graph} we can see that $|x_j-x_k| \leq 20C_0\epsilon$. We name the midpoint $x_{j,k}=(y_{j,k},z_{j,k}) := \frac{x_j+x_k}{2}$ and we see that
                    \begin{align*}
                        \mathcal{A}_{j,k} := [B^2_{3C_0\epsilon|\ln\epsilon|}(y_{j,k})\setminus B^2_{2C_0\epsilon|\ln\epsilon|}(y_{j,k})] \times [B^{n-2}_{5C_0\epsilon}(z_j)\cap B^{n-2}_{5C_0\epsilon}(z_k)]\subseteq\mathcal{A}_j\cap\mathcal{A}_k\subseteq \Omega_{j,k}.
                    \end{align*}
                    So we can compute that
                    \begin{align*}
                        \int_{\mathcal{A}_{j,k}} |\xi_j-\xi_k|^2 
                        \le 2\int_{\mathcal{A}_j}|(\theta +\xi_j) - \theta_h|^2+2\int_{\mathcal{A}_k}|(\theta +\xi_k) - \theta_h|^2.
                    \end{align*}
                    We know that $(\theta+\xi_j)-\theta_h$ and $(\theta+\xi_k)-\theta_h$ have zero mean on $\mathcal{A}_j$ and $\mathcal{A}_k$, respectively. Hence, we can apply \cref{Poincare-on-thin-annulus-appendix} on each annulus to see that
                    \begin{align*}
                        \int_{\mathcal{A}_{j,k}} \epsilon^{-2}|\xi_j-\xi_k|^2
                        &\leq C\int_{\mathcal{A}_j}|d(\theta +\xi_j) - d\theta_h|^2+ C\int_{\mathcal{A}_k}|d(\theta +\xi_k) - d\theta_h|^2,
                    \end{align*}
                    and we can bound
                    $$|d(\theta +\xi_j) - d\theta_h|\le|\alpha-d\theta|+|\alpha_h-d\theta_h|+|\alpha+d\xi_j-\alpha_h|.$$
                    As before, on $\mathcal{A}_j$ we have
                    $$|\alpha-d\theta|^2+|\alpha_h-d\theta_h|^2\le|u|^{-2}|\nabla u|^2+|u_h|^{-2}|\nabla_hu_h|^2\le\epsilon^{\beta+3n}$$
                    by exponential decay, and the same holds for $k$. Together with \cref{gauge-estimate-on-balls} we thus estimate
                    \begin{align}\label{pre-poincare-overlap-estimate-2}
                        &\int_{\mathcal A_{j,k}} \epsilon^{-2}|\xi_j-\xi_k|^2 \leq C|\ln\epsilon|^{4}(|\log\E|^2\E{(j)}+|\log\E|^2\E{(k)}) + C\epsilon^{\beta+3n}.
                    \end{align}
                    By \cref{pre-poincare-overlap-estimate-1}--\cref{pre-poincare-overlap-estimate-2}, using \cref{Poincare-on-thin-annulus-appendix} and \cref{rmk:finale}, we arrive at
                    \begin{align}\label{gauge-overlap-estimate}
                        \int_{\Omega_{j,k}} \epsilon^{-2}|\xi_j-\xi_k|^2 \leq C|\ln\epsilon|^{4}(|\log\E|^2\E{(j)}+|\log\E|^2\E{(k)})+\epsilon^{\beta+3n}.
                    \end{align}
                    We also need to estimate the difference of $\xi_k-\xi_0$ for all $1\leq k \leq N$. Defining
                    \begin{align*}
                        \Omega_{0,k} := \mathcal{C}_k \cap [\mathcal{A}_{s,\delta}\setminus Z_{C_0|\epsilon\ln\epsilon|}],
                    \end{align*}
                    we see that
                    \begin{align*}
                        \Omega_{0,k} \subseteq \mathcal{A}_k':= [B^2_{5C_0\epsilon|\ln\epsilon|}(y_k)\setminus B^2_{(C_0/2)\epsilon|\ln\epsilon|}(y_k)]\times B^{n-2}_{5C_0\epsilon}(z_k).
                    \end{align*}
                    Note that by \cref{far-away-gauge-definition} we have $\xi_0 = \theta_h - \theta$ in $\Omega_{0,k}$. Hence, we can apply \cref{Poincare-on-thin-annulus-appendix} and compute that
                    \begin{align*}
                        \epsilon^{-2}\int_{\Omega_{0,k}} |\xi_k - \xi_0|^2 &\leq  \epsilon^{-2}\int_{\mathcal{A}_k'} |(\theta+\xi_k) - \theta_h|^2\\
                        &\leq  \int_{\mathcal{A}_k'} |d(\theta+\xi_k) - d\theta_h|^2\\
                        &\leq  \int_{\mathcal{A}_k'} [|\alpha-d\theta|^2 + |\alpha_h-d\theta_h|^2 +|(\alpha+d\xi_k)-\alpha_h|^2],
                    \end{align*}
                    where again we used the fact that $(\theta+\xi_k) - \theta_h$ has zero mean on $\mathcal{A}_k\subset\mathcal{A}_k'$.
                    Summing the previous bounds and noting that
                    there are at most $(C\epsilon^{2-n})^2$ pairs $(j,k)$, while
                    any point belongs to at most $C=C(n,C_0)$ domains $\Omega_{j,k}$, we arrive at
                    \begin{align}\label{gauge-overlap-estimate-summed}
                        \sum_{0\leq j < k \leq N}\int_{\Omega_{j,k}} \epsilon^{-2}|\xi_j-\xi_k|^2 \leq C|\ln\epsilon|^{4}\int_{P(\mathcal{A}_{s-\delta,3\delta})}|\log\E_z|^2\E_z+C\E+ C\epsilon^{\beta+1},
                    \end{align}
                    for some $C=C(n,\beta)$ (recall \cref{e:convention}).
                    
                    \fbox{\textit{Constructing the global gauge via the partition of unity}.} Recall the definition of the partition of unity in \cref{partition-of-unity-1}. We define the global gauge transformation function as follows:
                    \begin{align}\label{global-gauge-definition}
                        \xi := \sum_{k=0}^{N} \phi_k\xi_k \quad\text{on } \mathcal{A}_{s,\delta}.
                    \end{align}
                    Then we estimate
                    \begin{align*}
                        \int_{\mathcal{A}_{s,\delta}} \epsilon^{-2}\left|e^{i\xi}u - \sum_{k=0}^N \phi_k e^{i\xi_k}u\right|^2
                        &\leq \epsilon^{-2}\int_{\mathcal{A}_{s,\delta}} \sum_{k=0}^N\phi_k|\xi - \xi_k|^2\\
                        &\leq 2\sum_{j<k} \int_{\Omega_{j,k}}\epsilon^{-2}|\xi_j - \xi_k|^2\\
                        &\leq C|\ln\epsilon|^{4}\int_{P(\mathcal{A}_{s-\delta,3\delta})}|\log\E_z|^2\E_z+C\E + C\epsilon^{\beta+1}.
                    \end{align*}
                    In particular,
                    \begin{align}\label{pre-final-gauge-estimate-3}
                    \begin{aligned}
                        \int_{\mathcal{A}_{s,\delta}} \epsilon^{-2}|e^{i\xi}u - u_h|^2 &\leq  2\int_{\mathcal{A}_{s,\delta}} \left[\epsilon^{-2}\left|e^{i\xi}u - \sum_{k=0}^{N}\phi_k e^{i\xi_k}u\right|^2 + \sum_{k=0}^{N}\phi_k|e^{i\xi_k}u - u_h|^2\right] \\
                        &\leq C|\ln\epsilon|^{8}\int_{P(\mathcal{A}_{s-\delta,3\delta})}|\log\E_z|^2\E_z +C\E +C\epsilon^{\beta+1},
                    \end{aligned}
                    \end{align}
                    where we used \cref{min.good.set} and \cref{min.bad.set} to estimate the term involving $e^{i\xi_k}u - u_h$.
                    
                    Moreover, for the connection part, we can bound
                    \begin{align}\label{pre-final-gauge-estimate-1}\begin{aligned}
                        &\int_{\mathcal{A}_{s,\delta}}|(\alpha+d\xi)-\alpha_h|^2 \\
                        &\leq 2\int_{\mathcal{A}_{s,\delta}} \left[\sum_{k=0}^{N}\phi_k|(\alpha+d\xi_k)-\alpha_h|^2 + \left|\sum_{k=0}^N d\phi_k \xi_k\right|^2\right]. 
                    \end{aligned}\end{align}
                    The first term is bounded by \cref{gauge-estimate-on-balls} and \cref{gauge-away-from-vrotex-set-estimate-1}.
                    We are left to bound the last term. Since the functions $\phi_k$ form a partition of unity, we have
                    \begin{align*}
                        \sum_{k=0}^N d\phi_k(z) = 0.
                    \end{align*}
                    We can then write
                    $$\sum_{k=0}^N d\phi_k \xi_k=\sum_{j,k=0}^N \phi_j d\phi_k(\xi_k-\xi_j).$$
                    Since $|d\phi_k|\le C\epsilon^{-1}$, the last term above is bounded by
                    $$C\epsilon^{-2}\sum_{j<k}\int_{\Omega_{j,k}}|\xi_j-\xi_k|^2,$$
                    which is a quantity that we already estimated.
                    Combining these bounds, we see that
                    \begin{align*}
                        \int_{\mathcal{A}_{s,\delta}} [\epsilon^{-2}|e^{i\xi}u - u_h|^2 + |(\alpha+d\xi)-\alpha_h|^2]
                        \leq C|\ln\epsilon|^{8} \int_{P(\mathcal{A}_{s-\delta,3\delta})} [\E_z+\uno_{\mathcal{G}^\eta}\E_z|\log\E_z|^2]+ \epsilon^{\beta}.
                    \end{align*}
                    This is indeed the desired conclusion.
            \end{proof}
            \end{proposition}
            
            We now turn to the bounds which were postponed in the previous proof.
            
			\begin{lemma}\label{min.good.set}
			Assume that $k\in G$. Then
			$$\int_{\mathcal{C}_k}\epsilon^{-2}|e^{i\xi_k}u - u_h|^2
						\leq C|\ln\epsilon|^{8} |\log\E|^2\E(k) + C\epsilon^{\beta+3n}$$
						and
			$$\int_{\mathcal{C}_k}\epsilon^2|d\alpha-d\alpha_h|^2
			\leq C|\ln\epsilon|^{4} \eta^{-2} |\log\E|^2\E(k) + C\epsilon^{\beta+3n}.$$
			\end{lemma}
			
            \begin{proof}        
                    We bound each part separately.
                    
                    \fbox{\textit{Estimating $d\alpha-d\alpha_h$}.} Recalling the definition of $\E_1$ in \cref{e:new-excess} and the construction of $\alpha_h$, we have
                    $$\int_{\mathcal{C}_k}\epsilon^2|d\alpha-d\alpha_h|^2
                    \le \int_{\mathcal{C}_k}\epsilon^2|d\alpha(e_1,e_2)-d\alpha_h(e_1,e_2)|^2
                    +C\E_1(k)+C\int_{B_{5C_0\epsilon}^{n-2}(z_k)}|dh|^2.$$

                    We are going to use some estimates from \cite{Halavati-stability} which are slightly more refined than
                    \cref{ymh-dim-2-stability-estimate}. Compared to the main result of \cite{Halavati-stability}, these hold under some additional assumptions, which are however satisfied on good slices:
                    in particular, for any $z\in B_{5C_0\epsilon}^{n-2}(z_k)$, the function $u$ vanishes linearly at a unique point along the slice $B_1^2\times\{z\}$. We will often compare $u$ with the function
                    $u_{h_0}$, where $h_0$ is the function built in \cref{Lipschitz-approximation-of-zero-loci-lemma}, whose graph approximates the zero set; along the good slice, this function vanishes at the same point as $u$,
                    and is just a translation of the standard degree-one planar solution.
                    
                    Specifically, using an $\epsilon$-rescaling of \cite[Theorem 3.1]{Halavati-stability} and \cite[Corollary 3.3]{Halavati-stability}
                    (applied with $N=1$; see also the part below \cite[Corollary 3.3]{Halavati-stability} for the choice of functions to which this result is applied),
                     we have the following estimate:
                    \begin{align}\label{curvature-estimate-temp-2}
                    \begin{aligned}
                        &\int_{B^2_1\times\{z\}} [\epsilon^{-2}||u|-|u_{h_0}||^2+|u_{h_0}|^{2+1/2}|(\alpha-d\theta)_{(1,2)} - (\alpha_{h_0}-d\theta_{h_0})_{(1,2)}|^2 \\
                        &\phantom{\int_{B^2_1\times\{z\}} [}+\epsilon^2|d\alpha(e_1,e_2) - d\alpha_{h_0}(e_1,e_2)|^2] \\
                        & \leq C\E_z,
                    \end{aligned}
                    \end{align}
                    for an absolute constant $C$, where the subscript $(1,2)$ means that we restrict the one-form along the slice.
                    
                    Now by the construction in \cref{pull-back-proposition} we can see that $(u_h,\alpha_h)$, along the slice $B^2_1\times \{z\}$, is equal to $(u_{h_0},\alpha_{h_0})$ translated to vanish at the barycenter $\Phi_{\chi(x_1,x_2)}(z)$. As shown in \cref{zero-set-distance-barycenter-slice},
                    the translation is by a vector $v$ with $|v| \leq C\epsilon|\log\E_z|\sqrt{\E_z}+e^{-K/\epsilon}$. By the mean value theorem, we then have
                    \begin{align}\label{translation-estimate-for-taubes-solution}
                    \begin{aligned}
                        &\int_{B_{5C_0\epsilon|\log\epsilon|}^2(y_k)\times \{z\}}
                        [\epsilon^2|d\alpha_h(e_1,e_2) - d\alpha_{h_0}(e_1,e_2)|^2
                        + |(\alpha_h-\alpha_{h_0})_{(1,2)}|^2 \\
                        &\phantom{\int_{B_{5C_0\epsilon|\log\epsilon|}^2(y_k)\times \{z\}}
                                                [}+ \epsilon^{-2} |u_h-u_{h_0} |^2]\\
                        &\le C\epsilon^2|\log\epsilon|^2\cdot |v|^2\cdot C\epsilon^{-4}\\
                        &\le C|\ln\epsilon|^2|\log\E_z|^2\E_z+e^{-K/\epsilon},
                    \end{aligned}
                    \end{align}
                    since $\E_z$ is bounded on good slices and the differential of each quantity
                    (such as $\epsilon d\alpha_h(e_1,e_2)$ and so on) is bounded by $C\epsilon^{-2}$.
                    The claimed estimate follows by combining the previous bounds (together with item (iv) from \cref{Lipschitz-approximation-thm}, which gives $|dh|^2(z)\le C\E_z+e^{-K/\epsilon}$).

                    \fbox{\textit{Estimating $e^{i\xi_k}u-u_h$}.} Writing formally $u=|u|e^{i\theta}$ and using a similar notation for $u_h$ and $u_{h_0}$, recall that on the annulus
                    $$\mathcal{A}_{k,z}:=[B^2_{5C_0\epsilon|\ln\epsilon|}(y_k)\setminus B^2_{C_0\epsilon|\ln\epsilon|}(y_k)]\times \{z\}$$
                    the differences $\theta-\theta_h$, $\theta-\theta_{h_0}$ and $\theta_h-\theta_{h_0}$ are well-defined.
                    We record the following estimate:
                    \begin{align}\label{phase-translation-estimate}
                        \int_{\mathcal{A}_{k,z}} \epsilon^{-2}|\theta_h-\theta_{h_z}|^2
                        \leq C|\log\E_z|^2\E_z.
                    \end{align}
                    This holds again by the mean value theorem, since $|(d\theta_h)_{(1,2)}|(y)\le C|y-y_k|^{-1}$.
                    We are going to use the Caffarelli--Kohn--Nirenberg-type inequality from \cref{CKN-poincare-lemma-appendix}, which implies that
                    \begin{align*}
                        \int_{B_R^2(y_k)} |y-h_{0}(z)|^2|f(y)|^2 \leq CR^{3/2}\int_{B_R^2(y_k)}|y-h_{0}(z)|^{2+1/2}|df(y)|^2,\quad\text{for $f\in C^1_c(B_R^2(y_k))$},
                    \end{align*}
                    with $R:=5C_0\epsilon|\log\epsilon|$ (since there exists a biLipschitz transformation sending $B_R^2(y_k)$ to itself and mapping the origin to $h_0(z)$).
                    Recalling that the standard degree-one solution vanishes linearly at the origin, by the construction of $u_{h_0}$ in \cref{pull-back-proposition} we have
                    \begin{align*}
                        C^{-1} \leq \frac{|u_{h_0}|(y,z)}{\min\{\epsilon^{-1}|y-h_0(z)|,1\}} \leq C
                    \end{align*}
                    on the good slice,
                    for some universal constant $C$. Moreover,
                    \begin{align*}
                        1\leq \frac{\epsilon^{-1}|y-h_0(z)|}{\min\{\epsilon^{-1}|y-h_0(z)|,1\}} \leq C|\ln\epsilon|
                        \quad\text{for all $y\in B_{5C_0\epsilon|\log\epsilon|}^2(y_k)$}.
                    \end{align*}
                    Hence, given a $C^1$ function $f$ on $B^{2}_{5C_0\epsilon|\ln\epsilon|}(y_k)\times \{z\}$
                    vanishing near the boundary, we can write
                    \begin{align}\label{ckn-temp-1}
                        \int_{B^{2}_{5C_0\epsilon|\ln\epsilon|}(y_k)\times \{z\}} |u_{h_0}|^2|f|^2
                        &\leq C\epsilon^{2}|\ln\epsilon|^{4}\int_{B^{2}_{5C_0\epsilon|\ln\epsilon|}(y_k)\times \{z\}} |u_{h_0}|^{2+1/2}|df|^2.
                    \end{align}
                    
                    To estimate $ue^{i\xi_k}-u_{h_0}$, we first notice that $u$ and $u_{h_0}$ have the same unique zero point (with the same degree around it), and hence the difference $\theta-\theta_{h_0}$ gives a well-defined smooth function on the full slice.
                    
                    We define a cut-off $\chi:B^2_1 \to \R$ with $\chi = 1$ on $B^2_{C_0\epsilon|\ln\epsilon|}(y_k)$ and $\chi=0$ outside of $B^2_{5C_0\epsilon|\ln\epsilon|}(y_k)$, with $|d\chi| \leq C|\epsilon\ln\epsilon|^{-1}$. Then we use the first term of \cref{curvature-estimate-temp-2} to bound $|u|-|u_{h_z}|$ and \cref{ckn-temp-1} to see that
                    \begin{align*}
                        &\epsilon^{-2}\int_{B^{2}_{5C_0\epsilon|\ln\epsilon|}(y_k)\times \{z\}} \chi^2|e^{i\xi_k}u - u_{h_0}|^2 \\
                        &\leq C\int_{B^{2}_{5C_0\epsilon|\ln\epsilon|}(y_k)\times \{z\}} \chi^2\left[\frac{\left||u| - |u_{h_0}|\right|^2}{\epsilon^2} + \frac{|u_{h_0}|^2}{\epsilon^2}|(\theta+\xi_k)-\theta_{h_0}|^2\right]\\
                        &\leq C\E_z + C(\bm{\mathrm{I}}+\bm{\mathrm{II}}),
                    \end{align*}
                    where
                    \begin{align*}
                    \bm{\mathrm{I}}&:=|\ln\epsilon|^4\int_{B^{2}_{5C_0\epsilon|\ln\epsilon|}(y_k)\times \{z\}} |u_{h_0}|^{2+1/2}\chi^2|d(\theta+\xi_k-\theta_{h_0})_{(1,2)}|^2,\\
                    \bm{\mathrm{II}}&:=|\ln\epsilon|^4\int_{B^{2}_{5C_0\epsilon|\ln\epsilon|}(y_k)\times \{z\}}  |u_{h_0}|^{2+1/2}|d\chi|^2|\theta+\xi_k-\theta_{h_0}|^2.
                    \end{align*}
                    First we estimate $\textbf{I}$ using the second term in \cref{curvature-estimate-temp-2} and \cref{translation-estimate-for-taubes-solution} (to replace $\alpha_{h_0}$ with $\alpha_h$):
                    \begin{align*}
                        \textbf{I} &\leq C|\ln\epsilon|^6|\log\E_z|^2\E_z + C|\ln\epsilon|^4\int_{B^{2}_{5C_0\epsilon|\ln\epsilon|}(y_k)\times \{z\}}\chi^2|(\alpha+d\xi_k) - \alpha_{h}|^2.
                    \end{align*}
                    Then we estimate $\textbf{II}$: we note that $|d\chi|$ is supported in $\mathcal{A}_{k,z}$ and $|d\chi| \leq C|\epsilon\ln\epsilon|^{-1}$, and hence we can use \cref{phase-translation-estimate} to estimate
                    \begin{align*}
                        \textbf{II} &\leq C|\ln\epsilon|^2|\log\E_z|^2\E_z+ C|\ln\epsilon|^2\int_{\mathcal{A}_{k,z}} \epsilon^{-2}|\theta+\xi_k-\theta_h|^2.
                    \end{align*}
                    Putting $\textbf{I}$ and $\textbf{II}$ together and  integrating over $B^{n-2}_{5C_0\epsilon}(z_k)$, we see that
                    \begin{align*}
                        &\int_{\mathcal{C}_k} \epsilon^{-2}|e^{i\xi_k}u - u_{h_0}|^2 \\
                        & \leq C|\ln\epsilon|^6|\log\E|^2\E{(k)}+C|\ln\epsilon|^4\int_{\mathcal{C}_k}|\alpha+d\xi_k - \alpha_h|^2
                        + C|\ln\epsilon|^2\int_{\mathcal{A}_k} \epsilon^{-2}|\theta+\xi_k-\theta_h|^2 \\
                        &\leq C|\ln\epsilon|^8|\log\E|^2\E{(k)}+ C|\ln\epsilon|^2\int_{\mathcal{A}_k} \epsilon^{-2}|\theta+\xi_k-\theta_h|^2+\epsilon^{\beta+3n},
                    \end{align*}
                    where we used \cref{gauge-estimate-on-balls} (which uses only the previous bound on $d\alpha-d\alpha_h$).
                    Recalling that we imposed
                    \begin{align*}
                        \int_{\mathcal{A}_k}(\theta+\xi_k-\theta_h) =0,
                    \end{align*}
                    we can apply \cref{Poincare-on-thin-annulus-appendix} (suitably rescaled) and \cref{gauge-estimate-on-balls} another time to see that
                    \begin{align}\label{phase-estimate-on-good-balls}\begin{aligned}
                        &|\ln\epsilon|^2\int_{\mathcal{A}_k} \epsilon^{-2}|\theta+\xi_k-\theta_h|^2 \\
                        &\leq C|\ln\epsilon|^2\int_{\mathcal{A}_k}|d(\theta+\xi_k) - d\theta_h|^2\\
                        &\leq C|\ln\epsilon|^2\int_{\mathcal{A}_k} [|\alpha-d\theta|^2 + |\alpha_h-d\theta_h|^2 + |(\alpha+d\xi_k)-\alpha_h|^2] \\
                        &\leq C|\ln\epsilon|^6|\log\E|^2\E{(k)}+ \epsilon^{\beta+3n}
                    \end{aligned}\end{align}
                    up to taking $C_0$ large enough (the last inequality follows from the exponential decay of energy);
                    combining these bounds with  \cref{translation-estimate-for-taubes-solution},
                    we get the desired bound for $e^{i\xi_k}u-u_h$.
%
            \end{proof}

            \begin{lemma}\label{min.bad.set}
            For $k\in B$ we have
            $$\int_{\mathcal{C}_k} [\epsilon^{-2}|e^{i\xi_k}u - u_h|^2 + \epsilon^2|d\alpha - d\alpha_h|^2]
            \leq C|\ln\epsilon|^2\E_1{(k)}.$$
            \end{lemma}
                    
            \begin{proof}
                    On the bad set we simply use $L^\infty$ bounds: we have
                    \begin{align}\label{gauge-estimate-on-bad-balls}\begin{aligned}
                        &\int_{\mathcal{C}_k} [\epsilon^{-2}|e^{i\xi_k}u - u_h|^2 + \epsilon^2|d\alpha - d\alpha_h|^2]  \\ 
                        &\leq C|\ln\epsilon|^2\epsilon^{n-2} \\
                        &\leq C|\ln\epsilon|^2\E_1{(k)},
                    \end{aligned}\end{align}
                    where we used the definition of bad index in the last inequality.
            \end{proof}
            
            \begin{corollary}\label{interp.cor}
            We have
            \begin{align*}
                            &\int_{\mathcal{A}_{s,\delta}} [\epsilon^{-2}|e^{i\xi}u-u_h|^2 + |(\alpha+d\xi)-\alpha_h|^2] \leq C(n,\beta)|\ln\epsilon|^{10}\int_{P(\mathcal{A}_{s-\delta,3\delta})} \E_z+ \epsilon^\beta.
                        \end{align*}
            \end{corollary}
            
            \begin{proof}
            Recalling that for a good $z$ the sliced excess $\E_z$ is small,
            it suffices to split the integral of $\E_z|\log\E_z|^2$ on the two sets $\{\E_z\le\epsilon^{\beta+1}\}$
            and $\{\E_z>\epsilon^{\beta+1}\}$. On the second one, we bound $|\log\E_z|^2\le C|\log\epsilon|^2$,
            while on the first one we have $\E_z|\log\E_z|^2\le C\epsilon^{\beta+1}|\log\epsilon|^2$.
            \end{proof}

            \section{Proof of a stronger decay of tilt-excess for local minimizers}
            \subsection{Strong approximation of the excess for minimizers}
            In this section we use variational arguments and the estimates from \cref{Interpolation-gauge-proposition} to construct competitors. As a consequence, we prove that the full excess $\E$ is well approximated by the Dirichlet energy of a harmonic approximation $w$ built as in \cref{harmonic-approx-prop}.
            
            \begin{proposition}[Strong harmonic approximation of minimizers]\label{strong-approximation-of-excess-proposition}
                For any $\nu,\beta>0$ and any radius $0<s<1$ there exist small constants $\epsilon_0(n,s,\nu,\beta),\tau_0(n,s,\nu,\beta),\eta_0(n,\nu,\beta)>0$ with the following properties. Let $(u,\nabla)$ be a minimizer of $E_\epsilon$ defined on $B^n_2(0)$, with $\epsilon\leq \epsilon_0$
                and the energy bound
                $$\frac{1}{|B_2^{n-2}|}\int_{B_2^n}e_\epsilon(u,\nabla)\le 2\pi+\tau_0.$$
                After a suitable rotation, let $h:B^{n-2}_1\to B^2_1$ be the Lipschitz approximation defined in \cref{Lipschitz-approximation-thm} with $\eta:=\eta_0$. Then the following holds,
                assuming
                $$Ce^{-K/\epsilon}\le C\epsilon^\beta\le\E:=\E(u,\nabla,B^n_{2},\R^{n-2})$$
                for some $C=C(n,\nu,\beta)$ and $K=K(n)$:
                there exists a harmonic function $w:B_1^{n-2}\to\R^2$ such that
                \begin{enumerate}[label=(\roman*)]
                    \item $\int_{B_1^{n-2}}|dw|^2\le C(n)$;
                    \item we have $$\int_{B_1^{n-2}}\left|\frac{h-(h)_{B_1^{n-2}}}{\sqrt{\E}}-w\right|^2\le\nu,$$ where $(h)_{B_1^{n-2}}$ is the average;
                    \item most importantly, we have
                    $$\int_{B_s^{n-2}}\E_z\le\E\int_{B_s^{n-2}}\frac{|dw|^2}{2}+\nu\E.$$
                \end{enumerate}
                \begin{proof}
                	We prove the statement by compactness and contradiction. Fix $\nu,\beta,s$ and assume that there exist  sequences $\epsilon_k,\tau_k\to0$ and a sequence of minimizers $(u_k,\nabla_k)$ for $E_{\epsilon_k}$ with the previous energy bound for $\tau_0=\tau_k$, violating the conclusion. Moreover, let $h_k:B^{n-2}_1\to B^2_1$ be the Lipschitz approximation for the threshold $\eta_0$, to be chosen below.
                	
                	\fbox{\emph{Lower bound on the energy of the given pair.}}
                	By item (i) in \cref{Lipschitz-approximation-thm} we have
                    \begin{align*}
                        \int_{B^{n-2}_1} |dh_k|^2 \leq C(n)\E^{(k)},\quad \E^{(k)}:=\E(u_k,\nabla_k,B^n_{2},\R^{n-2}),
                    \end{align*}
                    for $k$ large enough.
                    Hence, up to a subsequence, we can extract a weak limit
                    $$\frac{h_k-(h_k)_{B_1^{n-2}}}{\sqrt{\E^{(k)}}}\weakto w$$
                    in $W^{1,2}$, so that
                    $$\int_{B_1^{n-2}}|dw|^2\le C(n).$$
                    By \cref{harmonic-approximation-lemma-1} and \cref{harmonic-approx-prop}, $w$ is harmonic with $w(0)=0$.
                    This shows that the first two conclusions hold, so we must have
                    \begin{align}\label{contra.nu}
                    \int_{B_s^{n-2}}\E_z^{(k)}>\E^{(k)}\int_{B_s^{n-2}}\frac{|dw|^2}{2}+\nu\E^{(k)}.
                    \end{align}
                    By \cref{energy-identity-on-slice} and the bound $Ce^{-K/\epsilon}\le\frac{\nu}{5}\E^{(k)}$, this gives
                    $$\frac{1}{2\pi}\int_{B_1^2\times B_s^{n-2}}e_{\epsilon_k}(u_k,\nabla_k)>|B_s^{n-2}|
                    +\E^{(k)}\int_{B_s^{n-2}}\frac{|dw|^2}{2}+\frac{4\nu}{5}\E^{(k)}.$$
                    
                    Let $a,b\in(s,1)$ with $a<b$, which we write as $b=a+4\delta$.
                    Calling $\mathcal{G}^{k}$ the good set for $(u_k,\nabla_k)$,
                    since the indicator function $\uno_{(B^{n-2}_a\setminus B^{n-2}_s)\cap\mathcal{G}^{k}}\to1$
                    strongly $L^2(B^{n-2}_a\setminus B^{n-2}_s)$, we have
                    $$\uno_{(B^{n-2}_a\setminus B^{n-2}_s)\cap\mathcal{G}^{k}}\frac{dh_k}{\sqrt{\E^{(k)}}}\weakto dw$$
                    weakly in this space, and hence
                    $$\int_{B^{n-2}_a\setminus B^{n-2}_s} \frac{|dw|^2}{2}
                    \le\liminf_{k\to\infty}\int_{(B^{n-2}_a\setminus B^{n-2}_s)\cap\mathcal{G}^{k}}\frac{|dh_k|^2}{2\E^{(k)}}.$$
                    Using  item (iv) in \cref{Lipschitz-approximation-thm} and the assumption \(\E^{(k)}\ge C\epsilon_{k}^{\beta}\ge Ce^{-K/\epsilon_{k}}\), we deduce
                    $$\int_{B^{n-2}_a\setminus B^{n-2}_s}\frac{|dw|^2}{2}
                    \le\liminf_{k\to\infty}\int_{(B^{n-2}_a\setminus B^{n-2}_s)\cap\mathcal{G}^{k}}\frac{\E_z^{(k)}}{\E^{(k)}}.$$
                    Combined with \cref{contra.nu}, this gives
                    $$\int_{B^{n-2}_a}\E_z^{(k)}>\E^{(k)}\int_{B_a^{n-2}}\frac{|dw|^2}{2}+\frac{3\nu}{4}\E^{(k)}.$$
                    Using again \cref{energy-identity-on-slice} and \(\E^{(k)}\ge C\epsilon_{k}^{\beta}\ge Ce^{-K/\epsilon_{k}}\), we obtain
                    \begin{align}\label{lb.a}
                    \frac{1}{2\pi}\int_{B_1^2\times B_a^{n-2}}e_{\epsilon_k}(u_k,\nabla_k)>|B_a^{n-2}|
                                        +\E^{(k)}\int_{B_a^{n-2}}\frac{|dw|^2}{2}+\frac{3\nu}{5}\E^{(k)}.
                    \end{align}
                    
                    Note that, for a fixed small $\delta>0$ to be specified later, we can find $a$ and $b=a+4\delta$ in $(s,1)$ such that
                    \begin{align}\label{pigeonhole}
                    \int_{B_{b}^{n-2}\setminus B_a^{n-2}}[|dh_k|^2+\E^{(k)}|dw|^2+\E_z^{(k)}]
                    \le C(n,s)\delta\E^{(k)},
                    \end{align}
                    along a subsequence, by the classical pigeonhole argument.
                    
                    Now we take a cut-off function $\chi$ such that $\chi=1$ on $B_{a}^{n-2}$ and $\chi=0$
                    outside of $B_{a+\delta}^{n-2}$, and we let
                    $$f_k:=(1-\chi)h_k+\chi(\sqrt{\E^{(k)}}w+(h_k)_{B_1^{n-2}}).$$
                    Since $\|h_k-(h_k)_{B_1^{n-2}}-\sqrt{\E^{(k)}}w\|_{L^2}^2=o(\E^{(k)})$,
                    the Dirichlet energy of $f_k$ on $B_b^{n-2}\setminus B_a^{n-2}$ is
                    $$\int_{B_b^{n-2}\setminus B_a^{n-2}}\left[(1-\chi)^2\frac{|dh_k|^2}{2}+\E^{(k)}(2\chi-\chi^2)\frac{|dw|^2}{2}\right]+o(\E^{(k)}).$$
                    In particular, by \cref{pigeonhole} we have
                    \begin{align}\label{pigeon.f}
                    \int_{B_b^{n-2}\setminus B_a^{n-2}}|df_k|^2\le C\delta\E^{(k)}.
                    \end{align}
                    We apply \cref{pull-back-proposition} to obtain a new pair $(u_{f_k},\nabla_{f_k})$.
                    
                    \fbox{\emph{Construction of the competitor.}}
                    We want to glue the latter to $(u_k,\nabla_k)$ in a suitable annular region
                    and obtain a new pair whose energy in $B_1^2\times B_b^{n-2}$ is strictly lower than $(u_k,\nabla_k)$, obtaining a contradiction to minimality.
                    From now on, we restrict attention to the region $B_1^2\times B_b^{n-2}$.
                    We will also drop the subscript $k$ in the sequel. Note that $f=h$ on $B_{a+4\delta}^{n-2}\setminus B_{a+\delta}^{n-2}$.
                    
                    For technical reasons, it will be convenient to glue on an annulus of width $\sqrt\epsilon$.
                    We first select $t\in[a+2\delta,a+3\delta]$ such that
                    \begin{align}\label{pigeon.bis}
                    \int_{B_{t+2\sqrt\epsilon}^{n-2}\setminus B_{t-\sqrt\epsilon}^{n-2}}[\E_z+|dh|^2]
                    \le C(n,s,\delta)\sqrt\epsilon\E.
                    \end{align}
                    
                    We first apply \cref{Interpolation-gauge-proposition} and \cref{interp.cor} to replace $(u,\nabla)$
                    with a gauge-equivalent pair, still denoted $(u,\nabla)$,
                    such that $\frac{u}{|u|}=\frac{u_f}{|u_f|}$ on $(B_1^2\setminus B_{1/2}^2)\times B_b^{n-2}$, with
                    $$
                    \int_{\mathcal{A}}[\epsilon^{-2}|u-u_h|^2+|\alpha-\alpha_h|^2]\le C|\ln\epsilon|^{10}
                    \int_{P(\hat{\mathcal{A}})}\E_z+\epsilon^\beta,
                    $$
                    where $\mathcal{A}:=B_1^2\times(B_{t+\sqrt\epsilon}^{n-2}\setminus B_{t}^{n-2})$
                    and $\hat{\mathcal{A}}:=B_1^2\times(B_{t+2\sqrt\epsilon}^{n-2}\setminus B_{t-\sqrt\epsilon}^{n-2})$.
                    In particular, by \cref{pigeon.bis} we have
                    \begin{align}\label{interp.est}
                    \int_{\mathcal{A}}[\epsilon^{-2}|u-u_h|^2+|\alpha-\alpha_h|^2]\le C\sqrt\epsilon|\ln\epsilon|^{10}
                    \int_{P(\hat{\mathcal{A}})}\E_z+\epsilon^\beta
                    =o(\E)+\epsilon^\beta,
                    \end{align}
                    where the notation $o(\E)=o(\E^{(k)})$ indicates a quantity infinitesimal with respect to $\E^{(k)}$, as $k\to\infty$.
                    
                    We take another cut-off function $\varphi$ with $\varphi=1$ on $B_1^2\times B_t^{n-2}$
                    and $\varphi=0$ outside of $B_1^2\times B_{t+\sqrt\epsilon}^{n-2}$.
                    On $B_1^2\times B_b^{n-2}$, we define
                    $$\tilde u:=(1-\varphi)u+\varphi u_h,\quad\tilde\alpha:=(1-\varphi)\alpha+\varphi\alpha_h.$$
                    
                    We claim that
                    \begin{align}\label{claim.palloso}
                    \int_{P(\mathcal{A})}\tilde\E_z\le
                    o(\E)+C\epsilon^\beta.\end{align}
                    Once this is done, using \cref{energy-identity-on-slice}, we obtain
                    $$\frac{1}{2\pi}\int_{P(\mathcal{A})}e_\epsilon(\tilde u,\tilde \nabla)\le
                    |P(\mathcal{A})|+o(\E_z)+C\epsilon^\beta,
                    $$
                    and hence by \cref{pull-back-proposition}, together with \cref{pigeonhole} and \cref{pigeon.f}, we get
                    \begin{align*}
                    &\frac{1}{2\pi}\int_{B_1^2\times B_b^{n-2}}e_\epsilon(\tilde u,\tilde \nabla)\\
                    &\le
                    |B_b^{n-2}|+(1+O(\eta_0^2))\int_{B_{b}^{n-2}}\frac{|df|^2}{2}
                    +\int_{B_b^{n-2}\setminus B_a^{n-2}}\E_z
                    +o(\E)+C\epsilon^\beta\\
                    &\le|B_b^{n-2}|+\E\int_{B_{a}^{n-2}}\frac{|dw|^2}{2}+\int_{B_b^{n-2}\setminus B_a^{n-2}}[|df|^2+\E_z]
                    +C\eta_0^2\E+o(\E)+C\epsilon^\beta\\
                    &\le|B_b^{n-2}|+\E\int_{B_{a}^{n-2}}\frac{|dw|^2}{2}+\frac{\nu}{5}\E,
                    \end{align*}
                    once we take $\eta_0$ and $\delta$ small enough.
                    In the same way, \cref{lb.a} gives
                    $$\frac{1}{2\pi}\int_{B_1^2\times B_b^{n-2}}e_{\epsilon}(u,\nabla)>|B_b^{n-2}|
                    +\E\int_{B_a^{n-2}}\frac{|dw|^2}{2}+\frac{2\nu}{5}\E.$$
                    This gives a contradiction: near $\de B_1^2\times B_{t+\sqrt\epsilon}^{n-2}$,
                    using the fact that $\tilde u$ and $u_f$ have the same phase,
                    it is easy to modify $(\tilde u,\tilde\nabla)$ in order to make it agree with
                    $(u,\nabla)$ (while this already holds on $\de B_1^2\times (B_b^{n-2}\setminus B_{t+\sqrt\epsilon}^{n-2})$),
                    in a way which changes the energy by $O(e^{-K/\epsilon})\le\frac{\nu}{5}\E$:
                    it is enough to interpolate between the two pairs on the set
                    $$(B_1^2\setminus B_{1/2}^2)\times B_b^{n-2},$$
                    using the fact that here the energy density is exponentially small,
                    and hence we can write $|\tilde u-u_f|=|(1-|\tilde u|)-(1-|u_f|)|\le e^{-K/\epsilon}$ and $|\tilde\alpha-\alpha_f|\le e^{-K/\epsilon}$ (since, writing $\tilde u=e^{i\tilde\theta}$, we have
                    $|d\tilde\theta-\tilde\alpha|\le|\tilde u|^{-1}|\tilde \nabla \tilde u|$ and similarly $|d\tilde\theta-\alpha_f|=|d\theta_f-\alpha_f|\le|u_f|^{-1}|\nabla_fu_f|$).
                    
                    \fbox{\emph{Bounding the energy on the interpolation annulus.}}
                    It remains to check the previous claim.
                    We first write
                    \begin{align*}
                    2\pi\tilde\E_z&=\int_{B^{2}_{1}\times \{z\}}\sum_{j\ge3}|\tilde\nabla_{e_j}\tilde u|^2
	                +\sum_{(j,j')\neq(1,2)}\epsilon^2|d\tilde\alpha(e_j,e_{j'})|^2
	                +|\tilde\nabla_{e_1}\tilde u+i\tilde\nabla_{e_2}\tilde u|^2
	                +\left|\epsilon d\tilde\alpha(e_1,e_2)-\frac{1-|\tilde u|^2}{2\epsilon}\right|^2\\
	                &=:\bm{\mathrm{I}}+\bm{\mathrm{II}}+\bm{\mathrm{III}}+\bm{\mathrm{IV}}.
                    \end{align*}
                    
                    We start by bounding $\bm{\mathrm{II}}$:
                    we have
                    $$d\tilde\alpha=(1-\varphi)d\alpha+\varphi d\alpha_h+d\varphi\wedge(\alpha_h-\alpha).$$
                    Hence, using the fact that $|d\varphi|\le C\epsilon^{-1/2}$, we have
                    $$\bm{\mathrm{II}}\le C\E_z+C|dh|^2(z)+C\epsilon^2\cdot C\epsilon^{-1}\int_{B^{2}_{1}\times \{z\}}|\alpha_h-\alpha|^2.$$
                    The last term is bounded by the left-hand side of \cref{interp.est}; together with \cref{pigeon.bis}, this gives the desired bound.
                    
                    As for $\bm{\mathrm{IV}}$, we note that
                    $$\frac{1-|\tilde u|^2}{2\epsilon}=(1-\varphi)\frac{1-|u|^2}{2\epsilon}+\varphi\frac{1-|u_h|^2}{2\epsilon}
                    +O(\epsilon^{-1}|u-u_h|),$$
                    and hence $\bm{\mathrm{IV}}$ is the squared norm of
                    $$(1-\varphi)\left[\epsilon d\alpha(e_1,e_2)-\frac{1-|u|^2}{2\epsilon}\right]+\varphi\left[\epsilon d\alpha_h(e_1,e_2)-\frac{1-|u_h|^2}{2\epsilon}\right]+O(\sqrt\epsilon|\alpha-\alpha_h|)+O(\epsilon^{-1}|u-u_h|).$$
                    Thus, we have
                    $$\bm{\mathrm{IV}}\le C\E_z+C\epsilon\int_{B^{2}_{1}\times \{z\}}|\alpha-\alpha_h|^2+C\epsilon^{-2}\int_{B^{2}_{1}\times \{z\}}|u-u_h|^2,$$
                    again bounded by \cref{pigeon.bis} and \cref{interp.est}.
                    
                    We finally turn to $\bm{\mathrm{I}}$; the bound for $\bm{\mathrm{III}}$
                    is obtained in the same way and hence will be skipped.
                    We note that
                    \begin{align*}
                    \tilde\nabla\tilde u&=d[(1-\varphi)u+\varphi u_h]-i[(1-\varphi)u+\varphi u_h][(1-\varphi)\alpha+\varphi\alpha_h]\\
                    &=(1-\varphi)\nabla u+\varphi\nabla_h u_h+(u_h-u)d\varphi+O(|u-u_h||\alpha-\alpha_h|).
                    \end{align*}
                    Hence,
                    $$\int_{B^{2}_{1}\times \{z\}}|\tilde\nabla_{e_j}\tilde u|^2
                    \le C\E_z+C|dh|^2(z)+C\epsilon^{-1}\int_{B^{2}_{1}\times \{z\}}|u-u_h|^2+C\int_{B^{2}_{1}\times \{z\}}|\alpha-\alpha_h|^2.$$
                    Again, the last two terms are bounded by the left-hand side of \cref{interp.est}.
                    This completes the proof of \cref{claim.palloso}, and hence the proof of the proposition.
                \end{proof}
            \end{proposition}

            \subsection{Proof of \cref{tilt-excess-decay-minimizer}}
            In this section we finish the proof of the stronger decay of excess for minimizers. 
				We rescale $B_1^n(0)$ to $B_2^n(0)$ and
				apply \cref{strong-approximation-of-excess-proposition}, with some $s\in(0,\mz)$ and $\nu>0$ to be chosen later
				and with $\beta+1$ in place of $\beta$.
				We obtain that either $\E=\E(u,\nabla,B_2^n(0),\R^{n-2})\le C\epsilon^{\beta+1}$ or
				the conclusions of \cref{strong-approximation-of-excess-proposition} hold true
				(provided that the picture is rotated in such a way that $\E$ is small enough).
				
				In the first situation, we clearly have
				$\min_S\E(u,\nabla,B_2^n,S)\le\epsilon^\beta$ for $\epsilon$ small enough and we are done.
				Hence, in the sequel, we can assume that we are in the second situation.
				
				We will assume for simplicity that $\R^{n-2}$ minimizes $\E(u,\nabla,B_2^n,\cdot)$ and that 
				\begin{equation}\label{e:delta}
			|dw(0)|\le\delta
			\end{equation}
			 with $\delta>0$ to be chosen momentarily.
				
				Since $|dw(z)-dw(0)|\le s\sup_{B_s^{n-2}}|D^2w|$ on $B_s^{n-2}$, we have
				$$\int_{B_s^{n-2}}|dw|^2\le C(n)s^{n-2}\delta^2+C(n)s^{n}\sup_{B_s^{n-2}}|D^2w|^2
				\le C(n)s^{n-2}(\delta^2+s^2),$$
				where the last inequality comes from the bound $\|dw\|_{L^2}\le C(n)$ and standard elliptic estimates.
				
				By item (iii) from \cref{strong-approximation-of-excess-proposition} we then have
				$$\int_{B_s^{n-2}}\E_z\le \E\int_{B_s^{n-2}}\frac{|dw|^2}{2}+\nu\E
				\le C(n)s^{n-2}(\delta^2+s^2+s^{2-n}\nu)\E.$$
				This immediately gives
				$$\E(u,\nabla,B_s^n,\R^{n-2})\le C(n)(\delta^2+s^2+s^{2-n}\nu)\E.$$
				The theorem follows under the assumption \eqref{e:delta}  by taking  $\delta,s$ and \emph{subsequently} $\nu$ small enough.
				The general case can be reduced to this one by the very same argument of  Section \ref{Tilt-subsection-label};
				the only differences here are that we use item (ii) from \cref{strong-approximation-of-excess-proposition} in order to bound
				$$\|h-(h)_{B_1^{n-2}}-\sqrt{\E} w\|_{L^2}^2\le\nu\E$$
				and that $\E_1$ is replaced by $\E$ throughout that argument. 

            \subsection{Proof of \cref{main-result-global-n4}: the case of minimizers}
            We prove the following theorem, which contains the second part of \cref{main-result-global-n4}.
            \begin{thm}
                For $n\geq2$, there exists $\tau_0(n)>0$ such that the following holds. If $(u,\nabla)$ is an entire, local minimizer for the energy $E_1$, with $u(0)=0$ and the energy bound
                \begin{align*}
                    &\frac{1}{|B^{n-2}_R|}\int_{B_R^n} \left[|\nabla u|^2 + |F_\nabla|^2 + \frac14(1-|u|^2)^2\right] \leq 2\pi+\tau_0,
                \end{align*}
                then $(u,\nabla)$ is two-dimensional. More precisely,
                we have $(u,\nabla) = P^*(u_0,\nabla_0)$ up to a change of gauge, where $P$ is the orthogonal projection onto a two-dimensional subspace and $(u_0,\nabla_0)$ is the standard degree-one solution of Taubes \cite{Taubes} (or its conjugate), centered at the origin.
                \begin{proof}
                	We can assume $n\ge3.$
                	We proceed exactly as in the proof of \cref{main-result-global-n4}: letting $\beta:=n-2>0$,
                	it is enough to prove that
                	$$\lim_{R\to\infty}R^\beta\min_S\E_1(u,\nabla,B_R^n,S)=0.$$
                	This follows from the stronger excess decay statement for minimizers, using
                	the same iteration argument employed in the proof of \cref{main-result-global-n4}.
                \end{proof}
            \end{thm}

                    \appendix
                    \section{Barycenter and variance of good slices}
				 We show two lemmas which give a more refined control of a critical pair on a \emph{good slice}
				 $B_1^2\times\{z\}$, with $z\in\mathcal{G}^\eta$, the good set defined in \cref{good-set-definition}.
				 We assume that $(u,\nabla)$ is a critical pair for $E_\epsilon$, defined on $B_1^2\times B_1^{n-2}$,
				 with $\epsilon\le\epsilon_0$ and
				 $$E_\epsilon(u,\nabla)\le|B_1^{n-2}|(2\pi+\tau_0)$$
				 (as well as \cref{u-le-1}--\cref{modica-bounds}).
				 Under this assumption, we have
				 $$\int_{(B_{3/4}^2\setminus B_{1/2}^2)\times\{z\}}e_\epsilon(u,\nabla)\le e^{-K(n)/\epsilon}$$
				 for $z\in B_{3/4}^{n-2}$, since this part of the slice is far from the vorticity set. Recall that the barycenter
				 $$h(z)=\Phi_{\chi(x_1,x_2)}(z)$$
				 was defined using a cut-off function $\chi$ supported in $B_{3/4}^2$, with $\chi=1$ on $B_{1/2}^2$
				 (the notation in the subscript means $\chi\cdot(x_1,x_2)$).
				 
                 \begin{lemma}[Barycenter of a good slice]\label{zero-set-distance-barycenter-slice}
                    For $\epsilon_0,\tau_0,\eta_0>0$ small enough, if $\eta\leq\eta_0$ and $z\in \G^\eta$, then we have the following estimate (for a possibly different $K=K(n)$):
                    \begin{align*}
                        |h(z) - h_0(z)| \leq C(n)\epsilon|\log(\E_2)_z|(\E_2)_z^{1/2}+e^{-K/\epsilon},
                    \end{align*}
                    where $h_0$ is the map from \cref{Lipschitz-approximation-of-zero-loci-lemma} giving the zero set on good slices.
                 \end{lemma}
                    In other words, the barycenter of the good slice is close to the actual zero of $u$ here
                    (unique in $B_{1/2}^2\times\{z\}$).
                    \begin{proof}
                    	Recall that, by definition, we have
                        \begin{align*}
                            h(z)=\Phi_{\chi(x_1,x_2)}(z) = \frac{1}{2\pi}\int_{B^2_{1/2}\times\{z\}} (x_1,x_2)J(u,\nabla)(e_1,e_2)
                            +Ce^{-K/\epsilon}.
                        \end{align*}
                        Since the integral of the Jacobian on $B_{1/2}^2\times\{z\}$
                        is $2\pi+O(e^{-K/\epsilon})$ (see, e.g., the proof of \cite[Lemma 6.11]{Pigati-1}),
                        we get
                        $$h(z)-h_0(z)=\frac{1}{2\pi}\int_{B^2_{1/2}\times\{z\}} [(x_1,x_2)-h_0(z)]J(u,\nabla)(e_1,e_2)
                                                    +Ce^{-K/\epsilon}.$$
                        On the other hand, using the notation from \cref{pull-back-proposition},
                        we have
                        $$\left|\frac{1}{2\pi}\int_{B_{1/2}^2\times\{z\}}[(x_1,x_2)-h_0(z)]J(u_{h_0},\nabla_{h_0})(e_1,e_2)\right|
                                                    \le Ce^{-K/\epsilon},$$
                        by symmetry of the standard planar solution.
                        
                        Moreover, $u(\cdot,z)$ vanishes linearly at $h_0(z)$, as observed in \cref{cone}.
                        We can then apply a rescaling of \cite[Theorem C.1]{Halavati-stability}, which gives
                        \begin{align}\label{halast}
                        \int_{B^2_{1/2}\times\{z\}} |J(u,\nabla)(e_1,e_2)-J(u_{h_0},\nabla_{h_0})(e_1,e_2)| \leq C\sqrt{(\E_2)_z} + Ce^{-{K}/{\epsilon}}.
                        \end{align}
                        Selecting a radius $C(n)\epsilon\le r\le\frac{1}{4}$, we have
                        $$e_\epsilon(u,\nabla)(y,z)\le C(n)\epsilon^{-2}e^{-K|y-h_0(z)|/\epsilon}\quad\text{on }[B_{1/2}^2\setminus B_r^2(h_0(z))]\times\{z\}$$
                        for a possibly different $K$,
                        since as observed in \cref{cone} the distance from the vorticity set $Z$ is comparable to the distance from
                        $$Z\cap (B_{3/4}^2\times\{z\})\subseteq B_{C(n)\epsilon}^2(h_0(z))\times\{z\},$$
                        on good slices.
                        Hence,
                        \begin{align*}
                        &\int_{B^2_{1/2}\times\{z\}}|(x_1,x_2)-h_0(z)||J(u,\nabla)(e_1,e_2)-J(u_{h_0},\nabla_{h_0})(e_1,e_2)|\\ &\leq r\int_{B_r^2(h_0(z))\times\{z\}}|J(u,\nabla)(e_1,e_2)-J(u_{h_0},\nabla_{h_0})(e_1,e_2)|\\
                        &\quad+C\epsilon^{-2}\int_{B_{1/2}^2\setminus B_r^2(h_0(z))}|y-h_0(z)|e^{-K|y-h_0(z)|/\epsilon}\,dy\\
                        &\le Cr\sqrt{\E_z} + C\epsilon e^{-{Kr}/{\epsilon}}
                        \end{align*}
                        for a possibly different $K$.
                        Taking $r:=M\epsilon|\log\E_z|$ for big enough $M$, we get
                        $$r\sqrt{\E_z} + \epsilon e^{-{Kr}/{\epsilon}}\le M\epsilon\sqrt{\E_z}|\log\E_z|+\epsilon\sqrt{\E_z}
                        \le C(n)\epsilon\sqrt{\E_z}|\log\E_z|$$
                        (recall that $\E_z\le\mz$, by definition of good set), unless $r<C(n)\epsilon$ or $r>\frac14$.
                        The situation $r<C(n)\epsilon$ cannot happen, once $M$ is taken large enough, while in the last case we obtain
                        $\E_z\le e^{-K'/\epsilon}$ and thus we are done again, by taking $r:=\frac14$ above.
                    \end{proof}
                
                Next we show that on a good slice the variance is close to $\epsilon^2 v_0$, where $v_0$ is the variance of the standard degree-one planar solution.
                    \begin{lemma}[Variance of a good slice]\label{variance-of-good-slices}
                        For any $\sigma\in(\epsilon,1)$ such that $|h_0(z)|\le\sigma$, we have the following estimate on good slices, for any $c\in\R^2$ with $|c|\le\sigma$:
                        \begin{align*}
                        &\left|\frac{1}{2\pi}\int_{B_{1/2}^2\times\{z\}}[(x_1-c_1)^2+(x_2-c_2)^2]e_\epsilon(u,\nabla)
                        -\epsilon^2v_0\right|\\
                        &\le C(n)\epsilon^2 |\log(\E_2)_z|^2\sqrt{(\E_2)_z} + C(n)\sigma^2(\E_1)_z + C(n)|h(z)-c|^2\sqrt{(\E_2)_z} + C(n)e^{-{K\sigma}/{\epsilon}},
                        \end{align*}
                        for a possibly different $K=K(n)$.
                        \begin{proof}
                        	First of all, since the integrand in the definition of excess $\E=\E_1+\E_2$ upper bounds $e_\epsilon(u,\nabla)-J(u,\nabla)$, we can replace $e_\epsilon(u,\nabla)$ with $J(u,\nabla)$, up
                        	to an error bounded as follows: for $\E_1$, we bound separately the contribution of $B_{2\sigma}^2$ and the complement (where we use exponential decay) obtaining the error $$9\sigma^2(\E_1)_z+C(n)e^{-K\sigma/\epsilon};$$
                            as for $\E_2$, we argue as in the previous proof, obtaining the error
                            $$C(n)|h_0(z)-c|^2(\E_2)_z+C(n)\epsilon^2|\log(\E_2)_z|^2(\E_2)_z+e^{-K/\epsilon},$$
                            where the first term comes from replacing $c$ with the actual location $h_0(z)$ of the zero.
                            Moreover, by definition of $v_0$, we have
                            \begin{align*}
                                \frac{1}{2\pi}\int_{B^2_{1/2}\times\{z\}} |(x_1,x_2)-h_0(z)|^2 J(u_{h_0},\nabla_{h_0})(e_1,e_2) = \epsilon^2 v_0 + O(e^{-{K}/{\epsilon}});
                            \end{align*}
                            since $|(x_1,x_2)-h_0(z)|^2-|(x_1,x_2)-c|^2=2\ang{(x_1,x_2)-h_0(z),c-h_0(z)}-|c-h_0(z)|^2$
                            and
                            $$\int_{B_{1/2}^2\times\{z\}}[(x_1,x_2)-h_0(z)]J(u_{h_0},\nabla_{h_0})=O(e^{-K/\epsilon}),$$
                            we obtain
                            $$\frac{1}{2\pi}\int_{B^2_{1/2}\times\{z\}} |(x_1,x_2)-c|^2 J(u_{h_0},\nabla_{h_0})(e_1,e_2) = |h_0(z)-c|^2+\epsilon^2 v_0 + O(e^{-{K}/{\epsilon}}).$$
                            As in the previous proof, we can replace $J(u_0,\nabla_0)$ with $J(u,\nabla)$ here, up to an error of the form $C(n)(|h_0(z)-c|^2 + \epsilon^2|\log\E_z|^2)\sqrt{(\E_2)_z}$ (using \cref{Jacobian-energy-stability-estimate} from \cref{Jacobian-energy-stability-prop}).
                            Finally, we can bound
                            $$|h_0(z)-c|^2\le 2|h_0(z)-h(z)|^2+2|h(z)-c|^2 \le 2|h(z)-c|^2+C(n)\epsilon^2|\log(\E_2)_z|^2(\E_2)_z$$
                            using the previous proposition, and the claim follows.
                        \end{proof}
                    \end{lemma}
                    
                    \begin{rmk}\label{concave}
                    Since $t\mapsto t|\log t|^2$ is concave for $t>0$ small enough,
                    we have
                    $$\int_S\E_z|\log\E_z|^2\le \left(\int_S\E_z\right)\left|\log\left(\fint_S\E_z\right)\right|^2
                    \le C\left(\int_S\E_z\right)\left|\log\left(\fint_S\E_z\right)\right|^2$$
                    for sets $S\subseteq\mathcal{G}^\eta$ of measure comparable with $1$.
                    \end{rmk}
                    
                    \section{Poincaré--Gaffney-type inequalities}\label{poincare-section}
                    In the construction of the interpolation gauge in \cref{Interpolation-gauge-proposition} we make frequent use of Poincaré-type inequalities for functions and differential forms. These inequalities are well known; we present the special cases used in this paper for the convenience of the reader.
                    
                    The following lemma is a consequence of results first appeared in the original paper of Gaffney \cite{Gaffney} (see also \cite{nonlinear-hodge} for a systematic treatment on manifolds with boundary), but for our application we need it to hold uniformly for cylinders of the form $B^2_1\times B^{n-2}_r$ of arbitrarily small width $r>0$.
                    \begin{lemma}[Poincaré--Gaffney-type inequality for thin cylinders]\label{gaffney-poincare-thin-cylinder-lemma}
                        Given a $1$-form $\alpha \in \Omega^{1}(\bar B^2_1\times \bar B_r^{n-2})$ with $r\leq 1$ and the Neumann boundary condition $\alpha(\nu) = 0$ on $\de(B^2_1\times B^{n-2}_r)$, the following inequality holds:
                        \begin{align*}
                            \int_{B^2_1\times B_r^{n-2}} |\alpha|^2 \leq C(n)\int_{B^2_1\times B^{n-2}_r} [|d\alpha|^2 + |d^*\alpha|^2].
                        \end{align*}
                        \begin{proof}
                            Since $B^2_1\times B^{n-2}_r$ is a convex domain and $\iota_\nu\alpha=0$ at its boundary, we can apply \cite[Remark 9]{CSATO2018461} to see that
                            \begin{align*}
                                \int_{B^2_1\times B^{n-2}_r} |\nabla \alpha|^2 \leq \int_{B^2_1\times B^{n-2}_r }[|d\alpha|^2 + |d^*\alpha|^2].
                            \end{align*}
                            Now we rescale the domain with the map $\phi:B^2_1\times B^{n-2}_1 \rightarrow B^2_1\times B_r^{n-2}$ given by $$\phi(x_1,\dots,x_n) := (x_1,x_2,rx_3,\dots,rx_n),$$ and define $\tilde{\alpha}(x) := \alpha(\phi(x))$ (notice that this is different from the pullback $\phi^*(\alpha)$). Then we claim that there exists a constant $C(n)>0$ such that
                            \begin{align}\label{Gaffney-appendix-1}
                                \int_{B^2_1\times B^{n-2}_1}|\tilde\alpha|^2 \leq C(n)\int_{B^2_1\times B^{n-2}_1}|\nabla\tilde\alpha|^2\,.
                            \end{align}
                            We prove this by compactness and contradiction. By homogeneity, suppose there exists a sequence $\tilde\alpha_k$ with $\iota_\nu\tilde\alpha_k = 0$ on $\de(B_1^2\times B^{n-2}_1)$ and
                            \begin{align*}
                                \int_{B^2_1\times B^{n-2}_1}|\tilde\alpha_k|^2=1,\quad\lim_{k\rightarrow \infty}\int_{B^2_1\times B^{n-2}_1}|\nabla\tilde\alpha_k|^2 = 0.
                            \end{align*}
                            Note that by the display above we have the bound $\|\tilde\alpha_k\|_{W^{1,2}(B^2_1\times B_1^{n-2})} \leq 2$ for all large $k\geq0$.
                            Up to extracting a subsequence, we can assume that $\tilde\alpha_{k}$ converges weakly
                            to $\tilde\alpha_\infty$ in $W^{1,2}$. By Rellich--Kondrachov, the convergence is strong in $L^2$. Thus,
                            $$\int_{B^2_1\times B^{n-2}_1}|\tilde\alpha_\infty|^2=1,\quad \nabla\tilde\alpha_\infty=0.$$
                            Hence, $\tilde\alpha_\infty=v$ is a constant covector. The boundary condition passes to the limit, giving that $v(\nu)=0$ on $\de(B^2_1\times B^{n-2}_1)$; since the normal vectors to the boundary of this domain span all of $\R^n$, we get that $v=0$, a contradiction establishing \cref{Gaffney-appendix-1}. Then we compute that
                            \begin{align*}
                                \int_{B^2_1\times B^{n-2}_r} |\alpha|^2 &= r^{n-2}\int_{B^2_1\times B^{n-2}_1} |\tilde\alpha|^2 \\ &\leq C(n)r^{n-2} \int_{B^2_1\times B^{n-2}_1} |\nabla\tilde\alpha|^2 \\
                                & \leq C(n)\int_{B^2_1\times B^{n-2}_r}|\nabla \alpha|^2\\
                                &\leq C(n)\int_{B^2_1\times B^{n-2}_r} [|d\alpha|^2 + |d^*\alpha|^2],
                            \end{align*}
                            as desired.
                        \end{proof}
                    \end{lemma}
                    The next lemma is a weighted Poincaré estimate for functions in two dimensions.
                    \begin{lemma}\label{CKN-poincare-lemma-appendix}
                        There exists a constant $C>0$ such that for any compactly supported function $f\in C^1_c(B^2_R)$ the following weighted Poincare type estimate holds:
                        \begin{align*}
                            \int_{B^2_R} |x|^2|f|^2(x) \leq C R^{3/2}\int_{B^2_R} |x|^{5/2}|df|^2(x).
                        \end{align*}
                        \begin{proof}
                            By scaling the domain, we can assume that $R=1$. Then by \cite[eq.\ (1.4)]{CKN} (for the choice of constants $\alpha:=5/4$, $a:=1$, $p=q=r:=2$, $\gamma,\sigma:=1/4$) we can see that
                            \begin{align*}
                                \int_{B^2_1}|x|^2|f|^2(x) \leq \int_{B^2_1}|x|^{1/2}|f|^2(x) \leq C\int_{B^2_1} |x|^{5/2}|df|^2(x).
                            \end{align*}
                            This is indeed the desired conclusion.
                        \end{proof}
                    \end{lemma}
                    \begin{lemma}[Poincaré inequality on a thin annulus]\label{Poincare-on-thin-annulus-appendix}
                        Given $a>b\ge c\ge 0$, there exists a constant $C(n,a,b,c)>0$ with the following property. Let $f$ be a function in $W^{1,2}((B^2_a\backslash B^2_c)\times \Omega)$, where $\Omega\subseteq\R^{n-2}$
                        is a convex bounded domain, such that
                        \begin{align*}
                            \int_{(B^2_{a}\backslash B^2_b)\times \Omega} f = 0.
                        \end{align*}
                        Then the following Poincaré inequality holds:
                        \begin{align*}
                            \int_{(B^2_{a}\backslash B^2_c)\times \Omega} |f|^2 \leq C(n,a,b,c)\operatorname{diam}(\Omega)^2\int_{(B^2_{a}\backslash B^2_c)\times \Omega} |df|^2.
                        \end{align*}
                        \begin{proof}
                            First we apply the standard Poincaré inequality on each two dimensional slice $(B^2_a\backslash B^2_c)\times\{z\}$ for any $z\in \Omega$:
                            \begin{align*}
                                \int_{\Omega}\left[\int_{(B^2_a\backslash B^2_c)\times \{z\}}|f|^2\right]\,dz
                                &\leq C(a,b,c)\int_{(B^2_a\backslash B^2_c)\times\Omega}|df|^2 + \int_{\Omega}\left|\int_{(B^2_c\backslash B^2_b)\times\{z\}}f\right|^2\,dz.
                            \end{align*}
                            Notice that the function $g(z):=\int_{(B^2_a\backslash B^2_b)\times\{z\}}f$ has zero average on $\Omega$. Hence we can apply the Poincaré inequality on $\Omega$ to see that
                            $$\int_{\Omega}|g|^2\le C(n)\operatorname{diam}(\Omega)^2\int_{\Omega}|dg|^2.$$
                            Indeed, it is well-known that the Poincaré inequality on a convex domain holds with a constant depending only on its diameter and $n$.
                            This yields the desired conclusion.
                        \end{proof}
                    \end{lemma}
\begin{rmk}\label{rmk:finale}
 The same conclusion holds if we assume that 
   \begin{equation*}
                            \int_{(B^2_{a}\backslash B^2_b)\times \Omega'} f = 0
                        \end{equation*}
                        for some \(\Omega'\) with \(|\Omega'|\ge \alpha |\Omega|\) (the constant depending also on \(\alpha\)).
                        \end{rmk}
            \printbibliography
        \end{document}